\documentclass[
11pt,  reqno]{amsart}
\pdfoutput=1

\usepackage[margin=1.2in,marginparwidth=1.5cm, marginparsep=0.5cm]{geometry}

\usepackage{booktabs} 
\usepackage{microtype}
\usepackage{amssymb}
\usepackage{mathrsfs}

\usepackage{color}
\usepackage[implicit=true]{hyperref}

\usepackage{xsavebox}
\usepackage{comment}

\usepackage{cases}

\allowdisplaybreaks[2]

\usepackage[all]{xy}

\definecolor{gr}{rgb}   {0.,   0.69,   0.23 }
\definecolor{bl}{rgb}   {0.,   0.5,   1. }
\definecolor{mg}{rgb}   {0.85,  0.,    0.85}
\definecolor{yl}{rgb}   {0.8,  0.7,   0.}
\definecolor{or}{rgb}  {0.7,0.2,0.2}

\usepackage{tikz} 
\usepackage{marginnote}
\usepackage{scalerel} 

\usetikzlibrary{shapes.misc}
\usetikzlibrary{shapes.symbols}
\usetikzlibrary{decorations}
\usetikzlibrary{decorations.markings}

\tikzset{
	dot/.style={circle,fill=black,draw=black,inner sep=0pt,minimum size=0.6mm},
	>=stealth,
	}

\tikzset{
	dot2/.style={circle,fill=black,draw=black,inner sep=0pt,minimum size=0.2mm},
	>=stealth,
	}

\tikzset{
	ddot/.style={circle,fill=white,draw=black,inner sep=0pt,minimum size=0.8mm},
	>=stealth,
	}

\tikzset{decision/.style={ 
        draw,
        diamond,
        aspect=1.5
    }}

\tikzset{dia2/.style
={diamond,fill=white,draw=black,inner sep=0pt,minimum size=1mm},
	>=stealth,
	}

\tikzset{dia/.style
={star,fill=black,draw=black,inner sep=0pt,minimum size=1mm},
	>=stealth,
	}

\tikzset{dia/.style
={diamond,fill=black,draw=black,inner sep=0pt,minimum size=1.3mm},
	>=stealth,
	}

\makeatletter
\def\DeclareSymbol#1#2#3{\xsavebox{#1}{\tikz[baseline=#2,scale=0.15]{#3}}}
\def\<#1>{\xusebox{#1}}
\makeatother

\newcommand{\pe}{\mathbin{\scaleobj{0.7}{\tikz \draw (0,0) node[shape=circle,draw,inner sep=0pt,minimum size=8.5pt] {\scriptsize  $=$};}}}

\newcommand{\pez}{\mathbin{\scaleobj{0.7}{\tikz \draw (0,0) node[shape=circle,draw,
fill=white, 
inner sep=0pt,minimum size=8.5pt]{} ;}}}

\usetikzlibrary{decorations.pathmorphing, patterns,shapes}

\DeclareSymbol{dot}{-2.7}{\draw (0,0) node[dot]{};}

\DeclareSymbol{1}{0}{\draw[white] (-.4,0) -- (.4,0); \draw (0,0) -- (0,1.5) node[dot] {};}
\DeclareSymbol{2}{0}{\draw (-0.5,1.2) node[dot] {} -- (0,0) -- (0.5,1.2) node[dot] {};}
\DeclareSymbol{3}{0}{\draw (0,0) -- (0,1.2) node[dot] {}; \draw (-.7,1) node[dot] {} -- (0,0) -- (.7,1) node[dot] {};}
\DeclareSymbol{30}{-3}{\draw (0,0) -- (0,-1); \draw (0,0) -- (0,1.2) node[dot] {}; \draw (-.7,1) node[dot] {} -- (0,0) -- (.7,1) node[dot] {};}
\DeclareSymbol{31}{-3}{\draw (0,0) -- (0,-1) -- (1,0) node[dot] {}; \draw (0,0) -- (0,1.2) node[dot] {}; \draw (-.7,1) node[dot] {} -- (0,0) -- (.7,1) node[dot] {};}
\DeclareSymbol{32}{-3}{\draw (0,0) -- (0,-1) -- (1,0) node[dot] {}; \draw (0,0) -- (0,-1) -- (-1,0) node[dot] {}; \draw (0,0) -- (0,1.2) node[dot] {}; \draw (-.7,1) node[dot] {} -- (0,0) -- (.7,1) node[dot] {};}

\DeclareSymbol{320}{-3}{
\draw (0,0.4) -- (0,-1.2);
\draw (0,0.6) -- (0,-0.4) -- (1,0.6) node[dot] {}; 
\draw (0,0.6) -- (0,-0.4) -- (-1,0.6) node[dot] {}; 
\draw (0,0.6) -- (0,1.8) node[dot] {};
 \draw (-.7, 1.6) node[dot] {} -- (0,0.6) -- (.7,1.6) node[dot] {};}

\DeclareSymbol{20}{-3}{\draw (0,0) -- (0,-1);\draw (-.7,1) node[dot] {} -- (0,0) -- (.7,1) node[dot] {};}
\DeclareSymbol{22}{-3}{\draw (0,0.3) -- (0,-1) -- (1,0) node[dot] {}; \draw (0,0.3) -- (0,-1) -- (-1,0) node[dot] {};\draw (-.7,1) node[dot] {} -- (0,0.3) -- (.7,1) node[dot] {};}
\DeclareSymbol{31p}{-3}{\draw (0,0) -- (0,-1) -- (1,0) node[dot] {}; \draw (0,0) -- (0,1.2) node[dot] {}; \draw (-.7,1) node[dot] {} -- (0,0) -- (.7,1) node[dot] {}; 
\draw (0,-1) node{\scaleobj{0.5}{\pez}}; 
\draw (0,-1) node{\scaleobj{0.5}{\pe}}; }
\DeclareSymbol{32p}{-3}{\draw (0,0) -- (0,-1) -- (1,0) node[dot] {}; \draw (0,0) -- (0,-1) -- (-1,0) node[dot] {}; \draw (0,0) -- (0,1.2) node[dot] {}; \draw (-.7,1) node[dot] {} -- (0,0) -- (.7,1) node[dot] {}; 
\draw (0,-1) node{\scaleobj{0.5}{\pez}};
\draw (0,-1) node{\scaleobj{0.5}{\pe}};}
\DeclareSymbol{22p}{-3}{\draw (0,0.3) -- (0,-1) -- (1,0) node[dot] {}; \draw (0,0.3) -- (0,-1) -- (-1,0) node[dot] {};\draw (-.7,1) node[dot] {} -- (0,0.3) -- (.7,1) node[dot] {}; 
\draw (0,-1) node{\scaleobj{0.5}{\pez}};
\draw (0,-1) node{\scaleobj{0.5}{\pe}};}
\DeclareSymbol{21p}{-3}{\draw (0,0.3) -- (0,-1) -- (1,0) node[dot] {}; 
\draw (-.7,1) node[dot] {} -- (0,0.3) -- (.7,1) node[dot] {}; 
\draw (0,-1) node{\scaleobj{0.5}{\pez}};
\draw (0,-1) node{\scaleobj{0.5}{\pe}};}

\DeclareSymbol{70}{-3}{
\draw (0,-0.4) -- (0,-1.2);
\draw (0,0.8) node[dot] {} -- (0,-0.4) -- (0.8,0.6); 
\draw (0.8, .6) -- (1.4, 1.35)  node[dot] {}; 
\draw (0.8, .6) -- (1.8, .8)  node[dot] {}; 
\draw (0.8, .6) -- (0.85, 1.6)  node[dot] {}; 
\draw (0,.6) -- (0,-0.4) -- (-0.8,0.6) ;
\draw (-0.8, 0.6) -- (-1.4, 1.35)  node[dot] {}; 
\draw (-0.8, .6) -- (-1.8, .8)  node[dot] {}; 
\draw (-0.8, .6) -- (-0.85, 1.6)  node[dot] {}; 
}

\DeclareSymbol{7}{-3}{
\draw (0,0.2) node[dot] {} -- (0,-1) -- (0.8,0); 
\draw (0.8, 0) -- (1.4, 0.75)  node[dot] {}; 
\draw (0.8, 0) -- (1.8, 0.2)  node[dot] {}; 
\draw (0.8, 0) -- (0.85, 1)  node[dot] {}; 
\draw (0,0) -- (0,-1) -- (-0.8,0) ;
\draw (-0.8, 0) -- (-1.4, 0.75)  node[dot] {}; 
\draw (-0.8, 0) -- (-1.8, 0.2)  node[dot] {}; 
\draw (-0.8, 0) -- (-0.85, 1)  node[dot] {}; 
}

\DeclareSymbol{11p}{-3}{\draw (0,0) node[dot2] {} -- (0,-1) -- (1,0) node[dot] {}; 
\draw (0,0) -- (0,1.2) node[dot] {};  
\draw (0,-1) node{\scaleobj{0.5}{\pez}}; 
\draw (0,-1) node{\scaleobj{0.5}{\pe}}; }
\DeclareSymbol{100}{-3}
{\draw[white] (-.4,0) -- (.4,0); 
\draw (0,0) node[dot2] {} -- (0,-1)  ; 
\draw (0,0) -- (0,1.2) node[dot] {};}  
\usetikzlibrary{mindmap,backgrounds,trees,arrows,decorations.pathmorphing,decorations.pathreplacing,positioning,shapes.geometric,shapes.misc,decorations.markings,decorations.fractals,calc,patterns}

\tikzset{>=stealth',
         cvertex/.style={circle,draw=black,inner sep=1pt,outer sep=3pt},
         vertex/.style={circle,fill=black,inner sep=1pt,outer sep=3pt},
         star/.style={circle,fill=yellow,inner sep=0.75pt,outer sep=0.75pt},
         tvertex/.style={inner sep=1pt,font=\scriptsize},
         gap/.style={inner sep=0.5pt,fill=white}}
\tikzstyle{mybox} = [draw=black, fill=blue!10, very thick,
    rectangle, rounded corners, inner sep=10pt, inner ysep=20pt]
\tikzstyle{boxtitle} =[fill=blue!50, text=white,rectangle,rounded corners]

\tikzstyle{decision} = [diamond, draw, fill=blue!20,
    text width=4.5em, text badly centered, node distance=2.5cm, inner sep=0pt]
\tikzstyle{block} = [rectangle, draw, fill=blue!20,
    text width=2.7cm, text centered, rounded corners, minimum height=3em]

\tikzstyle{line} = [draw, very thick, color=black!50, -latex']

\tikzstyle{cloud} = [draw, ellipse,fill=red!40, 
node distance=2.5cm,
    minimum height=1em]

\tikzstyle{cloud2} = [draw, ellipse,fill=red!30, text=white,text width=10em, node distance=2.5cm, text centered, minimum height=4em]

\tikzstyle{cloud3} = [draw, ellipse, fill=cyan!30, 
node distance=2.5cm,
    minimum height=3em, 
    text width=1.7cm]

\tikzstyle{cloud4} = [draw, ellipse,fill=orange!70, node distance=2.5cm,
   text width=2.1cm,  
           minimum height=3em]

\tikzstyle{cloud5} = [draw, ellipse,fill=red!20, node distance=2.5cm,
    text centered, 
    text width=3.0cm, 
    minimum height=3em]

\tikzstyle{cloud6} = [draw, ellipse,fill=red!20, node distance=2.5cm,
    text width=1.5cm, 
    text centered, 
    minimum height=3em]

\tikzset{
    position/.style args={#1:#2 from #3}{
        at=(#3.#1), anchor=#1+180, shift=(#1:#2)
    }
}

\newtheorem{theorem}{Theorem} [section]

\newtheorem{lemma}[theorem]{Lemma}
\newtheorem{proposition}[theorem]{Proposition}
\newtheorem{remark}[theorem]{Remark}

\newtheorem{definition}[theorem]{Definition}

\DeclareMathOperator*{\supp}{supp}
\DeclareMathOperator{\med}{med}

\renewcommand{\1}{\hspace{0.2mm}\text{I}\hspace{0.2mm}}
\newcommand{\II}{\text{I \hspace{-2.8mm} I} }

\newcommand{\noi}{\noindent}
\newcommand{\Z}{\mathbb{Z}}
\newcommand{\R}{\mathbb{R}}

\newcommand{\T}{\mathbb{T}}
\newcommand{\bul}{\bullet}

\newcommand{\NN}{\mathcal{N}}
\newcommand{\Rr}{\mathcal{R}}
\newcommand{\Prob}{\mathbb{P}}
\newcommand{\MM}{\mathfrak{M}}

\newcommand{\TT}{\mathfrak{T}}
\newcommand{\ze}{\zeta}

\let\P= \undefined
\newcommand{\P}{\mathbf{P}}

\newcommand{\E}{\mathbb{E}}

\renewcommand{\L}{\mathcal{L}}

\newcommand{\F}{\mathcal{F}}

\newcommand{\al}{\alpha}
\newcommand{\be}{\beta}
\newcommand{\dl}{\delta}

\newcommand{\nb}{\nabla}

\newcommand{\Dl}{\Delta}
\newcommand{\eps}{\varepsilon}
\newcommand{\kk}{\kappa}
\newcommand{\g}{\gamma}

\newcommand{\ld}{\lambda}
\newcommand{\Ld}{\Lambda}
\newcommand{\s}{\sigma}

\newcommand{\ft}{\widehat}

\newcommand{\wt}{\widetilde}
\newcommand{\cj}{\overline}

\newcommand{\dt}{\partial_t}

\newcommand{\ta}{\theta}

\renewcommand{\l}{\ell}
\renewcommand{\o}{\omega}
\renewcommand{\O}{\Omega}

\newcommand{\les}{\lesssim}
\newcommand{\ges}{\gtrsim}

\newcommand{\jb}[1]
{\langle #1 \rangle}

\newcommand{\tw}[1]
{\widetilde{#1}}

\newcommand{\mf}[1]
{\mathfrak{#1}}

\newcommand{\ind}{\mathbf 1}

\renewcommand{\S}{\mathcal{S}}

\newcommand{\too}{\longrightarrow}

\newcommand{\N}{\mathbb{N}}

\renewcommand{\H}{\mathcal{H}}

\newtheorem*{ackno}{Acknowledgements}

\newcommand{\I}{\mathcal{I}}

\newcommand{\B}{\mathcal{B}}

\newcommand{\FL}{\mathcal{F} L}

\numberwithin{equation}{section}
\numberwithin{theorem}{section}

\DeclareMathOperator{\Sym}{\mathtt{Sym}}

\begin{document}

\title[On the convergence of the 2-$d$ singular SCGL]
{On the inviscid limit of the singular stochastic complex Ginzburg-Landau equation at statistical equilibrium}

\author[Y.~Zine]
{Younes Zine}


\address{
Younes Zine, School of Mathematics\\
The University of Edinburgh\\
and The Maxwell Institute for the Mathematical Sciences\\
James Clerk Maxwell Building\\
The King's Buildings\\
Peter Guthrie Tait Road\\
Edinburgh\\ 
EH9 3FD\\
 United Kingdom}

\email{y.p.zine@sms.ed.ac.uk}



\begin{abstract} We study the limiting behavior of the two-dimensional singular stochastic stochastic cubic nonlinear complex Ginzburg-Landau with Gibbs measure initial data.  We show that in the appropriate small viscosity and small noise regimes, the limiting dynamics is given by the deterministic cubic nonlinear Schr\"odinger equation at Gibbs equilibrium. In order to obtain this convergence, our approach combines both heat and Schr\"odinger analysis, within the framework of the Fourier restriction norm method of Bourgain (1993).
\end{abstract}


\maketitle
%


\tableofcontents

\baselineskip = 14pt

\section{Introduction}\label{SEC1}
\subsection{$\Phi^4_2$-measure and dynamics}\label{SUBSEC:dyn}
The complex-valued $\Phi ^4_ 2$-measure on $\T^2 = (\R / 2\pi \Z)^2$ is the Gibbs measure formally given by

\noi
\begin{align}
d \rho (u) = Z^{-1}  \exp \Big( - \frac14 \int_{\T^2} |u|^4 dx \Big) d \mu(u).
\label{Phi1}
\end{align}

\noi
Here, $Z$ is a normalization factor and $\mu$ is the massive Gaussian free field defined as the law of the random variable $\phi$ given by

\noi
\begin{align}
\phi(x) := \sum_{n \in \Z^2} \frac{g_n}{\jb n} e^{i n \cdot x},
\label{GFF}
\end{align}

\noi
where $\{g_n\}_{n \in \Z^2}$ is a family of i.i.d. standard complex Gaussian variables. The Gibbs measure \eqref{Phi1} is an important object of study in Euclidean quantum field theory and can be constructed rigorously; see Proposition \ref{PROP:mes} below. The energy functional (Hamiltonian) $H$ associated to $\rho$ is given by

\noi
\begin{align}
H (u) = \frac12 \int_{\T^2} |u|^2 dx + \frac14 \int_{\T^2} |u|^4 dx.
\label{H1}
\end{align}

\noi
In the complex-valued setting, the Hamiltonian \eqref{H1} induces different dynamics on $\T^2$ which are all expected to leave the measure $\rho$ invariant:

\begin{itemize}

\item[1.] the hyperbolic $\Phi^4_2$-model, which is given by the stochastic damped cubic nonlinear wave equation (SdNLW$_{\eps,\g}$):\footnote{Strictly speaking, the $\Phi^{4}_2$-measure has to be coupled with (up to some rescaling) the white noise measure on $\T^2$ on the coordinate $\dt u$. However, for convenience, we omit these technicalities in this discussion.}

\noi
\begin{align}
\eps^2 \dt^2 u_{\eps, \g} +  \dt u_{\eps, \g}  = (\g + i)( \Dl - 1 ) u_{\eps, \g} - (\g + i) |u_{\eps, \g}|^{2} u_{\eps, \g} + \sqrt{2 \g} \xi,
\label{SdNLW}
\end{align}

\noi
for $\eps >0$ and $\g >0$.

\vspace{1mm}

\item[2.] the parabolic $\Phi^4_2$-model, which is given by the stochastic complex cubic nonlinear Ginzburg-Landau equation (SCGL$_\g$):

\noi
\begin{align}
\dt u_\g  = (\g + i)( \Dl - 1 ) u_\g - (\g + i) |u_\g|^2 u_\g + \sqrt{2 \g} \xi,
\label{SCGL}
\end{align}

\noi
for $\g >0$.

\vspace{1mm}

\item[3.] the dispersive $\Phi^4_2$-model, which given by the deterministic cubic nonlinear Schr\"odinger equation (NLS):

\noi
\begin{align}
\dt u  =  i( \Dl - 1 ) u - i |u|^2 u ,
\label{NLS}
\end{align}

\end{itemize}

\noi
Here, $\xi(x, t)$ denotes (Gaussian) space-time white noise on $\T^2\times \R$ defined on a probability space $(\O,\Prob)$ and
with the space-time covariance (formally) given by
\[ \E\big[ \xi(x_1, t_1) \xi(x_2, t_2) \big]
= \dl(x_1 - x_2) \dl (t_1 - t_2) \]

In the quantum field theory community, the equations SdNLW$_{\eps,\g}$ \eqref{SdNLW} and SCGL$_\g$ \eqref{SCGL} were introduced in order to study properties of the $\Phi^{4}_2$-measure. This idea of studying the measure \eqref{Phi1} through these equations is known as stochastic quantization and was streamlined by Parisi and Wu \cite{PW} in the context of heat equations; see also \cite{RSS}. On the other hand, the study of NLS \eqref{NLS} with Gibbs measure initial data was originally motivated by statistical mechanics and was initiated (in the one-dimensional case) by Lebowitz, Rose, and Speer \cite{LRS}. 

Let us briefly review the well-posedness issue for \eqref{SdNLW}, \eqref{SCGL}, and \eqref{NLS} in the two-dimensional setting. Regarding the wave and heat equations \eqref{SdNLW} and \eqref{SCGL}, the well-posedness question for both deterministic and Gibbsian initial data was resolved by Gubinelli, Koch, Oh, and Tolomeo \cite{GKO1, GKOT}, and Da Prato and Debussche \cite{DPD} (see also \cite{Trenberth, Matsuda}), respectively. Furthermore, in a seminal work, Bourgain \cite{BO96} established well-posedness for NLS \eqref{NLS} with Gibbs measure initial data. The main difficulty in studying these equations, even locally in time, comes from the roughness of space-time white noise (and/or of the Gibbs measure initial data \eqref{Phi1}). In fact, the solutions to the equations \eqref{SdNLW}, \eqref{SCGL}, and \eqref{NLS} are all expected to be distributions, not functions. For this reason, these dynamical problems are often called singular stochastic partial differential equations.  Moreover, the cubic terms in \eqref{SdNLW}, \eqref{SCGL}, and \eqref{NLS} are ill-defined, and need to be renormalized; see Subsection \ref{SUBSEC:REN} for more details. 

When $\eps = 0$, SdNLW$_{\eps,\g}$ \eqref{SdNLW} formally corresponds to SCGL$_\g$ \eqref{SCGL}; while for $\g = 0$, SCGL$_\g$ \eqref{SCGL} corresponds to NLS \eqref{NLS}. Hence, we expect the solution of SdNLW$_{\eps,\g}$ \eqref{SdNLW} to converge to the solution of SCGL$_\g$ \eqref{SCGL} as $\eps \to 0$, which in turn should converge to the solution of NLS \eqref{NLS} as $\g \to 0$. The first convergence (from \eqref{SdNLW} to \eqref{SCGL}) is referred to as the Smoluchowski-Kramers approximation (or non-relativistic limit) and has been studied extensively in various contexts \cite{CF1, CF2, CF3, CF4, CG, CS1, CS2, CS3, FHI, YZ}. In the singular setting, this convergence was established rigorously by the author \cite{YZ} for both deterministic and Gibbs measure initial data.\footnote{We note that a similar convergence result was obtained independently by Fukuizumi, Hoshino, and Inui \cite{FHI} for Gibbs measure initial data.} The second convergence issue; namely from the solution of SCGL$_\g$ \eqref{SCGL} to that of NLS \eqref{NLS} (as $\g \to 0$) is called the {\it inviscid limit} and has been studied in either the deterministic (i.e. $\xi \equiv 0$) or smooth noise settings  \cite{BJ, MN, OY, Wu, KS, Shirikyan}.

In the present work, we address the question of the inviscid limit in the singular setting for the equations SCGL$_\g$ \eqref{SCGL} and NLS \eqref{NLS} with Gibbsian data. We now state a formal version of our main result, but postpone a rigorous version to Theorem \ref{THM:main} below.

\noi
\begin{theorem}[Global-in-time convergence, formal version]\label{THM:meta}
After a renormalization procedure, the equations $\textup{SCGL}_\g$ \eqref{SCGL}, for any $\g \in (0,1]$, and \textup{NLS} \eqref{NLS} are almost surely globally well-posed with respect to the \textup{(}renormalized\textup{)} $\Phi^4_2$-measure \eqref{Phi1}. Moreover, the associated dynamics preserve the $\Phi^4_2$-measure \eqref{Phi1}. Furthermore, the solution $u_\g$ to $\textup{SCGL}_\g$ converges to the solution $u$ to \textup{NLS} in $C \big( \R_+; H^{-\dl}(\T^2) \big)$, $\dl >0$, as $\g \to 0$, almost surely.
\end{theorem}

\noi
\begin{remark}[Convergence of $\Phi^4_2$-models]\rm 

\noi
Together with the results in \cite{YZ}, Theorem \ref{THM:meta} settles the question of convergence of $\Phi^4_2$-models, as illustrated below.
\noi
{\vspace{-0.3mm}
\begin{equation*} \xymatrix{\underset{\text{\small $\eps^2 \dt^2 u_{\eps, \g} +  \dt u_{\eps, \g}  = (\g + i)( \Dl - 1 ) u_{\eps, \g} - (\g + i) |u_{\eps, \g}|^{2} u_{\eps, \g} + \sqrt{2 \g} \xi$}}{\text{hyperbolic $\Phi^4_2$-model} }  \ar@{~>}[d]_{\txt<3pc>{\scriptsize  $\eps \to 0$}}
\\ 
  \underset{\text{\small $\dt u_\g  = (\g + i)( \Dl - 1 ) u_\g - (\g + i) |u_\g|^{2} u_\g + \sqrt{2 \g} \xi$}}{\text{parabolic $\Phi^4_2$-model} }   \ar@{~>}[d]_{\txt<3pc>{\scriptsize  $\g \to 0$}} 
  \\
   \underset{\text{\small $\dt u  = i( \Dl - 1 ) u - i |u|^{2} u$}}{\text{dispersive $\Phi^4_2$-model} }  }
  \end{equation*}} 
\end{remark}

\subsection{Renormalization and rigorous statement of the main result}\label{SUBSEC:REN} We describe here the renormalization procedure that we carry out to make sense of both the measure $\rho$ \eqref{Phi1} and the associated dynamical problems \eqref{SCGL} and \eqref{NLS}.

We first aim to define the measure $\rho$ \eqref{Phi1} rigorously. It is easy to see that the random variable $\phi$ \eqref{GFF} only belongs to $\phi \in H^{0-}(\T^2) \setminus L^2(\T^2)$\footnote{Hereafter, we use $a-$ 
(and $a+$) to denote $a- \eps$ (and $a+ \eps$, respectively)
for arbitrarily small $\eps > 0$.
If this notation appears in an estimate, 
then  an implicit constant 
is allowed to depend on $\eps> 0$ (and it usually diverges as $\eps \to 0$).}, almost surely. Hence, if $u$ is a typical element in the support of $\mu$, the nonlinear term $|u|^4$ is ill-defined as a product of distributions. We thus need to renormalize the interaction potential in \eqref{Phi1} in order to cure divergences.

Let $\P_{\le N}$ be the spatial frequency projection onto $\{ n \in \Z^2 : \jb{n} \le N \}$. Namely,

\noi
\begin{align}
\P_{\le N} f (x) = \sum_{\jb n \le N} \ft f (n) e^{i n \cdot x}.  
\label{proj1}
\end{align}

\noi
Let $N \in \N$ and define the truncated renormalized interaction potential $R_N$ as follows:

\noi
\begin{align}
R_N(u) = - \frac14 \int_{\T^2} | \P_{\le N} u|^4 dx  - \Big( \int_{\T^2} | \P_{\le N} u|^2 dx \Big)^2.
\label{ren_pot}
\end{align}

\noi
Then, we define the truncated renormalized probability measure $\wt \rho _N$ by

\noi
\begin{align}
\widetilde{\rho}_N := Z_N^{-1} \exp \bigg( - \frac14 \int_{\T^2} | \P_{\le N} u|^4 dx  - \Big( \int_{\T^2} | \P_{\le N} u|^2 dx \Big)^2 \bigg) d \mu(u)
\label{ren_mes}
\end{align}


%


\noi
It turns out that the above definition of the measure $\wt \rho_N$ \eqref{ren_mes} leads to a meaningful object in the sense that $\wt \rho _N$ admits a well-defined limit as $N \to \infty$. Justifying this procedure is the purpose of the next result.

\noi
\begin{proposition}\label{PROP:mes}
Let $R_N (u)$ and $\wt \rho_N$ be as in \eqref{ren_pot} and \eqref{ren_mes}, respectively. Then, the following holds:

\smallskip

\noi
\textup{(i)} The truncated renormalized interaction potentials $\{ R_N (u) \}_{N \in \N}$ form a Cauchy sequence in $L^p(\mu)$ for any finite $p \ge 1$; thus converging to some random variable $R(u) \in L^p(\mu)$.

\smallskip

\noi
\textup{(ii)} Given any finite $p \ge 1$, there exists $C_p>0$ such that

\noi
\begin{align*} 
\sup_{N \in \N} \big\| e^{R_N (u)} \big\|_{L^p (\mu)} \le C_p.
\end{align*} 

\noi
Moreover, we have

\noi
\begin{align} 
\lim_{N \to \infty} e^{R_N (u)} = e^{R (u)} \qquad \text{in $L^p(\mu)$}. 
\label{cv_den}
\end{align} 

\noi
As a consequence, the truncated renormalized Gibbs measure $\wt \rho_N$ converge, in the sense of \eqref{cv_den} to the renormalized Gibbs measure $\rho$ given by

\noi
\begin{align}
d \wt \rho(u) = Z^{-1} e^{R(u)} d \mu.
\label{Rmes}
\end{align}

\noi
Furthermore, the resulting Gibbs measure $\wt \rho$ is equivalent to the Gaussian measure $\mu$.
\end{proposition}

The proof of Proposition \ref{PROP:mes} follows by (a slight variation of) a standard argument by Nelson \cite{Nelson2} (see also \cite{OT}).

 Let us now describe how this renormalization procedure for the Gibbs measure $\rho$ \eqref{Phi1} naturally leads to a renormalization at the level of the equations \eqref{SCGL} and \eqref{NLS}. Fix $N \in \N$ and consider the truncated Hamiltonian $H_N (u)$ associated to $\widetilde{\rho}_N$ given by 

\noi
\begin{align}
\begin{split}
H_N(u) & = \frac12 \int_{\T^2} | \P_{\le N} \nb u|^2 dx + \frac12 \int_{\T^2} | \P_{\le N} u|^2 dx \\
& \qquad \quad + \frac14 \int_{\T^2} |\P_{\le N} u|^4 dx - \Big( \int_{\T^2} | \P_{\le N} u|^2 dx \Big)^2.
\end{split}
\label{H2}
\end{align}

\noi
The dynamical problems (2. and 3. in Subsection \ref{SUBSEC:dyn}) induced by $H_N$ are given by

\noi
\begin{align}
\dt u_{\g, N}  & = (\g + i)( \Dl - 1 ) u_{\g,N} - (\g + i) \P_{\le N} \mathfrak{N} \big( \P_{\le N}u_{\g,N} \big) + \sqrt{2 \g} \xi,\label{rSCGL1} \\
\dt u_N &  =  i( \Dl - 1 ) u_N - i \P_{\le N} \mf{N}\big( \P_{\le N}u \big). \label{rNLS1}
\end{align}

\noi
for $\g \in (0,1]$. Here, $\mf{N}$ denotes the {\it renormalized cubic nonlinearity} defined by

\noi
\begin{align}
\mf N (u)  & = \Big( |u|^2 - 2 \int_{\T^2} |u|^2 dx \Big)u. \label{NNR}
 \end{align}

\noi
We can then further decompose the nonlinearity $\mf N$ into 

\noi
\begin{align}
\mf N (u) = \NN(u) - \mathcal{R}(u),
\label{N2}
\end{align}

\noi
where $\NN$ and $\Rr$ are the nonlinearities associated to the trilinear forms

\noi
\begin{align}
\F_x \big( \NN(u_1, u_2, u_3) \big)(n,t) = \sum_{  \substack{n_1, n_2, n_3 \\ n_2 \neq n_1, n_3 }   } \ft{u_1}(n_1,t) \cj{ \ft{u_2}(n_2,t)  } \ft{u_3}(n_3,t),
\label{N3}
\end{align}

\noi
and

\noi
\begin{align}
\F_x \big( \Rr(u_1, u_2, u_3) \big)(n,t) = \ft{u_1}(n,t) \cj{ \ft{u_2}(n,t)  } \ft{u_3}(n,t),
\label{N4}
\end{align}

\noi
respectively. As $N \to \infty$, the dynamics \eqref{rSCGL} and \eqref{rNLS} formally converge to

\noi
\begin{align}
\dt u_\g  & = (\g + i)( \Dl - 1 ) u_\g - (\g + i) \mathfrak{N}(u_\g) + \sqrt{2 \g} \xi \label{rSCGL}, \\
\dt u &  =  i( \Dl - 1 ) u - i  \mf{N}(u) \label{rNLS},
\end{align}

\noi
respectively. We refer to this renormalization procedure as the {\it PDE renormalization} of equations \eqref{SCGL} and \eqref{NLS}. It was introduced by Bourgain \cite{BO96} in the study of the two-dimensional cubic NLS with Gibbs measure initial data. We point out that the canonical renormalization scheme in Euclidean quantum field theory and singular stochastic partial differential equations is the Wick renormalization. See Remark \ref{RM:Wick} below for a discussion on the implementation of the Wick renormalization in our context. We now state the precise version of our main result.

\noi
\begin{theorem}[Global-in-time convergence, rigorous version]\label{THM:main} The following holds\footnote{Hereafter, every quantity with the index $\g = 0$ will refer to the corresponding quantity for NLS. For instance $u_{\g=0}$ (or simply $u_0$) will denote the solution $u$ to \eqref{rNLS}, or \eqref{rSCGL1} with $\g=0$ will be interpreted as \eqref{rNLS1}.}:

\smallskip

\noi
\textup{(i) (global well-posedness and invariance)} The renormalized equations \eqref{rSCGL}, for any $\g \in (0,1]$, and \eqref{rNLS} are almost surely globally well-posed with respect to the measure $\wt \rho$ \eqref{Rmes}. Moreover, the \textup{(}renormalized\textup{)} Gibbs measure $\wt \rho$ is invariant under the dynamics. 

More precisely, for each $\g \in [0,1]$, there exists a non-trivial stochastic process $u_\g \in~C \big( \R_+ ; H^{0-}(\T^2) \big)$ such that, for any $T>0$, the solution $u_{\g,N}$ \textup{(}resp. $u_{N}$ for $\g = 0$\textup{)} to the truncated dynamics \eqref{rSCGL1} \textup{(}resp. \eqref{rNLS1}\textup{)} with initial data sampled from the truncated Gibbs measure $\wt \rho_N$ \eqref{ren_mes}, converges in probability to $u_\g$ as $N \to \infty$. Moreover, the law of $u_\g(t)$ is given by the renormalized Gibbs measure $\wt \rho$ \eqref{Rmes} for any $t \ge 0$.

\smallskip

\noi
\textup{(ii) (convergence)} Let $u_\g$ and $u$ be the solutions to \eqref{rSCGL} and \eqref{NLS} respectively,  distributed according to the Gibbs measure $\wt \rho$. Then, $u_\g$ converges to $u$ in $C \big( \R_+ ; H^{0-}(\T^2) \big)$ as $\g \to 0$, $\wt \rho \otimes \Prob$-almost surely\footnote{Here, $(\O, \P)$ denotes the probability space underlying the noise $\xi$}.
\end{theorem}

The main novelty in Theorem \ref{THM:main} is the convergence (ii). To the best of the author's knowledge, Theorem \ref{THM:main} (ii) is the first instance of a convergence result of the solution of the Ginzburg-Landau equation to that of the nonlinear Schr\"odinger equation in a singular setting.

Let us highlight the fact that in Theorem \ref{THM:main}, the convergence (ii) is proved along the continuous parameter $\g \to 0$. This is in sharp contrast with the literature on convergence problems where the convergence is often proved along a discrete parameter (i.e. in our context, a discrete sequence of $\g$'s); see for instance \cite{CF1, FHI}. We further discuss this point in Remark \ref{RMK:cont} below.

We also note that there is an additional difficulty in proving the global-in-time almost sure convergence result (ii) in Theorem \ref{THM:main}. Indeed, (ii) implies that there exists a single set of full $\wt \rho \otimes \Prob$-probability $\O_g$ such that almost sure global well-posedness holds for $u_\g$ for {\it all $\g \in [0,1]$ at the same time} on $\O_g$. Let us describe the naive (but natural) attempt to prove the existence of such a set: (a) for each fixed $\g \in [0,1]$, use Bourgain's invariant measure argument \cite{BO94} to construct a set of full $\wt \rho \otimes \Prob$-probability $\O_\g$ such that almost sure global well-posedness holds for $u_\g$ on $\O_\g$; (b) define $\O_g$ as the intersection of the sets $\O_\g$ over $\g \in [0,1]$. However, since the set $\O_g$ constructed in (b) is an uncountable intersection, it is not possible to ensure that it is of full $\wt \rho \otimes \Prob$-probability. We overcome this issue by combining Bourgain's invariant measure argument with pathwise analysis.

From an analytical perspective, the difficulty in proving Theorem \ref{THM:main} comes from the lack of smoothing under the Schr\"odinger flow \eqref{NLS}, whereas the heat flow \eqref{SCGL} essentially gains two derivatives. As such, the well-posedness issue for \eqref{NLS} is a much harder problem to consider than that for \eqref{SCGL} with $\g >0$. Thus, it is natural to choose a space appropriate for the study of NLS \eqref{NLS} as a common space where to compare the solutions of \eqref{SCGL} and \eqref{NLS}. In what follows, we consider the Fourier restriction norm method (introduced by Bourgain \cite{BO93} in the context of NLS) utilizing $X^{s,b}$-spaces (defined in Definition \ref{DEF:X} below) adapted to the Schr\"odinger flow. However, it turns out that these spaces are not well suited for the study of the parabolic equation \eqref{SCGL} for small $0 < \g \ll 1$, which causes issues in our analysis. We elaborate further on this point at the end of Section \ref{SUBSEC:outline}. 

\noi
\begin{remark}\rm \label{RMK:cont}
Here, we describe, at a formal level, how to achieve convergence of the solution of \eqref{rSCGL} to that of \eqref{NLS} along the continuous parameter $\g \to 0$ (i.e. Theorem \ref{THM:main} (ii)) for some time $T>0$. The key idea is to use the perspective introduced in the work \cite{YZ}. Namely, by viewing $\g$ as a ``second time parameter", we can write the dynamical problems \eqref{rSCGL} and \eqref{rNLS} as a single equation in the variables $(\g,t,x)\in [0,1] \times [0,T]\times \T^2$. We call this equation the {\it enhanced dynamical problem}. Let us assume that we can construct a solution $u = u(\g,t,x) := u_\g(t,x)$ to this enhanced equation in an abstract space of the form $C \big( [0,1] \times [0,T] ; X \big)$ (for some space $X$ of spatial functions). Then, the convergence of $u_\g$ to $u_{\g = 0}$ follows from the continuity of $u= u(\g,t,x)$ at $\g = 0$. Hence, in essence, Theorem \ref{THM:main} (ii) follows from the existence of a solution to the enhanced dynamical problem in an appropriate space of functions that are continuous in $\g$.
\end{remark}

\noi
\subsection{Outline of the proof}\label{SUBSEC:outline}

Here, we outline the proof of Theorem \ref{THM:main}. As discussed in the above, the globalization part in Theorem \ref{THM:main} follows from a suitable variation of Bourgain's invariant measure argument. Hence, in this section, we focus on the local theory. Furthermore, in view of the absolute continuity of $\wt \rho$ with respect to the Gaussian free field $\mu$, we consider \eqref{rSCGL} and \eqref{rNLS} with initial data distributed according to $\mu$.

For $\g \in [0,1]$, we define the operator $S_\g$ by

\noi
\begin{align}
S_\g(t,t') = e^{ ( \g |t| + i t') (\Dl - 1) }, \quad (t,t') \in \R^2.
\label{lin1}
\end{align}

\noi
We also set $S_\g(t) =S_\g(t,t)$ for $t \in \R$. Next, we define the following Duhamel operators:

\noi
\begin{align}
\begin{split}
& \I_{\g} (F)(t) =   \ind_{ \R_+ }(t)  \int_{0}^t S_\g(t-t',t-t') F(t') dt' + \ind_{ \R_- }(t)  \int_{0}^t S_\g(t+t',t-t') F(t') dt'    \\
& \I_0 (F)(t) = \int_0^t S_0(t-t') F(t') dt'.
\end{split}
\label{duha1}
\end{align}

\noi
We view $\{ \I_\g \}_{\g \in [0,1]}$ as an element $\I$ of $C \big( [0,1]; \S'(\T^2 \times \R) \big)$ by simply writing 

\noi
\begin{align}
\I (F)( \g) = \I_\g (F_\g).
\label{duha100}
\end{align}

Let $\g \in [0,1]$. We denote by $\Psi_{\g}$ the stochastic convolution

\noi
\begin{align}
\Psi_{\g}= \sqrt{2 \g} \int_{0}^t S_\g(t-t') dW(t'),
\label{conv1}
\end{align}

\noi
that is, the solution of the linear equation

\noi
\begin{align*}
\dt \Psi_\g = i (\Dl - 1)\Psi_\g + \sqrt{2 \g } \xi.
\end{align*}

\noi
In the above, $W$ denotes a cylindrical Wiener process on $L^2(\T^2)$:
\begin{align}
W(t)
: = \sum_{n \in \Z^2 } B_n (t) e^{in \cdot x}
\label{W1}
\end{align}

\noi
and  
$\{ B_n \}_{n \in \Z^2}$ 
is defined by 
$B_n(t) = \jb{\xi, \ind_{[0, t]} \cdot e_n}_{ x, t}$.
Here, $\jb{\cdot, \cdot}_{x, t}$ denotes 
the duality pairing on $\T^2\times \R$.
As a result, 
we see that $\{ B_n \}_{n \in \Z^2}$ is a family of mutually independent complex-valued Brownian motions. By convention, we normalized $B_n$ such that $\text{Var}(B_n(t)) = t$.

We also denote, for $ N \in \N$, by $\Psi_{\g,N}$ the truncated stochastic convolution

\noi
\begin{align}
\Psi_{\g,N} = \P_{\le N} \Psi_\g,
\label{conv2}
\end{align}

\noi
and we denote by $\<1>_\g$ and $\<1>_{\g,N} $ (for $N \in \N$) the stochastic processes defined by

\noi
\begin{align}
\<1>_{\g}(t) &=   S_\g(t) \phi + \Psi_{\g} \label{sto1}, \\
\<1>_{\g, N}(t) &= \P_{\le N} \<1>_{\g}(t) =   S_\g(t) \phi_N + \Psi_{\g,N} \label{sto2},
\end{align}

\noi
for $t \ge 0$, and where 

\noi
\begin{align}
\phi_N = \P_{\le N} \phi,
\label{GFF1}
\end{align}

\noi
with $\phi$ as in \eqref{GFF}. A standard computation shows that $\{ (\g,t) \mapsto \<1>_{\g,N}(t) \}_{N \in \N}$ belongs to $C \big( [0,1] \times [0,T]; W^{0-, \infty}(\T^2)\big)$ almost surely for any $T>0$; see Lemma \ref{LEM:sto1}.


Fix $N \in \N$ and $\g \in [0,1]$. Let $u_{\g,N}$ be the solution to \eqref{rSCGL1} (or \eqref{rNLS1} for $\g = 0$). We proceed with the following first order expansion (\cite{McKean, BO96, DPD}):

\noi
\begin{align}
u_{\g,N} =  v_{\g,N} +   \<1>_{\g} = (v_{\g,N} + \<1>_{\g,N}) + \P_{>N} \<1>_{\g},
\label{expa}
\end{align}

\noi
where $\P_{>N} = \operatorname{Id} - \P_{ \le N}$. We see that the dynamics of the truncated renormalized equation \eqref{rSCGL1} decouples into the linear dynamics for the high frequency part given by $\P_{>N} \<1>_{\g}$ and the nonlinear dynamics for the low frequency part of $\P_{\le N} u_{\g,N}$:

\noi
\begin{align}
\dt \P_{\le N} u_{\g, N}  & = (\g + i)( \Dl - 1 ) \P_{\le N} u_{\g,N} - (\g + i) \P_{\le N} \mathfrak{N} \big( \P_{\le N}u_{\g,N} \big) + \sqrt{2 \g} \P_{\le N} \xi,
 \end{align}
 
\noi
with initial data given by $\phi_N$ as in \eqref{GFF1}. Then, the nonlinear remainder $v_{\g,N} =  \P_{\le N} u_{\g, N} - \<1>_{\g,N}$ satisfies the following integral equation:

\begin{align}
\begin{split}
v_{\g,N}  & = - (\g + i) \P_{\le N} \I_\g \mf N \big(  \<1>_{\g,N} + v_{\g,N} \big) \\
& = - (\g + i) \P_{\le N} \Big(  \I_\g \NN(v_{\g,N})  + \I_\g \NN(\<1>_{\g,N}) \\
& \quad + 2 \I_\g \NN \big( \<1>_{\g,N}, v_{\g,N}, v_{\g,N} \big) + \I_\g \NN \big( v_{\g,N} , \<1>_{\g,N}, v_{\g,N} \big) \\
& \quad + 2 \I_\g \NN\big( v_{\g,N}, \<1>_{\g,N}, \<1>_{\g,N} \big) + \I_\g \NN \big( \<1>_{\g,N}, v_{\g,N}, \<1>_{\g,N}  \Big) \\
& \quad + \I_\g \Rr \big(\<1>_{\g,N} + v_{\g,N} \big) \Big),
\end{split}
\label{veq2}
\end{align}

\noi
with the zero initial data. 

In view of the ill-posedness of NLS \eqref{NLS} below $L^2(\T^2)$ \cite{Oh, Kishimoto}, we have to put $v_{\g,N}$ in $H^{s}(\T^2)$ for some $s>0$ and uniformly in $N \in \N$. We achieve this by solving a fixed point argument for $v_{\g,N}$; see Proposition \ref{PROP:lwp}. The terms in \eqref{veq2} fall into four categories:

\noi
\begin{itemize}
\item[(i)] a purely deterministic term:  $\I_\g \NN(v_{\g,N})$,

\item[(ii)] a stochastic term:

\noi
\begin{align}
\<30>_{\g,N} := \I_\g ( \<3>_{\g,N} ) \text{ with } \<3>_{\g,N} :=  \NN(\<1>_{\g,N}),
\label{sto_cubic}
\end{align}

\item[(iii)] two random matrix terms:

\noi
\begin{align}
\begin{split}
\mathfrak{M}^1_{\g,N} & : v \mapsto \I_\g \NN\big( v, \<1>_{\g,N}, \<1>_{\g,N} \big), \\
\mathfrak{M}^2_{\g,N} &:   v \mapsto \I_\g \NN \big( \<1>_{\g,N}, v, \<1>_{\g,N}  \big),
\end{split}
\label{rmt}
\end{align}

\noi
and two bilinear random operator terms: 

\noi
\begin{align}
\begin{split}
\mathfrak{T}^1_{\g,N} & : (u,v) \mapsto \I_\g \NN \big( \<1>_{\g,N}, u, v \big), \\
 \mathfrak{T}^2_{\g,N} & : (u,v) \mapsto \I_\g \NN \big( u , \<1>_{\g,N}, v \big),
\end{split}
\label{rbili}
\end{align}

\item[(iv)] a resonant term $\I_\g \Rr \big(\<1>_{\g,N} + v \big)$.
\end{itemize}

\noi
All these terms will be considered in $X^{s,b}$-spaces (see Definition \ref{DEF:X}). The deterministic term (i) is dealt with by a slight modification of Bourgain's trilinear estimate \cite{BO93}, see Lemma \ref{LEM:S2} in Section \ref{SEC:2}. The stochastic terms (ii) and (iii) are estimated in the following way: we first represent them using multiple stochastic integrals, see Appendix \ref{SEC:B}, and we then use counting arguments (see Section \ref{SEC:4}) combined with the random tensor estimate of Deng, Nahmod, and Yue \cite{DNY2} adapted to our stochastic setting (see Appendix \ref{SEC:C}) to close the relevant estimates. Furthermore, we construct the aforementioned terms in spaces of functions that are continuous in $\g \in [0,1]$ by using the standard Kolmogorov continuity criterion. After completing this program, (the local-in-time part of) Theorem \ref{THM:main} reduces to solving the equation \eqref{veq2} in a space of functions that are continuous in $\g \in [0,1]$ by the standard Banach fixed point theorem. This proves, the convergence of $v_{\g,N}$ to some non-trivial stochastic process $v_\g$ for each $\g \in [0,1]$ as $N \to \infty$. The process $u_\g = \<1>_\g + v_\g$ is then interpreted as the solution to \eqref{rSCGL} (or \eqref{rNLS} when $\g = 0$).

Let us now describe the main difficulty in establishing bounds for the objects (ii) and (iii). The key observation is that at the level of the heat Duhamel operators $\I_\g$, the dissipative smoothing effects come after a time integration. This smoothing mechanism is needed to handle the roughness in space of the noise $\xi$. However, in the $X^{s,b}$-spaces analysis, time integration is also used to benefit from multilinear dispersive smoothing effects. Thus, within the framework of $X^{s,b}$-spaces, there is a tradeoff, at the level of the heat model \eqref{rSCGL}, between parabolic and multilinear dispersive smoothing effects, which cannot be used simultaneously. We overcome this difficulty by combining both parabolic and dispersive analysis. However, in the weakly dissipative limit $\g \to 0$, this results in a loss of derivatives for the remainder $v_{\g}$; see Remark \ref{RMK:bad} below.

\noi
\begin{remark}\rm \label{RMK:bad}
Let us discuss the bounds available for the remainders $v_\g$, $\g \in [0,1]$. In the parabolic setting $\g >0$, Da Prato and Debussche \cite{DPD} essentially proved the following bound:

\noi
\begin{align*}
\|v_{\g }\|_{C_T H_x^{1}} \les_\g 1,
\end{align*}
\noi
for some small time $T>0$. At the level of NLS \eqref{rNLS}, Bourgain's argument gives the bound

\noi
\begin{align*}
\|v_{\g = 0}\|_{C_T H_x^{\frac12-}} < \infty,
\end{align*}

\noi
However, surprisingly, because of the issues described in the above, our argument only gives the following uniform (in $\g$) bound for the remainder $v_\g$:

\noi
\begin{align*}
\sup_{\g \in [0,1]} \| v_\g \|_{C_T H_x^{\frac14-}} < \infty,
\end{align*}

\noi
Hence, there appears to be a $\frac14$-derivative loss for the remainder $v_\g$ in the limit $\g \to 0$.
\end{remark} 

\noi
\begin{remark}\rm
We point out that the Fourier restriction norm method had already been applied to heat models in the literature. For instance, in \cite{MPV, MV1, MV2}, Molinet, Pilod, and Vento used $X^{s,b}$-type spaces to study the well-posedness issue for dispersive perturbations of the Burgers' equation in various contexts. They however do not address the question of the inviscid limit.
\end{remark}
\subsection{Further remarks}

\begin{remark}\rm
It would be of interest to study the inviscid limit proved in Theorem \ref{THM:main} for higher order nonlinearities. Let $k \in 2 \N +1$ with $k \ge 5$ and consider \eqref{SCGL} and \eqref{NLS} with the nonlinearity given by $|u|^{k-1} u$. With the corresponding Gibbs measure measure initial data\footnote{Namely, \eqref{Phi1} with the interaction potential given by $\frac{1}{k+1} \int_{\T^2} |u|^{k+1} dx$.}, well-posedness for the heat model \eqref{SCGL} was proved by Da Prato and Debussche \cite{DPD}, while well-posedness for \eqref{NLS} was established only recently by Deng, Nahmod, and Yue \cite{DNY1} in a breakthrough work. In order to achieve this result, they introduced the theory of random averaging operators. We plan to show the inviscid limit for these models by adapting the random averaging operators theory to parabolic (and stochastic) setting. However, in view of the issue discussed in Remark \ref{RMK:bad}, achieving this convergence might require using the random tensor theory \cite{DNY2}, a very recent extension of the theory of random averaging operators.
%
%
\end{remark}

\noi
\begin{remark}[Wick renormalization] \label{RM:Wick} \rm
Let us now consider the {\it Wick renormalization} scheme for \eqref{SCGL} and \eqref{NLS}. These equations read as follows for $\g \in [0,1]$:

\noi
\begin{align}
\dt u_{\g,N}  = (\g + i)( \Dl - 1 ) u_{\g,N} - (\g + i)  \mf(|u_{\g,N}|^2 - 2 \s_N) u_{\g,N} + \sqrt{2 \g} \xi_N,
\label{rSCGL10}
\end{align}

\noi
with initial data given by $\phi_N = \P_N \phi$ with $\phi$ as in \eqref{GFF}, and where $\s_N$ is defined by 

\noi
\begin{align}
\s_N = \E \big[  | \phi_N  |^2  \big] = \sum_{\jb n \le N} \frac{1}{ \jb{n}^2  } \sim \log (N).
\label{var}
\end{align}

Proceeding again with the first oder expansion of $u_{\g,N}$ gives

\noi
\begin{align}
u_{\g,N} = \<1>_{\g,N} + v_{\g,N}.
\label{expa}
\end{align}

\noi
We find after a computation, that $v_{\g,N}$ satisfies the following equation:

\begin{align}
\begin{split}
\dt v_{\g,N} - (\g + i)(\Dl - 1)v_{\g,N} & = - (\g + i) \Big( : \! | \<1>_{\g,N} |^2 \<1>_{\g,N}  \! : +  |v_{\g,N}|^2 \cj{v _{\g,N}} \\
& \quad + 2 : \! | \<1>_{\g,N} |^2  \! :  v_{\g,N} + \<1>_{\g,N}^2 \, v_{\g,N} \\
& \quad + 2 \<1>_{\g,N} |v_{\g,N}|^2 + \cj{\<1>_{\g,N}} \, v_{\g,N}^2 \Big) \\
& = - (\g + i) \big(  | \<1>_{\g,N} + v_{\g,N} |^2 - 2 \s_N \big) ( \<1>_{\g,N} + v_{\g,N} ).
\end{split}
\label{rSCGL11}
\end{align}

\noi
In the above, the terms $: \! | \<1>_{\g,N} |^2  \! :$ and $: \! | \<1>_{\g,N} |^2 \<1>_{\g,N}  \! : $ are the Wick renormalized powers given by

\noi
\begin{align}
\begin{split}
: \! | \<1>_{\g,N} |^2  \! : & = | \<1>_{\g,N} |^2 - \s_N \\
: \! | \<1>_{\g,N} |^2 \<1>_{\g,N}  \! : &= \big( | \<1>_{\g,N} |^2 - 2 \s_N \big) \<1>_{\g,N}
\end{split}
\label{lag1}
\end{align}

We would like to solve a fixed point problem (uniformly in $N \ge 1$) at the level of $v_{\g,N}$. However, this strategy is bound to fail. Indeed, for $\g = 0$, we can decompose $: \! | \<1>_{0,N} |^2 \<1>_{0,N}  \! :$ into two components:

\noi
\begin{align}
: \! | \<1>_{0,N} |^2 \<1>_{0,N}  \! :(x,t) & = \sum_{\substack{n_1, n_1, n_2 \\ n_2 \neq n_1, n_3  \\   \jb {n_j} \le N}} \frac{g_{n_1} \cj{ g_{n_2} } g_{n_3} } { \jb{n_1} \jb{n_2} \jb{n_3} } e^{i(n_1 - n_2 + n_3) \cdot x - i (\jb{n_1}^2 - \jb{n_2}^2  + \jb{n_3}^2 )t}  \\
& \qquad  + 2 \Big( \sum_{\jb{n_2} \le N}  \frac{|g_{n_2} |^2 -1  }{\jb{n_2}^2} \Big)  \sum_{ \jb{n_1} \le N } \frac{g_{n_1}}{ \jb {n_1} } e^{i n_1 \cdot x - i \jb{n_1}^2 t}  \\
& =: \<3>_{0,N}(x,t) + \<1>^{(1)}_{0,N}(x,t),
\end{align}

\noi
where $ \<3>_{0,N}$ is as in \eqref{sto_cubic}. Due to the roughness in space of the first order object $\<1>_{\g,N} \in H^{0-}(\T^2)$, one can only hope, a priori, to place $\<3>_{0,N}$ in $H^{0-}(\T^2)$. However, Bourgain \cite{BO96} showed that, under the Duhamel integral $\I_0$, the term $\<3>_{0,N}$ enjoys some smoothing and can be placed in $H^{\frac12 -}(\T^2)$. See Lemma \ref{LEM:sto_cubic}. This phenomenon is known as \textit{multilinear smoothing}: by exploiting the interaction of random waves through counting estimates, we can prove a  $\frac12-$ derivatives gain on $\<3>_{0,N}$. See also \cite{GKO2, DNY1, DNY2, Bring, OWZ, TS} for examples of multilinear smoothing in other contexts.

The second term $\<1>^{(1)}_{0,N}$ however verifies $\<1>^{(1)}_{0,N} = C_N  \<1>_{0,N}$, where $C_N$ is the almost surely converging sequence (as $N \to \infty$) given by $C_N = 2 \sum_{\jb{n_2} \le N}  \frac{|g_{n_2} |^2 -1  }{\jb{n_2}^2}$. Hence, in view of the regularity of $\<1>_{0,N}$ discussed above, $\<1>^{(1)}_{0,N} \in H^{0-}(\T^2) \setminus L^2(\T^2)$ and it is not possible to close the fixed point argument \eqref{rSCGL11} for $v_{0,N}$.

In order to solve this issue and remove these ``bad" terms, we follow Bourgain \cite{BO96} and introduce the following Gauge transformation:

\noi
\begin{align}
\mathcal{G}_{N}(u) = e^{i V_{N}(u)} u,
\label{gauge1}
\end{align}

\noi
with, 

\noi
\begin{align}
V_{N}(u)(t) = \int_{0}^t \Big( \int_{\T^2} |u(t') |^2 dx - \s_N  \Big) dt' \in \R
\label{gauge2}
\end{align}

\noi
 for any $t \ge 0$, $N \in \N$ and some smooth $u \in \mathcal{S}(\T^2 \times \R)$. 

Let $u^g_{\g,N}$ be defined by

\noi
\begin{align}
u^g_{\g,N} = \mathcal{G}_{N}(u_{\g,N}).
\label{gauge3}
\end{align}

\noi
where $u_{\g,N}$ solves \eqref{rSCGL10}. Then $u^g_{\g,N}$ solves

\noi
\begin{align}
\begin{split}
\dt u^g_{\g,N}  & = (\g + i)( \Dl - 1 ) u^g_{\g,N} - \g (|u_{\g,N}|^2 - 2 \s_N) u_{\g,N}  - i \mf N(u^g_{\g,N  
}) \\
& \qquad \qquad + \sqrt{2 \g}  e^{i V_N(u_{\g,N} ) }\xi_N, 
\end{split}
\label{rSCGL12}
\end{align}

The main problem in trying to solve \eqref{rSCGL12} is the non-Gaussianity of the noise term $\xi^g_N := e^{i V_N(u_{\g,N} ) }\xi_N$. Moreover, the Fourier modes of $\xi^g_N$ are not independent anymore, which is crucially needed for the proof of Theorem \ref{THM:main}. 
\end{remark}

\section{Notations, function spaces and basic lemmas}\label{SEC:2}
\subsection{Notations and basic function spaces}\label{SEC:2_1}

We write $A \les B$ to denote an estimate of the form $A \le C B$ for some constant $C >0$. Similarly, we write $A \sim B$ to denote $A \les B$ and $B \les A$ and use $A \ll B$ when we have $A \le c B$ for some small $c>0$.

We use throughout the paper shortcut notations such as $L^{\infty}_T H^s_x$ and $C_{\g,T} X$ (for functions of the form $f = f(\g,t,x)$) for $L^{\infty} \big( [0,T]; H^s(\T^2) \big)$ and $C \big( [0,1] \times [0,T]; X \big)$ respectively, etc. (Here, $X$ is any space of functions of the variables $x \in \T^2$. In this paper, $X$ can be either the usual Sobolev space $H^s(\T^2)$ or $X^{s,b}$ defined in Definition \ref{DEF:X} below.)

Let $f \in \mathcal{S}'(\T^2)$, $g \in \mathcal{S}'(\R)$ and $u \in \mathcal{S }'(\T^2 \times \R)$. We define the Fourier transforms $\F_x$, $\F_t$ and $\F_{x,t}$ through the following formulas:

\noi
\begin{align}
\begin{split}
 \F_{x}(f)(n) &= \frac{1}{(2 \pi)^2} \int_{\T^2} e^{-i n \cdot x} f(x) dx \\
 \F_t(g)(\ld) & = \frac{1}{2 \pi} \int_\R e^{-i \ld t} g(t) dt \\
 \F_{x,t}(u)(\ld,n)  &= \frac{1}{(2\pi)^3} \int_{\T^2 \times \R} e^{- i (n\cdot x + t \ld)} u(x,t) dx dt
\end{split}
\label{ft1}
\end{align} 

\noi
for $(n,\ld) \in \Z^2 \times \R$. When there is no confusion, we simple use $\ft u$ to denote the spatial, temporal or space-time Fourier transform of $u \in \mathcal{S }'(\T^2 \times \R)$. In what follows, we will omit the non-essential dependence in the factor $2 \pi$ in our estimates for convenience.

We also denote by $\tw u$ or $\tw \F (u)$ the twisted space-time Fourier transform of $u \in \mathcal{S }'(\T^2 \times \R)$ by

\noi
\begin{align}
\tw u (n,\ld) = \F_{x,t}(u) \big(n, \ld - \jb n ^2 \big).
\label{ft2}
\end{align}

For a number $N \in \Z$, we denote by $\P_N$ the (sharp) projection onto the set $\{ n \in \Z^2: \jb n \sim N \}$.  Namely, we have

\noi
\begin{align}
\P_N f(x) = \sum_{\jb n \sim N} \ft f (n) e^{i n \cdot x},
\label{proj2}
\end{align}

\noi
for any $N \in \N$. For any set $Q \subset \Z^2$ we denote by $\P_Q$ the Fourier projection onto $Q$. Also, recall, for $N \in \N$ the projections $\P_{\le N}$ \eqref{proj2}. We also denote by $\P_{\gg N}$ the sharp Fourier projection onto the set $\{ n \in \Z^2 : \jb n \gg N \}$. Note that with our notations, we have

\noi
\[ \P_{\gg N} = \sum_{N' \gg N} \P_{N'}. \]

For a complex number $z$ we sometimes use the notations $z^{1}$ and $z^{-1}$ for $z$ and $\cj z$ respectively. Given a set $P$, we denote by $|P|$ its cardinality. Given two positive numbers $a$ and $b$, we will also denote by $a \vee b$ and $a \wedge b$ the minimum and maximum, respectively, of $a$ and $b$.

%
%
%
%
%
\subsection{Fourier restriction norm and Strichartz estimates}\label{SEC:2_2}

For each 

Let $s \in \R$ and $1 \le p \le \infty$. We define the $L^2$-based Sobolev space $H^{s}(\T^2)$ by the norm:

\noi
\begin{align} 
\|f \|_{H^s} = \| \jb n ^s \ft f (n) \|_{\l ^2 _n}.
\end{align}

\noi
\begin{definition}\rm \label{DEF:X}
Let $(s, b) \in \R^2 $. We define the $X^{s,b}$ space as the completion of $\mathcal{S}(\T^2 \times \R)$ under the norm

\noi
\begin{align}
\|u\|_{X^{s,b}(\T^2 \times \R)} = \| \jb n ^s \jb \ld ^b \tw u (n , \ld)  \|_{ \l^2_n L^2_\ld( \Z^2 \times \R)}.
\label{X1}
\end{align}

\noi
Given an interval $I \subset \R$, we define the local-in-time version $X^{s,b}(I)$ as a restriction norm:

\noi
\begin{align}
\|u \|_{X^{s,b}(I)} = \inf \{ \|v\|_{X^{s,b}(\T^2 \times \R)} : v |_{I} =u \}.
\label{X2}
\end{align}

\noi
When $I = [0,T]$, we set $X^{s,b}_T = X^{s,b}(I)$.
\end{definition}

\noi
Let us note that for all $s \in \R $ and $b > \frac12$, we have $X^{s,b} \hookrightarrow C(\R; H^s(\T^2))$.

\noi
\begin{remark}\label{RMK:X2}\rm
The $X^{s,b}$ spaces are usually defined with a modulation variable $\jb{\ld + |n|^2}$ rather than $\jb{\ld + \jb{n}^2}$ as in \eqref{ft2} and \eqref{X1}. This however does not modify any estimates since these two respectively norms are equivalent.
\end{remark}

The following lemma is the so-called $L^4$-Strichartz estimate and was proved by Bourgain~\cite{BO93}.

\noi
\begin{lemma}\label{LEM:S1}
Let $N \in \N$ and $Q$ be a spatial frequency ball \textup{(}not necessarily centered at the origin\textup{)}. Then we have the bound

\noi
\begin{align*}
\| \P_Q u \|_{L^4(\T^2 \times [0,1])} \les |Q|^{\eps} \, \|u\|_{X^{0,\frac12}},
\end{align*}

\noi
for any $\eps >0$.
\end{lemma}

\noi
By interpolating the bound of Lemma \ref{LEM:S1} and the following inequality (obtained Sobolev's inequality in space followed by Minkowski's inequality and Sobolev's inequality in time),

\noi
\begin{align*}
\| \P_Q u \|_{L^4(\T^2 \times [0,1])} \les |Q|^{\frac12} \, \|u\|_{X^{0,\frac14}},
\end{align*}

\noi
gives the next trilinear estimate, which is a slight variation of Bourgain's trilinear estimate \cite{BO93}.

\noi
\begin{lemma}\label{LEM:S2}
Let $s >0$, $0 < T \le 1$, and $0 <\eps \ll 1 $. We have

\noi
\begin{align*}
& \| \NN \big( u_1,  u_2, u_3 \big) \|_{C_\g X_T^{s,-\frac12 + \eps}} \les \max_{\s \in S_3} \| u_{\s(1)}\|_{C_\g X_T^{s,\frac12-\eps}}  \, \|  u_{\s(2)} \|_{C_\g X_T^{10 \eps,\frac12-\eps}}  \, \| u_{\s(3)}\|_{C_\g X_T^{0,\frac12-\eps}}.
\end{align*}

\noi
Here, $S_3$ denotes the group of permutations on $\{1,2,3\}$.
\end{lemma}

\begin{remark}\rm 
(i) Note that the logarithmic losses in spatial derivatives in Lemma \ref{LEM:S2} is in the second largest frequency (i.e. not in $N_{\max})$. This is achieved by orthogonality. See for instance Lemma \ref{LEM:RMT} for an implementation of this trick (due to Bourgain \cite{BO93, BO96}).

\noi
(ii) As opposed to Bourgain's trilinear estimate \cite{BO93}, we only require our inputs to be in spaces of the form $X^{\al,b}$ for $b<\frac12$. This is because, in some cases, the stochastic convolution $\Psi_\g$ \eqref{conv1} will be such an input. As Brownian motion is $\big(\frac12-\eps \big)$-H\"older in time for any $\eps >0$, $\Psi_\g$ can only be placed in $X^{\al, b}$ for $\al < 0$ and $b < \frac12$; see Lemma \ref{LEM:sto1}.
\end{remark}

\noi
\begin{definition}\label{DEF:OP}
Given Banach spaces $A_1$, $A_2$ and $A_3$ we use $\mathcal{L}(A_1; A_3)$ and $\mathcal{B}(A_1 \times A_2, A_3) $ to denote the space of bounded linear and bilinear operators from $A_1$ to $A_3$ and $A_1 \times A_2$ to $A_3$, respectively. For $T>0$, we also defined the spaces

\noi
\begin{align}
\begin{split}
\L_{T_0}^{s_1, s_2, b_1, b_2} & := \bigcap_{0 < T < T_0} \L \Big( C \big( [0,1] ; X^{s_1, b_1}([0,T]) \big); C \big( [0,1] ; X^{s_2, b_2}([0,T]) \big) \Big) \\
\B_{ T_0}^{s_1, s_2, b_1, b_2} & := \bigcap_{0 < T < T_0} \B \Big( C \big( [0,1] ; X^{s_1, b_1}([0,T]) \big)^2; C \big( [0,1] ; X^{s_2, b_2}([0,T]) \big) \Big)\end{split}
\label{op1}
\end{align}

\noi
for any $s_1, s_2, b_1, b_2 \in \R$; endowed with the norms

\noi
\begin{align}
\begin{split}
\big\| \MM  \big\|_{\L_{T_0}^{s_1, s_2, b_1, b_2}} & := \sup_{0 < T < T_0} T^{-\ta} \big\| \MM \big\|_{ \L ( C_\g X^{s_1, b_1}_T;  C_\g X^{s_2, b_2}_T ) } \\
\big\| \TT  \big\|_{\B_{T_0}^{s_1, s_2, b_1, b_2}} & := \sup_{0 < T < T_0} T^{-\ta} \big\| \TT \big\|_{ \B ( ( C_\g  X^{s_1, b_1}_T )^2; C_\g X^{s_2, b_2}_T ) }, 
\end{split}
\label{op2}
\end{align}

\noi
respectively and for some small $\ta >0$.
\end{definition}
\subsection{Tools from stochastic analysis}\label{SEC:2_3}

In this subsection, 
by recalling some basic tools from probability theory and Euclidean quantum field theory
(\cite{Kuo, Nua, Simon}). First, 
recall the Hermite polynomials $H_k(x; \s)$ 
defined through the generating function:
\begin{equation*}
F(t, x; \s) \stackrel{\text{def}}{=}  e^{tx - \frac{1}{2}\s t^2} = \sum_{k = 0}^\infty \frac{t^k}{k!} H_k(x;\s).
 \end{equation*}
	
\noi
For readers' convenience, we write out the first few Hermite polynomials:
\begin{align*}
\begin{split}
& H_0(x; \s) = 1, 
\quad 
H_1(x; \s) = x, 
\quad
H_2(x; \s) = x^2 - \s,   
\quad
 H_3(x; \s) = x^3 - 3\s x.
\end{split}
\end{align*}
	
\noi
Note that the Hermite polynomials verify the following standard identity:

\noi
\begin{align}
H_k(x+y;\s) = \sum_{\l = 0}^k x^{k-\l} H_{\l}(y; \s). 
\label{Herm10}
\end{align}

Next, we recall the Wiener chaos estimate.
Let $(H, B, \mu)$ be an abstract Wiener space.
Namely, $\mu$ is a Gaussian measure on a separable Banach space $B$
with $H \subset B$ as its Cameron-Martin space.
Given  a complete orthonormal system $\{e_j \}_{ j \in \N} \subset B^*$ of $H^* = H$, 
we  define a polynomial chaos of order
$k$ to be an element of the form $\prod_{j = 1}^\infty H_{k_j}(\jb{x, e_j})$, 
where $x \in B$, $k_j \ne 0$ for only finitely many $j$'s, $k= \sum_{j = 1}^\infty k_j$, 
$H_{k_j}$ is the Hermite polynomial of degree $k_j$, 
and $\jb{\cdot, \cdot} = \vphantom{|}_B \jb{\cdot, \cdot}_{B^*}$ denotes the $B$-$B^*$ duality pairing.
We then 
denote the closure  of the span of
polynomial chaoses of order $k$ 
under $L^2(B, \mu)$ by $\mathcal{H}_k$.
The elements in $\H_k$ 
are called homogeneous Wiener chaoses of order $k$.
We also set
\[ \H_{\leq k} = \bigoplus_{j = 0}^k \H_j\]

\noi
 for $k \in \N$.
 
 \noi
 \begin{remark}\rm \label{complex Wiener}
Let us note that in this paper we deal with complex-valued random variables. However, one can reduce the analysis to the real-valued case by considering the real and imaginary parts of these random variables; see for instance \cite{OT}.
\end{remark}

Let $L = \Dl -x \cdot \nabla$ be 
 the Ornstein-Uhlenbeck operator.\footnote{For simplicity, 
 we write the definition of the Ornstein-Uhlenbeck operator $L$
 when $B = \R^d$.}
Then, 
it is known that 
any element in $\mathcal H_k$ 
is an eigenfunction of $L$ with eigenvalue $-k$.
Then, as a consequence
of the  hypercontractivity of the Ornstein-Uhlenbeck
semigroup $U(t) = e^{tL}$ due to Nelson \cite{Nelson2}, 
we have the following Wiener chaos estimate
\cite[Theorem~I.22]{Simon}.
See also \cite[Proposition~2.4]{TTz}.

\begin{lemma}\label{LEM:hyp}
Let $k \in \N$.
Then, we have
\begin{equation*}
\|X \|_{L^p(\O)} \leq (p-1)^\frac{k}{2} \|X\|_{L^2(\O)}
 \end{equation*}
 
 \noi
 for any $p \geq 2$
 and any $X \in \H_{\leq k}$.
\end{lemma}

\noi
\section{Deterministic estimates}\label{SEC:3} Here, we prove deterministic estimates adapted to our setting on the linear operators $S_\g$ and $\I_\g$ defined in \eqref{lin1} and \eqref{duha1}, respectively.
\subsection{Linear homogeneous estimates}\label{SEC:3_1}

\begin{lemma}\label{LEM:lin1}
Fix $T >0$ and $s \in \R$. We have the following bound:

\noi
\begin{align}
\| S_\g (t) f  \|_{C_{\g,T} H^s_x} \les \| f \|_{C_\g H^s_x},
\end{align}

\noi
for any $f \in H^s(\T^2)$.
\end{lemma}

\begin{proof} This follows immediately from the bound $\big| e^{(\g |t| + it) \jb n ^2  } \big| \le 1$ for any $t \in [0,T]$ and the dominated convergence theorem.
\end{proof}

\subsection{Linear nonhomogeneous estimates}\label{SEC:3_2}

Recall the Duhamel operators $\{ \I_\g \}_{\g \in [0,1]}$ defined in \eqref{duha1}. In this subsection, we study the boundedness and continuity in $\g \in [0,1]$ properties of the family of operators $\{ \I_\g \}_{\g \in [0,1]}$.

We use the following lemma.

\noi
\begin{lemma}\label{LEM:duha0}
Let $s \in \R$. Fix $\varphi$ a Schwartz function and $0 < T \le 1$. Then we have

\noi
\begin{align*}
\| \varphi  ( t/T ) u \|_{ X^{s,b_2} } \les T ^{b_2 - b_1} \| u \|_{X^{s,b_1}},
\end{align*}

\noi
for any $b_2 > b_1 > \frac12$ and $u \in X^{s,b_1}$ such that $u(0) = 0$.
\end{lemma}

\noi
\begin{proof} See \cite[Proposition 2.7]{DNY2}.
\end{proof}
\noi
\begin{proposition}\label{PROP:duha1}
Let $s \in \R$. Fix $0 < \eps, \dl \ll 1$ with $\eps \ll \dl^2$ and $0 < T \le 1$. Then, we have the bounds

\noi
\begin{align}
\big\| \I_\g (F_\g) \big\|_{ C_\g X^{ s, \frac12 + \eps } _T} &  \les \| F \|_{C_\g X_T ^{ s, -\frac12 + \dl } } \label{D00},  \\ 
\big\| \g^{\frac12 - \dl} \I_\g (F_\g) \big\|_{ C_\g X^{ s, \frac12 + \eps } _T} & \les \| F \|_{C_\g L_T^2 H ^{s - 1 + 2\dl} } \label{D01}.
\end{align}
\end{proposition}

\noi
The bound \eqref{D00} is the natural generalization of the standard nonhomogeneous estimate in $X^{s,b}$-spaces (see for instance \cite{TAO}) to the setting of our convergence problem (for which we need to prove bounds uniformly in the parameter $\g \in [0,1]$), while \eqref{D01} captures the dissipative smoothing effects of $\I_\g$ \eqref{duha1} in $X^{s,b}$-spaces. In particular, \eqref{D01} is crucial in the proof of Lemma \ref{LEM:bilin1}.

We first recall the following standard lemma.

\noi
\begin{lemma}\label{LEM:conv}
Let $0 \le \al , \be$ such that $\al + \be > 1$. Fix $\mu \in \R$. Then, we have 

\noi
\begin{align*}
\int_{\R} \frac{d \ld}{ \jb \ld ^\al \jb{\mu - \ld}^\be  } \les \frac{1}{\jb \mu ^\s}
\end{align*}

\noi
with 

\noi
\begin{align*}
\s = 
\begin{cases}
\al + \be - 1, \quad & \textup{if $\be < 1$} \\
\al - \eps, \quad & \textup{if $\be = 1$} \\
\al, \quad & \textup{if $\be > 1$},
\end{cases}
\end{align*}

\noi
for any $\eps >0$.
\end{lemma}

The next two lemmas will essentially reduce the proof of Proposition \ref{PROP:duha1} to bounds on kernels \eqref{duha1}.

\noi
\begin{lemma}\label{LEM:duha1}
Let $T$ be a linear integral operator with kernel $K$ defined on the \textup{(}time-\textup{)} frequency side by

\noi
\begin{align*}
\ft{T(F)}(\ld) = \int_\R \ft F(\mu) K(\ld,\mu) d \mu,
\end{align*}

\noi
such that the kernel $K$ satisfies the following estimate:

\noi
\begin{align}
|K (\ld, \mu)| \les \frac{1}{\jb \mu} \Big( \frac{1}{\jb \ld ^2} +\frac{1}{\jb{\ld-\mu} ^2}    \Big),
\label{L1}
\end{align}

\noi
for any $(\ld,\mu) \in \R^2$. Then, the following bound holds:

\noi
\begin{align}
\| T (F) \|_{H^b} \les \| F \|_{H^{-b'}},
\label{L2}
\end{align}

\noi
for any $0 \le b' < \frac12$, $0 \le b \le 1$ with $b + b' \le 1$.
\end{lemma}

\begin{proof} Fix $b$ and $b'$ as in the above. Let $\wt T$ be the linear integral operator with kernel given (on the time-Fourier side) by 

\noi
\begin{align}
\wt K (\ld, \mu) & = \frac{\jb \ld ^b \jb \mu ^{b'} }{\jb \mu} \Big( \frac{1}{\jb \ld ^2} +\frac{1}{\jb{\ld-\mu} ^2}    \Big)  \nonumber \\
& =: \wt K^1(\ld,\mu) + \wt  K^2(\ld,\mu). \label{T1}
\end{align}

\noi
Let us note that $H^{-b'} \to H^b$ bounds for $T$ follow from $L^2 \to L^2$ bounds for $\wt T$, which we prove now. 

Next, we define the kernels $\wt K^1(\ld,\mu) = \jb \mu ^{b'-1} \jb \ld ^{b-2} $ and $\wt K_2(\ld,\mu) =  \jb \mu ^{b'-1} \jb \ld ^b  \jb{\ld-\mu} ^{-2} $ for $(\ld, \mu) \in \R^2$, and denote by $\wt T^1$ and $\wt T^2$ the linear integral operators whose kernels are given (on the time-Fourier side) by $\wt K^1$ and $\wt K^2$ respectively. To conclude the proof, it thus suffices to prove

\noi
\begin{align} 
\big\| \wt T^1 (F) \big\|_{L^2} & \les \| F \|_{L^2}  \label{T2} \\
\big\| \wt T^2 (F) \big\|_{L^2} & \les \| F \|_{L^2}  \label{T3} 
\end{align}

\noi
From Cauchy-Schwarz inequality and Plancherel, we have 

\noi
\begin{align}
\begin{split}
\big\| \wt  T^1(F) \big\|_{L^2} & \les \| \wt  K_1(\ld,\mu)  \|_{L^2_{\ld, \mu}} \| F \|_{L^2} \\
& \les_{b,b'} \| F \|_{L^2}.
\end{split}
\label{L3}
\end{align}

\noi
This proves \eqref{T2}. We now bound $\| \wt T^2 \|_{L^2 \to L^2}$. We further decompose the kernel $\wt K^2$ in the following way: 

\noi
\begin{align*}
\wt K_2(\ld, \mu) & = \wt K_2(\ld, \mu) \ind_{|\ld| \les |\mu|} + \wt K_2(\ld, \mu) \ind_{| \ld | \gg |\mu|} \\
& =: \wt K^{2,<} (\ld, \mu) + \wt K^{2, >} (\ld,\mu).
\end{align*}

\noi
Let us again denote by $\wt T^{2, <}$ and $\wt T^{2,>}$ the associated operators. From the bound $\wt K^{2,<}(\ld,\mu) \les \wt K^1(\ld,\mu)$, we deduce that $\wt T^{2, >}$ is bounded from $L_\mu^2$ to $L_\ld^2$ as in \eqref{L3}. Next, regarding $T^{2, <}$, by Plancherel's theorem, Young's inequality and the condition $| \ld | \les | \mu |$, we have

\noi
\begin{align*}
\big\| \wt  T^2(F) \big\|_{L^2} & \les \Big\| \int_{\R} \jb{\mu}^{b+b'-1} \jb{\ld - \mu}^{-2} |\ft F (\mu)| d \mu \Big\|_{L^2_\ld} \\
& \les \Big\| \int_{\R} \jb{\ld - \mu}^{-2} |\ft F (\mu)| d \mu \Big\|_{L^2_\ld} \les \| F \|_{L^2},
\end{align*}

\noi
which proves \eqref{T3} and concludes the proof.
\end{proof}

\noi
\begin{lemma}\label{LEM:duha2}
Let $\{T_a\}_{a \in \R_+}$ be a family of linear integral operators with kernels $\{K_a\}_{a \in \R_+}$ defined on the \textup{(}time-\textup{)} frequency side by

\noi
\begin{align*}
\ft{T_a(F)} (\ld) = \int_\R \ft F(\mu) K_a(\ld,\mu) d \mu
\end{align*}

\noi
such that the kernels $\{K_a\}_{a \in \R_+}$ satisfy the following bounds:

\noi
\begin{align}
|K_a(\ld,\mu) | \les \frac{1}{ \jb{a + i \mu} } \min \Big( \frac{1}{\jb \ld} +\frac{1}{\jb{\ld-\mu}} , \frac{ \jb a }{\jb \ld ^2} +\frac{\jb a}{\jb{\ld-\mu} ^2}    \Big),
\label{L6}
\end{align}

\noi
for any $(\ld, \mu) \in \R^2 $, $a \in \R_+$, and

\noi
\begin{align}
| K_{a_2}(\ld,\mu) - K_{a_1}(\ld,\mu) | \les |a_2 - a_1|,
\label{L7}
\end{align}

\noi
for any $(\ld, \mu) \in \R^2$ and $(a_1, a_2) \in (\R_+)^2$. Then, the following bounds hold:

\noi
\textup{(i)} Let $0 \le b,b' < \frac12$. Then, we have

\noi
\begin{align*}
\| T_a(F) \|_{H^b} \les \jb{a}^{b'-\frac12} \|F\|_{H^{-b'}},
\end{align*}

\noi
for any $\eps > 0$ and $a \in \R_+$.

\noi
\textup{(ii)} Let $0 < b < \frac12$. Then, we have

\noi
\begin{align*}
\| T_a(F) \|_{H^b} \les \jb{a}^{-\frac12} \|F\|_{H^{\eps}},
\end{align*}

\noi
for any $\eps >0$ and $a \in \R_+$.

\noi
\textup{(iii)} Let $0 \le b' < \frac12$ and $0 \le b \le 1$ with $b+b' \le 1$. Then, we have

\noi
\begin{align*}
\| T_a(F) \|_{H^b} \les \jb{a} \|F\|_{H^{-b'}},
\end{align*}

\noi
for any $a \in \R_+$.

\noi
\textup{(iv)} Let $(a_1,a_2) \in (\R_+)^2$. Then, we have

\noi
\begin{align*}
\| (T_{a_2} - T_{a_1})(F) \|_{H^{-1}} \les |a_2 - a_1| \,   \|F\|_{H^{1}},
\end{align*}

\noi
for any $(a_1, a_2) \in (\R_+)^2$.

\noi
\textup{(v)} Fix any $\eps, \dl >0$ with $\eps \ll \dl^2$. Then, we have

\noi
\begin{align*}
 \| T_a(F) \|_{H^{\frac12 + \eps}} & \les \jb{a}^{-\frac \dl 2}  \|F\|_{H^{-\frac12 + \dl }}, \\
  \| T_a(F) \|_{H^{\frac12 + \eps}} & \les \jb{a}^{-\frac 1 2 + \dl}  \|F\|_{L^2}, \\
\| (T_{a_2} - T_{a_1})(F) \|_{H^{\frac12 + \eps}} & \les |a_2 - a_1|^\eps \,   \|F\|_{H^{-\frac12 + \dl}}.
\end{align*}

\noi
for any $(a, a_1, a_2) \in (\R_+)^3$. 
\end{lemma}

\begin{proof}
The bound (iv) follows from \eqref{L7} and Cauchy-Schwarz inequality. Note that we have

\noi
\begin{align}
\jb{a + i \mu} \ge \max( \jb \mu, \jb a).
\label{L8}
\end{align}

\noi
Hence, (iii) is a consequence of \eqref{L8}, \eqref{L6} and Lemma \ref{LEM:duha1}. Moreover, (v) follows by interpolating (i), (ii), (iii), and (iv).

We now prove (i). Let $0 \le b,b' < \frac12$. From \eqref{L6}, we can decompose 

\noi
\begin{align*}
\jb{\ld} ^b \jb \mu ^{b'} |K_a(\ld,\mu)| \les \frac{\jb{\ld} ^b \jb \mu ^{b'}}{ \jb{a + i \mu} \jb \ld } +   \frac{\jb{\ld} ^b \jb \mu ^{b'}}{ \jb{a + i \mu} \jb{\ld-\mu} } =: \wt K_a^1(\ld,\mu) + \wt K_a^2(\ld, \mu) 
\end{align*}

\noi
Let $\wt T_a^1$ and $\wt T_a^2$ be the linear operators with kernels $\wt K_a ^1$ and $\wt K_a ^2$ respectively. As in the proof of Lemma \ref{LEM:duha1}, proving (i) reduces to establish similar $L^2 \to L^2$ bounds for $\wt T_a^1$ and $\wt T_a^2$. Let us first consider $\wt T_a^1$. We have

\noi
\begin{align}
\| \wt K_a ^1 \|_{L^2_\ld L^2 _\mu}^2 = \int_{\R^2} \frac{ \jb \ld ^{2b} \jb \mu ^{2b'} }{ \jb{a + i \mu}^2 \jb \ld^2 } d \ld d \mu & \les \int_\R \frac{ \jb \mu ^{2b'} }{ 1 + a^2 + \mu^2 } d \mu
\label{L9}
\end{align}

\noi
If $|a| \le 1$, then we have $\eqref{L9} \les 1$. Otherwise, we bound, using a change of variable

\noi
\begin{align*}
\eqref{L9} \les \frac{1}{|a|}  \int_\R \frac{ \jb{a\mu} ^{2b'} }{ \frac{1}{a^2} + 1 + \mu^2 } d \mu \les |a|^{2b'-1} \int_{\R} \frac{1}{ \jb \mu ^{2 -2b' } } d\mu \les \jb{a}^{2b'-1}.
\end{align*}

\noi
Hence, we have

\noi
\begin{align}
\big\| \wt T^1 \big\|_{L^2 \to L^2} \les \jb{a}^{b' - \frac12},
\label{L10}
\end{align}

\noi
by Cauchy-Schwarz inequality.

In bounding $\wt T_a^2$, we may assume the additional condition $|\ld| \les |\mu|$ as in the proof of Lemma \ref{LEM:duha1}. Using Plancherel's theorem, \eqref{L8}, and the condition $0 \le b,b' < \frac12$ along with Young's inequality, we have

\noi
\begin{align}
\begin{split}
\big \| \wt T_a^2 (F) \big\|_{L^2} & \les \Big\|  \int_{\R}  \frac{\ind_{|\ld| \les |\mu|}}{ \jb{ a + i \mu}^{1 - b - b'} \jb{\ld - \mu} } |\ft F (\mu)| d \mu \Big\|_{L^2_\ld} \\
& \les  \jb{a}^{b+b'-1+ \eps}  \Big\| \int_{\R}  \frac{\ind_{|\ld| \les |\mu|}}{ \jb{\mu}^{- \eps} \jb{\ld - \mu} } |\ft F (\mu)| d \mu \Big\|_{L^2_\ld} \\
& \les  \jb{a}^{b+b'-1+ \eps}  \Big\| \int_{\R}  \frac{\ind_{|\ld| \les |\mu|}}{ \jb{\ld - \mu}^{1+\eps} } |\ft F (\mu)| d \mu \Big\|_{L^2_\ld}  \les \jb{a}^{b+b'-1+ \eps} \|F\|_{L^2},
\end{split}
\label{L10b}
\end{align}


\noi
for any $\eps >0$. Note that under the assumption $b < \frac12$, we have 

\noi
\begin{align*}
b + b' - 1 + \eps < b'-\frac12
\end{align*}

\noi
for $\eps >0$ small enough. Hence, by \eqref{L10b}, we have

\noi
\begin{align}
\big\| \wt T^2 \big\|_{L^2 \to L^2} \les \jb{a}^{b'- \frac12}
\label{L12}
\end{align}

\noi
Combining, \eqref{L10} and \eqref{L12} prove (i). Lastly, from \eqref{L6} and \eqref{L8}, we get

\noi
\begin{align*}
|K_a(\ld, \mu)| \les \frac{\jb a ^{-\frac12 + \eps }}{\jb{\mu}^{\frac12 + \eps}} \Big( \frac{1}{ \jb{\ld}  } + \frac{1}{  \jb{\ld - \mu}  }  \Big),
\end{align*}

\noi
for any $\eps >0$ and $a \in \R_+$. Making use of the above inequality and arguing as in the proof of Lemma \ref{LEM:duha1} yields (ii).

\end{proof}

We now prove Proposition \ref{PROP:duha1}.

\noi
\begin{proof}[Proof of Proposition \ref{PROP:duha1}]
We only prove \eqref{D00} as \eqref{D01} follows from similar arguments. Fix $0 < T \le 1$ and $0 < \eps, \dl \ll 1$ with $\eps^2 \ll \dl$. We first prove the following bound:

\noi
\begin{align}
\sup_{\g \in [0,1]} \big\| \I_\g (F) \big\|_{X^{ s, \frac12 + \eps } _T} \les \| F \|_{X_T ^{ s, -\frac12 + \dl }},
\label{D8}
\end{align}

\noi
for any $F \in X_T^{s, - \frac12 + \dl}$.

Let $\varphi = \varphi(t)$ be a smooth function such that $\varphi \equiv 1$ on $[-1,1]$ and $\varphi \equiv 0$ on $[-2,2]$. Let $F \in X_T^{s, - \frac12 + \dl}$ and fix $G$, an extension of $F$ to $\R$, such that $G_{|[0,T]} = F_{|[0,T]}$. Then, $ \varphi \I_\g(G)$ is an extension of $ \varphi \I_\g(F)$ to $\R$ which agrees with $ \varphi \I_\g(F)$ on $[0,T]$.  From \eqref{duha1}, we get

\noi
\begin{align}
\widetilde{ \varphi \I_\g(G)}(n,\ld) = \int_{\R}  \widetilde{G}(n,\mu)  K_\g(n,\ld,\mu) d \mu
\label{D1}
\end{align}

\noi
with

\noi
\begin{align}
K_\g(n,\ld,\mu) :=  \int_{\R}  \varphi (t) e^{-it \ld} \frac{e^{\g(t - |t|) \jb n ^2 + it \mu} - e^{-\g |t| \jb n ^2} }{ \g \jb n ^2 + i \mu } dt
\label{D2}
\end{align}

\noi
Note that we have
\noi
\begin{align}
 \|  \varphi \I_\g(F) \|_{X_T^{s,\frac12+ \eps}}  \le \|  \varphi \I_\g(G) \|_{X^{s,\frac12+\eps}} = \Big\|  \jb n^s \big\| \jb \ld ^b \widetilde{ \varphi \I_\g(G)}(n,\ld) \big\|_{L^2_\ld} \Big\|_{\l^2_n}
\label{D3}
\end{align}

\noi
Thus, \eqref{D00} reduces to estimate $\big\| \jb \ld ^b \widetilde{ \varphi \I_\g(F)}(n,\ld) \big\|_{L^2_\ld} $ with $n \in \Z ^2$ fixed. From Lemma \ref{LEM:duha2}, it suffices to obtain appropriate bounds on $K_\g(n, \ld,\mu) $. Using integration, we get

\noi
\begin{align}
| K_\g(n, \ld,\mu) |  \les \frac{1}{ \jb{  \g \jb n^2 + i \mu} } \min \Big( \frac{1}{\jb \ld} +\frac{1}{\jb{\ld-\mu}} , \frac{ \jb{  \g \jb n^2 } }{\jb \ld ^2} +\frac{\jb{ \g \jb n^2}}{\jb{\ld-\mu} ^2}    \Big)
\label{D4}
\end{align}

\noi
for any $\g \in [0,1]$. Furthermore, by the mean value theorem, we have

\noi
\begin{align}
| K_{\g_2} (n, \ld,\mu) -  K_{\g_1} (n, \ld,\mu) |  \les |\g_2 - \g_1| \jb n ^2,
\label{D5}
\end{align}

\noi
for any $(\g_1, \g_2) \in [0,1]^2$. The conditions \eqref{D4}, \eqref{D5} correspond to \eqref{L6} and \eqref{L7} in Lemma \ref{LEM:duha2}. Hence, from Lemma \ref{LEM:duha2} (iv), we deduce 

\noi
\begin{align}
\big\| \jb \ld ^{\frac12 + \eps} \widetilde{ \varphi \I_\g(G)}(n,\ld) \big\|_{L^2_\ld} & \les \big\| \jb \mu ^{-\frac12 + \dl} \widetilde{G}(n,\mu) \big\|_{L^2 _\mu} \label{D6}, \\
\big\| \jb \ld ^{\frac12 + \eps} \big( \widetilde{ \varphi  \I_{\g_2}} - \widetilde{ \varphi \I_{\g_2} } \big) (G)(n,\ld) \big\|_{L^2_\ld} & \les \min \big( 1,   |\g_2 - \g_1| ^\eps \jb n ^{2 \eps} \big)  \,  \big\|  \jb \mu ^{-\frac12 + \dl} \widetilde{G}(n,\mu) \big\|_{L^2 _\mu}. \label{D7}
\end{align}

\noi
Thus, by combining \eqref{D3} and \eqref{D6}, we get \eqref{D8} by definition of the $X_T^{s,b}$-restriction norm. Note that we deduce immediately from \eqref{D8}, the following estimate:

\noi
\begin{align}
\sup_{\g \in [0,1]} \big\| \I_\g (F_\g) \big\|_{X^{ s, \frac12 + \eps } _T} \les \| F_\g \|_{C_\g X_T ^{ s, -\frac12 + \dl }},
\label{D11}
\end{align}

\noi
for any $F \in C\big( [0,1]; X^{ s, \frac12 + \eps } _T \big)$. 

In order to finish the proof of \eqref{D00}, it suffices to show, by \eqref{D11}, the continuity of the map $\I : \g \mapsto \I_\g$ \eqref{duha100} from $C_\g X_T ^{ s, -\frac12 + \dl }$ to $C_\g X_T ^{ s, \frac12 + \eps}$. More precisely, we aim to show, for fixed $\g_1 \in [0,1]$, the following convergence: 

\noi
\begin{align}
 \| \I_{\g_2} (F_{\g_2}) - \I_{\g_1} (F_{\g_1}) \|_{X^{s, \frac12 + \ta }_T } \too 0,
\label{D9}
\end{align}

\noi
as $\g_2 \to \g_1$ and for any $F_\g \in C \big( [0,1] ; X^{s, -\frac12 + \dl}_T \big)$. By linearity and \eqref{D11}, \eqref{D9} follows from the convergence

\noi
\begin{align}
\big\| ( \I_{\g_2} - \I_{\g_1})  (F_{\g_1})  \big\|_{X_T^{s, \frac12 + \eps }} \too 0,
\label{D10}
\end{align}

\noi
as $\g_2 \to \g_1$. Consider as in the above, the cutoff function $\varphi$ and an extension $G \in X^{s, \frac12 + \eps }$ of $F_{\g_1}$ which agrees with $F_{\g_1}$ on $[0,T]$. Then, by \eqref{D7}, we estimate

\noi
\begin{align*}
\big \| \varphi \big( \I_{\g_2}  - \I_{\g_1} \big) (G)  \big\|_{X^{s, \frac12 + \eps }} & =  \Big\|  \jb n^s \big\| \jb \ld ^{\frac12 + \eps} \big( \widetilde{ \varphi  \I_{\g_2}} - \widetilde{ \varphi \I_{\g_2} } \big) (G)(n,\ld) \big\|_{L^2_\ld} \Big\|_{\l^2_n} \\
& \les \Big\|  \jb n^s  \min \big( 1,   |\g_2 - \g_1| ^\eps \jb n ^{2 \eps} \big)  \,   \|  \jb \mu ^{-\frac12 + \dl} \widetilde{G}(n,\mu) \|_{L^2 _\mu} \Big\|_{\l^2_n}  \too 0 ,
\end{align*}

\noi
as $\g_2 \to \g_1$, by the dominated convergence theorem in $\l^2_n$. This shows \eqref{D10} by the definition of the $X^{s,b}_T$-restriction norm and concludes the proof.
\end{proof}

\subsection{Resonant estimates}\label{SEC:3_3}

The aim of this subsection is to prove appropriate deterministic estimates to handle the resonant nonlinearity $\Rr$ \eqref{N4}.

\noi
\begin{lemma}\label{LEM:res}
Let $s \ge 0$ and $T>0$. We have the following estimate:

\noi
\begin{align*}
\big\| \mathcal{R} (u_1, u_2, u_3) \big\|_{C_{\g,T} H^s} \les \min_{\s \in S_3} \|u_{\s(1)}\|_{C_{\g,T} H^{-s}}  \|u_{\s(2)}\|_{C_{\g,T} \FL^{s, \infty}} \|u_{\s(3)}\|_{C_{\g,T} \FL^{s, \infty}}.
\end{align*}

\noi
Here, $S_3$ is the group of permutations on $\{1,2,3\}$.
\end{lemma}

\noi
\begin{proof} Fix $s \ge 0$. By H\"older's inequality, we have

\noi
\begin{align*}
\| \Rr(u_1, u_2, u_3 ) \|_{C_{\g,T} H^s} & =  \Big\|  \big\| \jb n ^s \ft{u_1}(\g, t, n) \cj{ \ft{u_2}(\g, t, n)  } \ft{u_3}(\g, t, n) \big\|_{\l^2_n} \Big\|_{L^{\infty}_{\g,T}} \\
& \le \min_{\s \in S_3}  \big\| \jb n ^{-s} \ft{u_{\s(1)}}(\g, t, n)  \big\|_{L^{\infty}_{\g,T} \l^2_n} \\
& \qquad \quad \times \| \jb n ^{s} \ft{u_{\s(2)}}(\g, t, n) \|_{L^{\infty}_{\g,T} \l^\infty_n} \| \jb n ^s \ft{u_{\s(3)}}(\g, t, n) \|_{L^{\infty}_{\g,T} \l^\infty_n}.
\end{align*}

\noi
This concludes the proof.

\end{proof}

\noi
\section{Counting estimates}\label{SEC:4}

We record here several counting estimates. The following lemma is from \cite{DNY2} but similar counting arguments are already discussed in the work of Bourgain \cite{BO96}.

Let $\{ (x,\iota_1), (y,\iota_2), (z, \iota_3 ) \}$ be a set consisting of integers $(x,y,z) \in (\Z^2)^3$ together with some signs $(\iota_1, \iota_2, \iota_3) \in \{ \pm1 \}^3$. We say that $(x,y)$ is a {\it pairing} if $x = y$ and $\iota_1 = - \iota_2$, and similarly for $(y,z)$, etc. With an abuse of notations, given an affine integral equation of the form 

\noi
\begin{align*}
\iota_1 x + \iota_2 y + \iota_3 z = d,
\end{align*}

\noi
we say that, for instance, that $(x,y)$ is are paired if $x = y$ and $\iota_1 = - \iota_2$.

\noi
\begin{lemma}\label{LEM:counting1}
Given dyadic numbers $N_1 \ges N_2 \ges N_3$, let $(\iota_1, \iota_2, \iota_3) \in \{ \pm1 \}^3$ be signs and consider the set $S$ given by

\noi
\begin{align*}
S = \{ & (x,y,z) \in (\Z^2)^3: \iota_1 x + \iota_2 y + \iota_3 z = d , \, \iota_1 \jb x^2 + \iota_2 \jb y^2 + \iota_3 \jb z ^2 = \al, \\
& \qquad \quad |x-a| \les N_1, |y-b| \les N_2, |z - c | \les N_3 \}.
\end{align*}

\noi
We assume that there is no pairing in $S$. Then, the following bound holds:

\noi
\begin{align*}
|S| \les N_2^{1 + \ta} N_3,
\end{align*}

\noi
for any $\ta >0$ and uniformly in $(a,b,c,d,\al) \in (\Z^2)^5$.
\end{lemma}

\noi
\begin{proof}
See \cite[Lemma 4.3]{DNY2}
\end{proof}

Let us now define the following phase function:

\noi
\begin{align}
\kk (\bar n) & := \jb{n}^2 - \jb{n_1}^2 + \jb{n_2}^2 - \jb{n_3}^2,
\label{dphase}
\end{align}

\noi
with the vectorial notation $\bar n = (n,n_1, n_2, n_3) \in \Z^4$. 

\noi
We use in the remainder of this section the notations of Appendix \ref{SEC:B}. Given a tensor $h = h_{n n_1 n_2 n_3}$, we define the norm 

\noi
\begin{align}
\| h \|_{1} := \max\big( \| h \|_{n \to n_1 n_2 n_3}, \| h \|_{n_1 \to n n_2 n_3}, \| h \|_{n n_2 \to n_1 n_3}, \| h \|_{n n_3 \to n_1 n_2} \big).
\label{tnorm1}
\end{align}

\noi
and

\noi
\begin{align}
\| h \|_{2} := \max\big( \| h \|_{n \to n_1 n_2 n_3}, \| h \|_{n n_1 \to n_2 n_3} \big).
\label{tnorm2}
\end{align}

The following tensor estimates will be useful when handling the random matrix terms in Lemma \ref{LEM:RMT}. In dealing with the different stochastic objects we will need counting estimates that take into account the scenario when the phase $\kk( \bar n)$ \eqref{dphase} is fixed and {\it dispersionless estimates} (i.e. when there is no condition on $\kk (\bar n)$).

\noi
\begin{lemma}[Tensor bounds $\1$] \label{LEM:t1}
Fix $(n_{\star, 1}, n_{\star, 3}, n_{\star, 3}) \in \Z^3$ and $(N_1, N_2, N_3) \in \N^3$. Let $h = h = h_{n n_1 n_2 n_3}$ be the tensor given by

\noi
\begin{align}
h_{n n_1 n_2 n_3} = \ind_{ \substack{ n = n_1 - n_2 + n_3 \\n_2 \neq n_1, n_3 }    } \prod_{j = 1}^3 \ind_{\jb{n_j  - n_{\star,j} } \sim N_j \ } 
\label{t1}
\end{align}

\noi
We also define, for $m \in \Z$, the tensor $h^m = h \ind_{\kk( \bar n) = m}$ with $\kk = \kk (\bar n)$ as in \eqref{dphase}.

\noi
\textup{(i)} The following bound holds:

\begin{align*}
\| h \|_1 \les  N_{\max} N_{\med},
\end{align*}

\noi
uniformly in $(n_{\star, 1}, n_{\star, 3}, n_{\star, 3}) \in \Z^3$.

\noi
\textup{(ii)} The following bound holds:

\begin{align*} 
\sup_{m \in \Z} \| h^m \|_1 \les N_{\max}^{\frac12 + \ta} N_{\med}^{\frac12} 
\end{align*}

\noi
for any $\ta >0$, and uniformly in $(n_{\star, 1}, n_{\star, 3}, n_{\star, 3}) \in \Z^3$.

\noi
\textup{(iii)} If $N_1 \ges \min(N_2,N_3)$, then we have the following improvements:

\noi
\begin{align*}
\| h \|_1 & \les  N_{\max} \min(N_2,N_3), \\
\sup_{m \in \Z} \| h^m \|_1 & \les  N_{\max}^{\frac12 + \ta} \min(N_2,N_3)^{\frac12},
\end{align*}

\noi
for any $\ta >0$, and uniformly in $(n_{\star, 1}, n_{\star, 3}, n_{\star, 3}) \in \Z^3$.
\end{lemma}

\noi
\begin{proof} The first bound (i) is a direct consequence of Schur's test. Let us now focus on (ii). By Schur's test and Lemma \ref{LEM:counting1}, we have

\noi
\begin{align}
\|h ^m\|_{n \to n_1 n_2 n_3}^2 & \les \sup_{n_1, n_2, n_3} | \{ n : n = n_1 - n_2 + n_3;  \, n_2 \neq n_1, n_3\}| \label{t4} \\
& \quad \times \sup_{n} | \{ (n_1, n_2, n_3) : n = n_1 - n_2 + n_3; \, n_2 \neq n_1, n_3 \text{ and } \kk(\bar n) = m\}| \notag \\
& \les N_{\med}^{1 + \ta} N_{\min}, \label{t5}
\end{align}

\noi
for any $\ta >0$. Note that we have used $\eqref{t4} \les 1$ since $n$ is uniquely determined as long as $(n_1, n_2, n_3)$ are fixed. 

Similarly and since the only bound available in general for $n$ is $\jb n \les N_{\max}$, we have

\noi
\begin{align}
\|h ^m\|_{n_1 \to n n_2 n_3}^2 \les N_{\max}^{1 + \ta} N_{\med}. \label{t6}
\end{align}

Applying again Schur's test, we bound

\noi
\begin{align}
\|h ^m\|_{n n_2 \to n_1 n_3}^2 & \les \sup_{n, n_2} | \{ (n_1, n_3) : n = n_1 - n_2 + n_3; \, n_2 \neq n_1, n_3 \text{ and } \kk(\bar n) = m \}|  \label{t7}  \\
& \quad \times\sup_{n_1, n_3} | \{ (n, n_2) : n = n_1 - n_2 + n_3; \, n_2 \neq n_1, n_3 \text{ and } \kk(\bar n) = m \}| \label{t8}.
\end{align}

\noi
Note that for $n = n_1 - n_2 + n_3$, we have

\noi
\begin{align}
\begin{split}
\kk(\bar n)  & =   |n|^2 - |n_1|^2 + |n_2|^2 - |n_3|^2  = 2 \langle n_2 - n_1, n_2 - n_3 \rangle \\
& = 2 \langle n - n_3, n_2 - n_3 \rangle = 2 \langle n - n_3, n - n_1 \rangle,
\end{split}
\label{t9}
\end{align}

\noi
where $\langle \cdot, \cdot \rangle$ is the usual inner product on $\R^2$ . This leads to the following formulas:

\noi
\begin{align}
\begin{split}
\kk (\bar n ) & = -2 \Big| n_1 - \frac{n + n_2}{2} \Big|^2  + 2 \Big| \frac{n - n_2}{2} \Big|^2 \\
& = -2 \Big| n_2 - \frac{n _1+ n_3}{2} \Big|^2  + 2 \Big| \frac{n_1 - n_3}{2} \Big|^2.
\end{split}
\label{t10}
\end{align}

\noi
Hence, if $n$ and $n_2$ are fixed then, from \eqref{t10}, $n_1$ belongs to a circle of radius at most $\sim N_{\max}$, which leads to the bound $\eqref{t7} \les N_{\max}^\ta$, for any $\ta >0$. Similarly, $\eqref{t8} \les N_{\max}^{\ta}$, for any $\ta >0$. This yields

\noi
\begin{align}
\|h ^m\|_{n n_2 \to n_1 n_3}^2 \les N_{\max}^{\ta},
\label{t10b}
\end{align}

\noi
for any $\ta >0$.

At last, we estimate

\noi
\begin{align}
\|h ^m\|_{n n_3 \to n_1 n_2}^2 & \les \sup_{n, n_3} | \{ (n_1, n_2) : n = n_1 - n_2 + n_3; \, n_2 \neq n_1, n_3 \text{ and } \kk(\bar n) = m \}| \label{t11}   \\
& \quad \times\sup_{n_1, n_2} | \{ (n, n_3) : n = n_1 - n_2 + n_3; \, n_2 \neq n_1, n_3 \text{ and } \kk(\bar n) = m \}| \label{t12}
\end{align}

\noi
By \eqref{t9}, we deduce that if $n$ and $n_3$ are fixed, then $n_1$ and $n_2$ both belongs to lines. Thus, $\eqref{t11} \les  \min(N_1,N_2) $. Similarly, $\eqref{t12} \les N_3$. This proves

\noi
\begin{align}
\|h ^m\|_{n n_3 \to n_1 n_2}^2 \les N_{\max} N_{\med}. 
\label{t13}
\end{align}

\noi
The estimate (ii) follows from \eqref{t5}, \eqref{t6}, \eqref{t10b}, and \eqref{t13}. The bounds (iii) follow from similar arguments.
\end{proof}

\noi
\begin{lemma}[Tensor bounds $\II$] \label{LEM:t2}
Fix $(n_{\star, 1}, n_{\star, 3}, n_{\star, 3}) \in \Z^3$ and $(N_1, N_2, N_3) \in \N^3$. Let $h = h = h_{n n_1 n_2 n_3}$ be the tensor given by

\noi
\begin{align}
h_{n n_1 n_2 n_3} = \ind_{ \substack{ n = n_1 - n_2 + n_3 \\n_2 \neq n_1, n_3 }    } \prod_{j = 1}^3 \ind_{\jb{n_j  - n_{\star,j} } \sim N_j \ } 
\label{t1}
\end{align}

\noi
We also define, for $m \in \Z$, the tensor $h^m = h \ind_{\kk( \bar n) = m}$ with $\kk = \kk (\bar n)$ as in \eqref{dphase}.

\noi
\textup{(i)} The following bound holds:

\begin{align*}
\| h \|_2 \les  N_{\max} N_{\min},
\end{align*}

\noi
uniformly in $(n_{\star, 1}, n_{\star, 3}, n_{\star, 3}) \in \Z^3$.

\noi
\textup{(ii)} The following bound holds:

\begin{align*} 
\sup_{m \in \Z} \| h^m \|_2 \les N_{\max}^{\frac12 + \ta} N_{\min}^{\frac12} 
\end{align*}

\noi
for any $\ta >0$, and uniformly in $(n_{\star, 1}, n_{\star, 3}, n_{\star, 3}) \in \Z^3$.
\end{lemma}

\noi
\begin{proof} The proof of (i) follows from Schur's test. Let us now prove (ii). By Schur's test, we have

\noi
\begin{align}
\|h ^m\|_{n \to n_1 n_2 n_3}^2 \les N_{\med}^{1 + \ta} N_{\min}.
\label{tb1}
\end{align}

\noi
for any $\ta >0$ as in the proof of Lemma \ref{LEM:t1}. Similarly, in view of \eqref{t11} and \eqref{t12} in the proof of Lemma \ref{LEM:t1}, we have

\noi
\begin{align}
 \| h^m \|_{n n_1 \to n_2 n_3}^2 & \les \sup_{n, n_1} | \{ (n_2, n_3) : n = n_1 - n_2 + n_3; \, n_2 \neq n_1, n_3 \text{ and } \kk(\bar n) = m \}| \notag   \\
& \quad \times\sup_{n_2, n_3} | \{ (n, n_1) : n = n_1 - n_2 + n_3; \, n_2 \neq n_1, n_3 \text{ and } \kk(\bar n) = m \}| \notag \\
& \les \min(N_2, N_3) \times N_1. \label{tb2}
\end{align}
\end{proof}

\noi
\section{Regularities of the stochastic terms}\label{SEC:6}
\subsection{Basic stochastic terms}\label{SEC:6_1}

The purpose of this section is to construct the stochastic objects $\Psi_\g$, $\<1>_\g$ and $\<30>_\g$ in appropriate spaces.

\noi
\begin{lemma}\label{LEM:sto1}
Let $T >0$, $b< \frac12$, and $\eps >0$.

\noi
\textup{(i)} The sequence $\{ \g, \mapsto \Psi_{\g,N}   \}_{N \in \N}$ defined in \eqref{conv2} is a Cauchy sequence in $C \big( [0,1] \times [0,T] ; W^{-\eps,\infty}(\T^2) \cap \FL^{1- \eps, \infty} (\T^2)  \big) \cap C\big( [0,1]; X^{-\eps, b} ([0,T]) \big)$. In particular, denoting the limit by $\{ \g, \mapsto \<1>_\g \}$ \textup{(}formally given by \eqref{conv1}\textup{)}, we have

\noi
\begin{align*}
\g \mapsto \Psi_\g \in C \big( [0,1] \times [0,T] ; W^{-\eps,\infty}(\T^2) \cap \FL^{1- \eps, \infty} (\T^2) \big) \cap C\big( [0,1]; X^{-\eps, b} ([0,T]) \big).
\end{align*}

\noi
\textup{(ii)} The sequence $\{ \g, \mapsto \<1>_{\g,N}   \}_{N \in \N}$ defined in \eqref{sto2} is a Cauchy sequence in $C \big( [0,1] \times [0,T] ; W^{-\eps,\infty}(\T^2) \cap \FL^{1- \eps, \infty} (\T^2)  \big)$. In particular, denoting the limit by $\{ \g, \mapsto \<1>_\g \}$ \textup{(}formally given by \eqref{sto1}\textup{)}, we have

\noi
\begin{align*}
\g \mapsto \<1>_\g \in C \big( [0,1] \times [0,T] ; W^{-\eps,\infty}(\T^2) \cap \FL^{1- \eps, \infty} (\T^2) \big).
\end{align*}
\end{lemma}

\noi
\begin{proof} We only consider (i) as the proof of (ii) is similar. Let $T >0$, $b< \frac12$, and $\eps >~0$. By arguing as in \cite[Proposition 4.1]{YZ} and making use of a bi-parameter Kolmogorov continuity criterion, one shows that~$\{ \g \mapsto~\Psi_{\g,N}   \}_{N \in \N}$ is a Cauchy sequence in $C \big( [0,1] \times [0,T] ; W^{-\eps,\infty}(\T^2) \cap \FL^{1- \eps, \infty} (\T^2) \big)$. The rest of the proof is standard and we omit details. See for instance \cite{GKO1, GKOT, YZ} for similar considerations.

%
\end{proof}

\noi
\begin{lemma}\label{LEM:sto_cubic}
For any $0 < s < \frac12$, $0 < \eps \ll 1$ and $T>0$, $\{ \g \mapsto \<30>_{\g,N} \}_{N \in \N}$ defined in \eqref{sto_cubic} is a Cauchy sequence in $ C \big( [0,1 ]; X^{s, \frac12 + \eps} ([0,T])  \big) $, almost surely. In particular, denoting by $\g \mapsto \<30>_{\g}$ the limit, we have

\noi
\begin{align*}
\g \mapsto \<30>_{\g} \in C \big( [0,1 ]; X^{\frac12 - \eps, b} ([0,T])  \big)
\end{align*}

\noi
for any $\eps >0$, almost surely.
\end{lemma}

\begin{proof}
By Proposition \ref{PROP:duha1}, it suffices to show that $\{ \g \mapsto \<3>_{\g,N} \}_{N \in \N}$ is Cauchy in $ C \big( [0,1 ]; X^{s, -\frac12 + \dl} ([0,T])  \big) $, almost surely for some $\dl >0$. Let $\g \in [0,1]$ and $N \in \N$. From \eqref{ob1}, we have

\noi
\begin{align}
\ind_{[0,T]} \ft{\<3>}_{\g,N}(n,t) = I_3 [h_{n,t}].
\label{c1}
\end{align}

\noi
with

\noi
\begin{align}
\begin{split}
h_{\g, n,t}(z_1,z_2,z_3) & = \ind_{[0,T]}(t)  \ind_{\substack{ n = n_1 - n_2 + n_3 \\ n_2 \neq n_1, n_3}} \, \prod_{j=1}^3 \big( e^{-(\g + i) t \jb{n_j}^2 } \ind_{\jb {n_j} \le N}  \big)^{\iota_j}  \\
& \qquad \times \prod_{j=1}^3 \Big( \ind_{\ze_j = -1} \ind_{[0,1]}(t_j) + \ind_{\ze_j = 1} \ind_{[0,t]}(t_j) \sqrt{2 \g} e^{(\g+i) t_j \jb{n_j}^2 } \Big)^{\iota_j},
\end{split}
\label{c2}
\end{align}

\noi
with $z_j = (n_j,t_j, \ze_j)$ and $\iota_j = 1$ if $j$ is odd and $-1$ otherwise. In what follows, we will often omit the dependence of the quantities in the variables $(z_j)_{j=1, 2 ,3}$ to ease our notations. 

We first aim to prove the following bound:

\noi
\begin{align}
\sup_{N \in \N} \sup_{\g \in [0,1]} \big\| \| \ind_{[0,T]} \<3>_{\g,N} \|_{X^{\frac12,-\frac12+\dl}}  \big\|_{L^p(\O)}  \les p^{\frac32} T^\al,
\label{c100}
\end{align}

\noi
for any $p \ge 1$ and some $\al >0$. By using Lemma \ref{LEM:hyp} with \eqref{c1}, proving \eqref{c100} reduces to the bound

\noi
\noi
\begin{align}
\big\| \| \ind_{[0,T]} \<3>_{\g,N} \|_{X^{\frac12,-\frac12+\dl}}  \big\|_{L^p(\O)}  \les T^\al,
\label{c101}
\end{align}

\noi
uniformly in $(N,\g) \in \N \times  [0,1]$, and for some small $\al >0$. We now discuss the proof of \eqref{c101}. Let us note that we may write $h_{\g, n,t}$ as a sum of terms of the form

\noi
\begin{align}
\begin{split}
h^A_{\g, n,t}(n_1,t_1,n_2, t_2, n_3, t_3) & = \ind_{[0,T]}(t)  \ind_{\substack{ n = n_1 - n_2 + n_3 \\ n_2 \neq n_1, n_3}} \, \prod_{j=1}^3 \big( e^{-(\g + i) t \jb{n_j}^2  } \ind_{\jb {n_j} \le N}  \big)^{\iota_j}  \\
& \qquad \times \prod_{j \in A}  \big( \ind_{[0,t]}(t_j) \sqrt{2 \g} e^{(\g+i) t_j \jb{n_j}^2 }\big)^{\iota_j} \cdot \prod_{j \in B}   \frac{ \ind_{[0,1]}(t_j)}{ \jb{n_j} },
\end{split}
\label{c3}
\end{align}

\noi
where $A$ and $B$ form a partition of $\{1, 2, 3\}$ such that $\ze_j =1$ for $j \in A$ and $\ze_j = -1$ for $j \in B$. Then, by using Lemma \ref{LEM:B3}, we compute the twisted space-time Fourier transform \eqref{ft2} of $\eqref{c1}$

\noi
\begin{align}
\widetilde{ \ind_{[0,T]}  \<3>_{\g,N} }(n,\ld) = \sum_{A \subset \{1,2,3\}} I_3 \big[\wt h _{\g, n, \ld} \big],
\label{c102}
\end{align}

\noi
where

\noi
\begin{align}
\begin{split}
\wt h^A_{\g, n, \ld} & = \ind_{\substack{ n = n_1 - n_2 + n_3 \\ n_2 \neq n_1, n_3}} \int_{t_{\max} (A)}^T  e^{ -t ( i ( \ld - \kk(\bar n)) + \g \be_0 ( \bar n) ) } dt \\
& \qquad \times \prod_{j \in A}  \big(  \sqrt{2 \g} e^{(\g+i) t_j \jb{n_j}^2 }\big)^{\iota_j} \cdot \prod_{j \in B} \frac{ \ind_{[0,1]}(t_j)}{ \jb{n_j} }  \cdot \prod_{j=1}^3 \ind_{[0,T]}(t_j), 
\end{split}
\label{c4}
\end{align}

\noi
with

\noi
\begin{align}
\begin{split}
\kk (\bar n) & := \jb{n}^2 - \jb{n_1}^2 + \jb{n_2}^2 - \jb{n_3}^2 \\
\be_0 (\bar n) & := \jb{n_1}^2 + \jb{n_2}^2 + \jb{n_3}^2,
\end{split}
\label{c5}
\end{align}

\noi
and

\noi
\begin{align}
t_{\max} (A) = \{ t_j : j \in A \}.
\label{c5b}
\end{align}

\noi
Furthermore, we localize the variables $n_j$ to the regions $\jb{n_j} \sim N_j$ for dyadics $N_j \ge 1$ ($j = 1,2, 3$). Let $N_{\star} = (N_1, N_2, N_3)$ and denote by $\wt h^{A, N_{\star}}_{\g, n, \ld}$ the contribution of $\jb{n_j} \sim N_j$ to $\wt{h^A_{\g, n, \ld}}$. Similarly, we will denote by $\<3>_{\g,N}^{ N_{\star} }$ the contribution of $\jb{n_j} \sim N_j$ to $\<3>_{\g,N}$. With these notations, the bound \eqref{c101} follows from

\noi
\begin{align}
\big\| \ind_{[0,T]} \<3>_{\g,N}^{ N_{\star}} \big\|_{L^2(\O) X^{s, -\frac12 + \dl}} \les N_{\max} ^{-\ta} T^\al,
\label{c103}
\end{align}

\noi
for any dyadics $N_{\star} = (N_1, N_2, N_3)$, uniformly in $(N,\g) \in \N \times  [0,1]$, and for some small $\ta >0$ and $\al >0$. Fix dyadics $N_{\star} = (N_1, N_2, N_3)$. Then, by using \eqref{c102}, and Lemma \ref{LEM:B1} (iii), we have

\noi
\begin{align}
\begin{split}
\| \ind_{[0,T]} \<3>_{\g,N}^{N_{\star}}  \|_{L^2(\O) X^{s,b_1}} & = \big\| \| \jb n ^s \jb \ld ^{b_1} \widetilde{ \ind_{[0,T]}  \<3>_{\g,N} ^{N_{\star}} }  (n,\ld) \|_{\l^2_n L^2_\ld} \big\|_{L^2(\O)} \\
& \le \max_{A \subset \{1,2,3\} } \big\| \| \jb n ^s \jb \ld ^{b_1} I_3 [\wt h^{A,N_{\star}}_{\g, n, \ld}   ] \|_{\l^2_n L^2_\ld} \big\|_{L^2(\O)} \\
& \le \max_{A \subset \{1,2,3\} } \big\| \jb n ^s \jb \ld ^{b_1}  \big\| \Sym( \wt h_{\g, n, \ld} ^{ A, N_{\star}}  ) \big\|_{L^2_{z_1,z_2,z_3} }  \big\|_{\l^2_n L^2_\ld},
\end{split}
\label{c2b}
\end{align}

\noi
for any $b_1 \in \R$. Hence, from \eqref{c2b} and \eqref{c102} with \eqref{c4}, we have the crude bound

\noi
\begin{align}
\| \ind_{[0,T]} \<3>_{\g,N}^{N_{\star}}  \|_{L^2(\O) X^{0,0}} \les N_{\max}^{10} T,
\label{c104}
\end{align}

\noi
uniformly in $(N,\g) \in \N \times  [0,1]$. Thus, in order to prove \eqref{c103}, it suffices to prove 

\noi
\begin{align}
 \big\| \ind_{[0,T]} \<3>_{\g,N}^{ N_{\star}} \big\|_{L^2(\O) X^{s, -\frac12 - \dl_1}} \les N_{\max} ^{-\ta} T^\al,
\label{c105}
\end{align}

\noi
uniformly in $(N,\g) \in \N \times  [0,1]$, and for some small $\ta >0$, $\al >0$ and $\dl_1 >0$. Indeed, interpolating \eqref{c105} with \eqref{c104} yields \eqref{c103}.

Next, by H\"older's inequality (using $-\frac12 - \dl_1 < -\frac12$) and \eqref{c2b}, \eqref{c105} follows from a similar bound on $ \big\| \Sym( \wt h^{ A, N_{\star}}_{\g, n, \ld} ) \big\|_{L^2_{z_1,z_2,z_3} }$. Hence, by Jensen's inequality \eqref{Jen}, \eqref{c2}, and relabelling, we only have to show

\noi
\begin{align}
\max_{A \subset \{1,2,3\} } \big\| \wt h^{ A, N_{\star}}_{\g, n, \ld}   \big\|_{L^2_{z_1,z_2,z_3} } \les N_{\max} ^{-\ta} T^\al,
\label{c106}
\end{align}

\noi
uniformly in $(N,\g) \in \N \times  [0,1]$, and for some small $\ta >0$ and $\al >0$. We have 

\noi
\begin{align}
\begin{split}
& \int_{t_{\max} (A)}^T  e^{ -t ( i ( \ld - \kk(\bar n)) + \g \be ( \bar n) ) } dt  \cdot \prod_{j \in A} \big(  \sqrt{2 \g} e^{(\g+i) t_j \jb{n_j}^2 }\big)^{\iota_j}  \\
& = (2 \g) ^{ \frac{| A  |}{2}} e^{ - t_{\max} (A) ( i ( \ld - \kk(\bar n)) + \g \be ( \bar n) ) } \, \frac{1 - e^{ -(T- t_{\max} (A)) ( i ( \ld - \kk(\bar n)) + \g \be ( \bar n) ) } }{i ( \ld - \kk(\bar n)) + \g \be ( \bar n) } \\
& \qquad  \times \prod_{j \in P} \big(  e^{(\g+i) t_j \jb{n_j}^2 }\big)^{\iota_j}
\end{split}
\label{c6}
\end{align}

\noi
We note that by the definition of $t_{\max} (A)$ in \eqref{c5b}, the following bound holds:

\noi
\begin{align}
\begin{split}
& \Big| e^{ - t_{\max} (A) ( i ( \ld - \kk(\bar n)) + \g \be ( \bar n) ) } \cdot \prod_{j \in A} \big(  e^{(\g+i) t_j \jb{n_j}^2 }\big)^{\iota_j} \Big| \les \prod_{j \in A} e^{-\g(t_{\max} (A) - t_j) \jb{n_j}^2}
\end{split}
\label{c7}
\end{align}

\noi
Further, by the mean value theorem, we have

\noi
\begin{align}
\begin{split}
\Big| \frac{1 - e^{ -(T- t_{\max} (A)) ( i ( \ld - \kk(\bar n)) + \g \be ( \bar n) ) } }{i ( \ld - \kk(\bar n)) + \g \be ( \bar n) }  \Big| & \les  \frac{T}{\jb{ i ( \ld - \kk(\bar n)) + \g \be ( \bar n) }} \\
& \les \frac{T}{ \jb{  \ld - \kk(\bar n))}^{\frac12}  \cdot \g^\frac12 \jb{n_{\max}(A)} },
\end{split}
\label{c7b}
\end{align}

\noi
with $n_{\max}(A)$ is the maximum over the set $\{ n_j : j \in A \}$. Similarly, we define $n_{\med}(A)$ and $n_{\min}(A)$ as the second largest and smallest elements in the set $\{ n_j : j \in A \}$, respectively. Hence, from \eqref{c4}, \eqref{c7} and \eqref{c7b}, we get for any $A \subset \{1,2,3\}$,

\noi
\begin{align}
\begin{split}
\big\| \wt h^{A, N_{\star} }_{\g,n, \ld}\big\|_{L^2_{t_1, t_2, t_3} }^2 & \les T^\al \ind_{\substack{ n = n_1 - n_2 + n_3 \\ n_2 \neq n_1, n_3}}     \prod_{j=1}^3 \ind_{ \jb {n_j} \sim N_j  }   \\
&  \quad \times \prod_{j \in B}  \frac{1}{ \jb{n_j}^2 } \cdot \frac{\g ^{ | A  |}}{ \jb{  \ld - \kk(\bar n))} \cdot \g \jb{n_{\max}(A)} ^2}  \\
& \quad \times \big\| \prod_{j \in A} e^{-\g(t_{\max} (A) - t_j) \jb{n_j}^2} \ind_{[0,T]}(t_j) \big\|^2 _{L^2_{ (t_j : j \in A )} }  \\
& \les  T^\al \ind_{\substack{ n = n_1 - n_2 + n_3 \\ n_2 \neq n_1, n_3}}  \prod_{j=1}^3 \ind_{ \jb {n_j} \sim N_j  } \prod_{j \in B}  \frac{1}{ \jb{n_j}^2 }  \\
& \quad \times \frac{\g ^{ | A  |}}{ \jb{  \ld - \kk(\bar n))} \cdot \g \jb{n_{\max}(A)} ^2} \cdot \frac{1}{ \g^{|A|- 1} \jb{ n_{\med} }^2 \jb{n_{\min}}^2 }  \\
& \les T^\al \ind_{\substack{ n = n_1 - n_2 + n_3 \\ n_2 \neq n_1, n_3}}  \prod_{j=1}^3 \frac{ \ind_{ \jb {n_j} \sim N_j  } }{\jb{n_j}^2 } \cdot \frac{1}{\jb{  \ld - \kk(\bar n))}},
\end{split}
\label{c8}
\end{align}

\noi
where $\al >0$ is a number that may change from line to line. From \eqref{c8}, we then deduce the following bound:

\noi
\begin{align}
\max_{A \subset \{1,2,3\} }  \big\| \wt h^{A, N_{\star}}_{\g,n, \ld} \big\|_{L^2_{z_1,z_2,z_3} }^2 & \les T^\al  \sum_{ \substack{ n = n_1 - n_2 + n_3 \\ n_2 \neq n_1, n_3 \\ \ \jb{n_j} \sim N_j \le N} } \frac{1}{\jb{  \ld - \kk(\bar n) }  \jb{n_1}^2  \jb{n_2}^2  \jb{n_3}^2 } \label{c9} \\
 &  \les T^\al \sum_{n \in \Z^2} \jb n ^{2s} \sum_{ \substack{ n = n_1 - n_2 + n_3 \\ n_2 \neq n_1, n_3 \\ \ \jb{n_j} \sim N_j \le N} } \frac{1}{ \jb{n_1}^2  \jb{n_2}^2  \jb{n_3}^2} \int_\R  \frac{d \ld}{ \jb{\ld}^{1 + 2 \dl} \jb{  \ld - \kk(\bar n) }}  \notag \\
& \les T^\al  \sum_{n \in \Z^2} \jb n ^{2s} \sum_{ \substack{ n = n_1 - n_2 + n_3 \\ n_2 \neq n_1, n_3 \\ \ \jb{n_j} \sim N_j \le N} } \frac{1}{ \jb{n_1}^2  \jb{n_2}^2  \jb{n_3}^2 \kk( \bar n )^{1 + 2 \dl}  }  \notag \\
& \les T^\al \sum_{ \kk \in \Z } \frac{1}{\jb{\kk}^{1 + 2 \dl}} \sum_{n \in \Z^2} \jb n ^{2s} \sum_{ \substack{ n = n_1 - n_2 + n_3 \\ n_2 \neq n_1, n_3 \\ \ \jb{n_j} \sim N_j \le N \\ \kk( \bar n ) = \kk } } \frac{1}{ \jb{n_1}^2  \jb{n_2}^2  \jb{n_3}^2 \  }  \notag\\
& \les T^\al \sup_{\kk \in \Z}  \sum_{n \in \Z^2} \jb n ^{2s} \sum_{ \substack{ n = n_1 - n_2 + n_3 \\ n_2 \neq n_1, n_3 \\ \ \jb{n_j} \sim N_j \le N \\ \kk( \bar n ) = \kk } } \frac{1}{ \jb{n_1}^2  \jb{n_2}^2  \jb{n_3}^2 \  } \label{c9a},
\end{align}

\noi
for some $\al >0$. Let us assume $N_1 \ge N_2 \ge N_3 $. The proof is similar in other cases. We then have from Lemma \ref{LEM:counting1}

\noi
\begin{align}
\eqref{c9a}&  \les T^\al \sup_{\kk \in \Z^2}  N^{2s} (N_1 N_2 N_3)^{-2} N_3^2   \sum_{  \substack{ n_3 = n - n_1 + n_2 \\ n_2 \neq n_1, n_3 \\ \ \jb{n_j} \sim N_j \le N \\ \kk (\bar n ) = \kk }  }  1 \notag \\
& \les T^\al N^{2s} (N_1 N_2 )^{-2}  (N \wedge N_2 )^{1+ \eps}   (N \vee N_2 ) \les T^\al N_{\max}^{2s - 1 + \eps}, \label{c10}
\end{align}

\noi
for any $\eps > 0$. This proves \eqref{c106} with $s < \frac12$. 

Fix $M \ge N$ two integers. We now aim to prove the following bound:

\noi
\begin{align}
\sup_{\g \in [0,1]} \big\| \| \ind_{[0,T]} ( \<3>_{\g,M} - \<3>_{\g,N}) \|_{X^{s,-\frac12+\dl}}  \big\|_{L^p(\O)}  \les p^{\frac32}  N^{-\ta} T^\al,
\label{c107}
\end{align}

\noi
for all $p \ge 1$, some small $\ta >0$ and $\al >0$. Note that $\<3>_{\g,M} - \<3>_{\g,N}$ is essentially as in \eqref{c1} and \eqref{c2} with the additional condition $\max ( \jb{n_1}, \jb{n_2}, \jb{n_3} ) \ge N$. Hence, by Lemma \ref{LEM:hyp} and by arguing as before, it suffices to prove the estimate

\noi
\begin{align}
\sup_{\g \in [0,1]} \big\| \| \ind_{[0,T]} ( \<3>^{N_{\star}}_{\g,M} - \<3>^{N_{\star}}_{\g,N}) \|_{X^{s,-\frac12+\dl}}  \big\|_{L^2(\O)}  \les N_{\max}^{-\ta} T^\al,
\label{c108}
\end{align}

\noi
for some small $\ta>0$ and $\al>0$, and for any dyadic numbers $N_{\star} = (N_1, N_2, N_3)$ (as, necessarily $N_{\max} \ges N$ so that we can gain a small power of $N$ from the right-hand-side of \eqref{c108}). The proof of bound \eqref{c108} follows as that of \eqref{c103}.

Lastly, we aim at proving the next bound:

\noi
\begin{align}
\big\| \| \ind_{[0,T]} \big( ( \<3>_{\g_2,M} - \<3>_{\g_2,N}) - ( \<3>_{\g_1,M} - \<3>_{\g_1,N}) \big) \|_{X^{s,-\frac12+\dl}}  \big\|_{L^p(\O)}  \les p^{\frac32}  N^{-\ta} (\g_2 - \g_1)^\ta T^\al
\label{c109}
\end{align}

\noi
for any integers $M \ge N$, $(\g_1, \g_2) \in [0,1]^2$, and some small $\ta >0$ and $\al >0$. By interpolation with \eqref{c108}, \eqref{c109} follows from 

\noi
\begin{align}
\sup_{N \in \N} \big\| \| \ind_{[0,T]} ( \<3>_{\g_2,N} - \<3>_{\g_1,N}) \|_{X^{s,-\frac12+\dl}}  \big\|_{L^p(\O)}  \les p^{\frac32}  (\g_2 - \g_1)^\ta T^\al
\label{c109}
\end{align}

\noi
for $(\g_1, \g_2) \in [0,1]^2$, and some small $\ta >0$ and $\al >0$. Again, by Lemma \ref{LEM:hyp} and dyadic localization \eqref{c109} follows from

\noi
\begin{align}
\sup_{N \in \N} \big\| \| \ind_{[0,T]} ( \<3>_{\g_2,N}^{N_{\star}} - \<3>^{N_{\star}}_{\g_1,N}) \|_{X^{s,-\frac12+\dl}}  \big\|_{L^2(\O)}  \les N_{\max}^{-\ta}  (\g_2 - \g_1)^\ta T^\al
\label{c109}
\end{align}

\noi
for any integers $M \ge N$, $(\g_1, \g_2) \in [0,1]^2$, dyadic numbers $N_{\star} = (N_1, N_2, N_3)$, and some small $\ta >0$ and $\al >0$.

We now prove \eqref{c109}. Let $(\g_1,\g_2) \in~(0,1]^2$ and $\g \in [0,1]$. From \eqref{c2} and the mean value theorem, we have

\noi
\begin{align}
|h_{\g_2, n,t} - h_{\g_1, n, t} | & \les T N_{\max}^{2}   \frac{ |\g_2 - \g_1| }{ ( \g_1 \vee \g_2 )^\frac12 } \label{c12a} \\
|h_{\g, n,t} - h_{0, n, t} | & \les T N_{\max}^{2}  \g^\frac12. \label{c12b}
\end{align}

\noi
Hence, interpolating \eqref{c12a} and \eqref{c12b} yields

\noi
\begin{align}
|h_{\g_2, n,t} - h_{\g_1, n, t} | & \les T N_{\max}^{2}  |\g_2 - \g_1|^\eps, \label{c12c}
\end{align}

\noi
for any $(\g_1,\g_2) \in [0,1]^2$ and for some $\eps > 0$. Making again use of Lemma \ref{LEM:B1} (iii) with \eqref{c12c}

\noi
\begin{align}
\| \ind_{[0,T]} ( \<3>^{ N_{\star} }_{\g_2,N} - \<3>^{N_{\star}}_{\g_1,N})  \|_{L^2(\O) X^{s,0}} & = \|  \<3>^{N_{\star}}_{\g_2,N} - \<3>^{N_{\star}}_{\g_1,N} \|_{L^2(\O) L^2_{T} H^s _x}^2 \notag \\ 
& \les \| \jb n ^s \| I_3 \big[ h_{\g_2,n,t} - h_{\g_1,n,t} \big] \|_{L^2(\O)}  \|_{L^2_T \l^2_n} \notag \\
& \les T^\al N_{\max}^{10} |\g_2 - \g_1|^{\ta}.
\label{c13}
\end{align}

\noi
for some small $\ta >0$ and $\al >0$. Hence, interpolating \eqref{c13} and \eqref{c103} gives \eqref{c109}. 

The convergence of $\{ \g \mapsto \<3>_{\g,N} \}_{N \in \N}$ in $ C \big( [0,1 ]; X^{s, -\frac12 + \dl} ([0,T])  \big) $ then follows from the bounds \eqref{c100} and \eqref{c109} together with the Kolmogorov continuity criterion. (In order to obtain the continuity in $\g \in [0,1]$.) See for instance \cite{OPTz} for details.

\end{proof}
\subsection{Linear random operators}\label{SEC:6_2} In this section, we deal with the linear random operators\footnote{Such operators are also called random matrices in the literature on random dispersive equations; see for instance \cite{BO96, DNY1, DNY2, OWZ, BDNY}} $\MM^1_{\g}$ and $\MM^2_{\g}$ defined in \eqref{rmt}. In particular, we prove the following proposition.

\begin{proposition}\label{PROP:RMT}
Let $0 < s < \frac14$, $0 < \eps, \ta \ll 1$, and $0 < T_0 \le 1$. Then, the sequences $ \{  \g \mapsto \MM^1_{\g,N}\}_{N \in \N}$ and $\{ \g \mapsto \MM^{2}_{\g,N}\}_{N \in \N}$ defined in \eqref{rmt} are Cauchy sequences in the class $\L_{T_0}^{s,s,\frac12+ \eps, \frac12+ \ta}$, almost surely. In particular, denoting the respective limits by $\g \mapsto \MM_\g ^{1}$ and $\g \mapsto \MM_\g^{2}$, we have

\noi
\begin{align*}
(\g \mapsto \MM_\g ^{1}, \g \mapsto \MM_\g ^{2}) \in \big( \L_{T_0}^{s,s,\frac12+ \eps, \frac12+ \ta} \big)^2.
\end{align*}
\end{proposition}

We now introduce the following trilinear forms:

\noi
\begin{align}
\begin{split}
\NN^{1,>}(u,v,w) & = \sum_{N_2, N_3} \NN \big( \P_{\gg N_2 \vee N_3}u, \P_{N_2} v, \P_{N_3}w \big) \\
\NN^{2, >}(u,v,w) & = \sum_{N_1, N_3}\NN \big( \P_{N_1} u , \P_{\gg N_1 \vee N_3} v, \P_{N_3} w \big),
\end{split}
\label{nonlin}
\end{align}

\noi
In \eqref{nonlin}, the summations are over dyadic numbers $N_1, N_2, N_3 \ge 1$. Similarly, we denote by $\NN^{j,<}$, the trilinear form

\noi
\begin{align}
\NN^{j,<}(u,v,w) := \NN^j(u,v,w) - \NN^{j, >}(u,v,w),
\label{nonlin2}
\end{align}

\noi
for $1 \le j \le 2$. We then decompose the random operators as follows

\noi
\begin{align}
\begin{split}
\mathfrak{M}^1_{\g,N} & =: \MM^{1,>}_{\g,N} + \MM^{1,<}_{\g,N} \\
\mathfrak{M}^2_{\g,N} & =: \MM^{2,>}_{\g,N} + \MM^{2,<}_{\g,N},
\end{split}
\label{rop}
\end{align}

\noi
with

\noi
\begin{align}
\begin{split}
\MM^{1,\dagger}_{\g,N}(v) & = \I_\g \NN^{1,\dagger}\big(v, \<1>_{\g,N}, \<1>_{\g,N} \big)  \\
\MM^{2,\dagger}_{\g,N}(v) & = \I_\g \NN^{2, \dagger} \big( \<1>_{\g,N}, v,  \<1>_{\g,N} \big),
\end{split}
\label{rop2}
\end{align}

\noi
for $\dagger \in \{ <, > \}$.
 
Proposition \ref{PROP:RMT} is a direct consequence of the following two lemmas.

\noi
\begin{lemma}\label{LEM:RMT}
Let $0< s < \frac12$, $0 < \eps, \ta \ll 1$, and $0 < T_0 < 1$. Then, the sequences $ \{  \MM^{1, >}_{\g,N}\}_{N \in \N}$ and $\{ \MM^{2,>}_{\g,N}\}_{N \in \N}$ defined in \eqref{rmt} are Cauchy sequences in the class $\L_{T_0}^{s,s,\frac12+ \ta, \frac12+ \eps}$, almost surely. In particular, denoting the respective limits by $\g \mapsto \MM_\g ^{1,>}$ and $\g \mapsto \MM_\g^{2,>}$, we have

\noi
\begin{align*}
( \MM_\g ^{1,>},  \MM_\g ^{2,>}) \in \big( \L_{T_0}^{s,s,\frac12+ \ta, \frac12+ \eps} \big)^2.
\end{align*}
\end{lemma}

\noi
\begin{lemma}\label{LEM:RMT2}
Let $0 <s < \frac14$, $0 < \eps, \ta \ll 1$, and $0 < T_0 < 1$. Then, the sequences $ \{ \MM^{1, <}_{\g,N}\}_{N \in \N}$ and $\{ \MM^{2,<}_{\g,N}\}_{N \in \N}$ defined in \eqref{rmt} are Cauchy sequences in the class $\L_{T_0}^{s,s,\frac12+ \ta, \frac12+ \eps}$, almost surely. In particular, denoting the respective limits by $\g \mapsto \MM_\g ^{1,<}$ and $\g \mapsto \MM_\g^{2}$, we have

\noi
\begin{align*}
( \MM_\g ^{1,>},  \MM_\g ^{2,>}) \in \big( \L_{T_0}^{s,s,\frac12+ \ta, \frac12+ \eps} \big)^2.
\end{align*}
\end{lemma}

We first start with the proof of Lemma \ref{LEM:RMT}. 

\noi
\begin{proof}[Proof of Lemma \ref{LEM:RMT}] We focus on the operator $\{ \I_\g \NN^{1,>}(\cdot, \<1>_{\g,N}, \<1>_{\g,N} ) \}_{N \in \N}$ since the case $j=2$ is similar; namely, we fix $j=1$. By Proposition \ref{PROP:duha1}, \eqref{rop}, and \eqref{rop2} it suffices to prove that $\{  \NN^{1,>}(\cdot,\<1>_{\g,N},  \<1>_{\g,N} ) \}_{N \in \N}$ is a Cauchy sequence in $\L_{T_0}^{s,s, - \frac12+\dl, \frac12+\eps}$ (without the $T^\ta$ factor gain in the $\L_{T_0}^{s,s, - \frac12+\dl, \frac12}$-norm \eqref{op2}) for some small $\dl >0$.  (We only prove the relevant uniform in $N \in \N$ bound as convergence follow these and considerations as in Lemma \ref{LEM:sto_cubic}.) By definition of the $X^{s,b}_T$ norms and \eqref{op2}, it suffices to prove

\noi
\begin{align}
\sup_{N \in \N} \sup_{\|v \|_{C_\g X^{s,\frac12 +\eps}} \le 1 }  \big\| \ind_{[0,1]}(t) \NN^{1,>} (v, \<1>_{\g,N} , \<1>_{\g,N} )  \big\|_{C_\g X^{s, -\frac12 + \dl}}  < \infty,
\label{rmt1}
\end{align}

\noi
almost surely.
From \eqref{nonlin}, \eqref{rmt1} follows from the following bound:

\noi
\begin{align}
\sup_{\|v \|_{C_\g X^{s,\frac12+ \eps}} \le 1 }  \big\| \ind_{[0,1]}(t) \NN \big( \P_{\gg N_2 \vee N_3 }v, \P_{N_2} \<1>_{\g,N} , \P_{N_3} \<1>_{\g,N} )  \big\|_{C_\g X^{s, -\frac12 + \dl}} \les (N_2 \vee N_3)^{-\dl_0},
\label{rmt1b}
\end{align}

\noi
almost surely and for some small $\dl_0>0$, uniformly in $N \in \N$, and for dyadic numbers $N_2, N_3$.   

Let us fix $N \in \N$, $\g \in [0,1]$, some dyadic numbers $N_2, N_3$, and $v \in X^{s,\frac12+\eps}$. Let $ \{ \Ld \}_{ \Ld \in \mathfrak{B} }$ be a finitely overlapping family of (countable) balls of radius $\sim N_2 \vee N_3$ which covers the set $\{ \xi \in \R^2 : |\xi| \gg N_2 \vee N_3\}$. By orthogonality, in order to show \eqref{rmt1b}, it suffices to prove the following bound:

\noi
\begin{align}
\begin{split}
&\sup_{\|v \|_{C_\g X^{s,\frac12+ \eps}} \le 1 }  \big\| \ind_{[0,1]}(t) \NN\big( \P_\Ld \P_{\gg N_2 \vee N_3 }v, \P_{N_2} \<1>_{\g,N} , \P_{N_3} \<1>_{\g,N} )  \big\|_{C_\g X^{s, -\frac12 + \dl}} \\
& \hspace{80mm} \les (N_2 \vee N_3)^{-\dl_0},
\end{split}
\label{rmt1c}
\end{align}

\noi
almost surely and uniformly in $\Ld \in \mathfrak{B}$. Next, we observe that by the Kolmogorov continuity criterion and linearity, \eqref{rmt1c} directly follows from the following bound:

\noi
\begin{align}
\begin{split}
& \sup_{\g \in [0,1]} \Big\| \sup_{\|v \|_{ C_{\g_0} X^{s,\frac12+ \eps}} \le 1 }  \big\| \ind_{[0,1]}(t) \NN\big( \P_\Ld \P_{\gg N_2 \vee N_3 }v, \P_{N_2} \<1>_{\g,N} , \P_{N_3} \<1>_{\g,N} )  \big\|_{X^{s, -\frac12 + \dl}} \Big\|_{L^p(\O)} \\
& \hspace{80mm} \les (N_2 \vee N_3)^{-\dl_0},
\end{split}
\label{rmt1c2}
\end{align}

\noi
for $p \ge 1$, uniformly in $\Ld \in \mathfrak{B}$, and $N \in \N$. Namely, we dropped the dependence of $\g$ in the $C_\g X^{s, -\frac12 + \dl}$-norm in \eqref{rmt1c}. In the remainder of the proof, we show \eqref{rmt1c2}.

Fix $\g \in [0,1]$, $N \in \N$, $\Ld \in \mathfrak{B}$ and two dyadic numbers $N_2$ and $N_3$. In what follows, we will omit the dependence of the objects in $\g$ and $N$ for convenience. Let us note however that the bounds we obtain are uniform in these parameters. Let $N_{\star} = (\Ld, N_2, N_3)$. Applying Lemma \ref{LEM:prod} gives

\noi
\begin{align}
\begin{split}
& \ind_{[0,1]}(t) \NN (\P_\Ld  \P_{\gg N_2 \vee N_3 } v, \P_{N_2} \<1>_{\g,N} , \P_{N_3} \<1>_{\g,N} ) (n, t) \\
& \hspace{50mm}  = \int_{\R} d\mu \sum_{n_1 \in \Z^2} \tw {\P_\Ld v} (n_1, \mu) I_2 \big[ h^{N_{\star} }_{n_1,n,\mu,t} \big],
\end{split}
\label{rmt2}
\end{align}

\noi
where $h^{N_{\star}}_{n_1,n,\mu,t} = h^{N_{\star} }_{n_1,n,\mu,t}(z_2, z_3) $ (with $z_j = (n_j, t_j, \ze_j) $ for $2 \le j \le 3$) is given by

\noi
\begin{align}
\begin{split}
h^{N_{\star} }_{n_1,n,\mu,t}(z_2,z_3) & = e^{it (\mu - \jb{n_1}^2)} \ind_{[0,1]}(t) \ind_{n_1 \in \Ld} \ind_{\jb{n_1} \gg N_2 \vee N_3} \prod_{j=2}^3 \ind_{\jb{n_j} \sim N_j}  \\
& \quad \quad \times \ind_{ \substack{ n_1 - n_2 + n_3 = n  \\ \jb{n_2}, \jb{n_3} \le N } } \ind_{n_2 \neq n_1, n_3} \, f_t \otimes \cj{f_t}(z_2,z_3),
\end{split}
\label{rmt4}
\end{align}

\noi
with $f_t$ as in \eqref{f}. We may then write $h^{N_{\star}}_{n_1,n,\mu,t}$ as a sum of terms of the form

\noi
\begin{align}
\begin{split}
h^{A, N_{\star}}_{n_1,n,\mu,t}(z_2,z_3) & = e^{- t ( i(  \mu + \jb n ^2 - \kk (\bar n )) + \g \be_1( \bar n) )} \ind_{[ t_{\max}(A),1]}(t) \ind_{n_1 \in \Ld} \ind_{\jb{n_1} \gg N_2 \vee N_3}  \prod_{j=2}^3 \ind_{\jb{n_j} \sim N_j}  \\
& \quad \times \ind_{ \substack{ n_1 - n_2 + n_3 = n  \\ \jb{n_2}, \jb{n_3} \le N } } \ind_{n_2 \neq n_1, n_3} \cdot \prod_{j \in B} \frac{\ind_{[0,1]} (t_j) }{\jb{n_j}} \\
& \quad \times \prod_{j \in A} \Big( \sqrt{2 \g} e^{(\g + i)t_j \jb{n_j}^2 } \ind_{[0,1]}(t_j) \Big)^{\iota_j} ,
\end{split} 
\label{rmt4a}
\end{align}

\noi
where $A $ and $B$ form a partition of $\{2,3\}$ such that $\ze_j = 1$ for $j \in A$ and $\ze_j = -1$ for $j \in B$. In the above, $\kk( \bar n )$ is as in \eqref{c5} and $t_{\max}(A)$ and $\be_1( \bar n)$ are defined by

\noi
\begin{align}
\begin{split}
t_{\max}(A) & := \max \{ t_j : j \in A \}  \\
\be_1( \bar n ) &:= \jb{n_2}^2 + \jb{n_3}^2.
\end{split}
\label{rmt4b}
\end{align}

\noi
Similarly, we define $n_{\min}(A) = \min \{ n_j : j \in A \}$ and $n_{\max}(A) = \max \{n_j : j \in A \}$ with the conventions of the proof of Proposition \ref{LEM:sto_cubic} if $A = \emptyset$, etc.

Taking the twisted space-time Fourier transform in \eqref{rmt2} and applying Lemma \ref{LEM:B3} gives

\noi
\begin{align}
\begin{split}
& \tw \F \big( \ind_{[0,1]}(t) \NN \big( \P_\Ld  \P_{\gg N_2 \vee N_3 }  v, \P_{N_2} \<1>_{\g,N} , \P_{N_3} \<1>_{\g,N}  \big)(n, \ld) \\
&  \qquad \quad = \sum_{A \subset \{1,2\} } \int_{\R} d\mu \sum_{n_1 \in \Z^2} \tw{\P_\Ld v} (n_1, \mu) H^{A, N_{\star}}(n_1,n,\mu, \ld),
\end{split}
\label{rmt5}
\end{align}

\noi
where $H^{A, N_{\star}} = I_2 \big[ \tw{ h }^{A, N_{\star}}_{n_1,n,\mu,\ld} \big]$, and $\tw{h}^{A, N_{\star}}_{n_1,n,\mu,\ld} = \tw{h}^{A, N_{\star}}_{n_1,n,\mu,\ld}(z_2, z_3) $ is given by

\noi
\begin{align}
\begin{split}
\tw {h}^{A, N_{\star}}_{n_1,n,\mu,\ld}(z_2,z_3) & =  \frac{e^{ -  \Phi_1 (\bar n ) } - e^{- t_{\max}(A)  \Phi_1 (\bar n ) } }{\Phi_1 (\bar n ) }  \ind_{n_1 \in \Ld} \ind_{\jb{n_1} \gg N_2 \vee N_3} \prod_{j=2}^3 \ind_{\jb{n_j} \sim N_j} \\
& \quad \times \ind_{ \substack{ n_1 - n_2 + n_3 = n  \\ \jb{n_2}, \jb{n_3} \le N } } \ind_{n_2 \neq n_1, n_3} \cdot \prod_{j \in B} \frac{\ind_{[0,1]} (t_j) }{\jb{n_j}} \\
& \quad \times \prod_{j \in A} \Big( \sqrt{2 \g} e^{(\g + i)t_j \jb{n_j}^2 } \ind_{[0,1]}(t_j) \Big)^{\iota_j} ,
\end{split}
\label{rmt6}
\end{align}

\noi
with

\noi
\begin{align}
\Phi_1 (\bar n ) := i ( \ld - \mu - \kk( \bar n)) + \g \be_1 (\bar n ).
\label{rmt7}
\end{align}

We now address \eqref{rmt1c2} with $-\frac12 + \dl$ replaced by $- \frac12 - \dl$ (arguing as in Lemma \ref{LEM:sto_cubic}). An interpolation argument will give \eqref{rmt1c} as in the proof of Lemma \ref{LEM:sto_cubic}. By definition of the $v$, and Minkowski and Cauchy-Schwarz inequalities (in $\mu$), we estimate

\noi
\begin{align}
& \Big\| \sup_{\|v \|_{C_{\g_0} X^{s,\frac12+ \eps}} \le 1 } \big\| \ind_{[0,1]}(t) \NN \big( \P_\Ld \P_{\gg N_2 \vee N_3 }  v, \P_{N_2} \<1>_{\g,N} , \P_{N_3} \<1>_{\g,N}  \big)  \big\|_{X^{s, -\frac12 - \dl}} \Big\|_{L^p(\O)} \notag \\
& \quad \les \max_{A \subset \{1, 2\}} \Big\|   \sup_{\|v \|_{C_{\g_0} X^{s,\frac12+ \eps}} \le 1 }  \big\|  \jb n ^s \jb{\ld}^{-\frac12 - \dl} \int_{\R} d\mu \sum_{n_1 \in \Z^2} \tw{\P_\Ld v} (\g_0, n_1, \mu) H^{A, N_{\star}}(n_1,n,\mu, \ld) \big\|_{\l^2_n L^2_\ld} \Big\|_{L^{p}(\O)} \notag \\
& \quad \les \max_{A \subset \{2,3\}} \sup_{\mu,\ld \in \R} \Big\| \big\| \mathfrak{H}^{A, N_{\star}}(n_1,n,\mu,\ld) \big\|_{\l^2_{n_1} \to \l^2_n} \Big\|_{L^p(\O)} \label{rmt8},
\end{align}

\noi
with 

\noi
\begin{align*}
\mathfrak{H}^{A, N_{\star}}(n_1,n,\mu,\ld) = \jb{n}^s \jb {n_1} ^{-s} H^{A, N_{\star}}(n_1,n,\mu, \ld).
\end{align*}

\noi
Hence, \eqref{rmt1c2} reduces to the bound

\noi
\begin{align}
\max_{A \subset \{2,3\}} \sup_{\mu,\ld \in \R} \Big\| \big\| \mathfrak{H}^{A, N_{\star}}(n_1,n,\mu,\ld) \big\|_{\l^2_{n_1} \to \l^2_n} \Big\|_{L^p(\O)} \les p (N_2 \vee N_3)^{-\dl_0},
\label{rmt200}
\end{align}

\noi
for $p \ge 1$, and some small $\dl_0 >0$. From \eqref{rmt6}, we have

\noi
\begin{align}
\mathfrak{H}^{A, N_{\star}}(n_1,n,\mu,\ld) = I_2 \big[ \mf h ^{A, N_{\star}}_{n , n_1}(n_2, n_3) \mf f^{A} _{n, n_1, \mu, \ld} (z_2,z_3)   \big],
\label{rmt9b}
\end{align}

\noi
with

\noi
\begin{align}
\begin{split}
\mf h ^{A, N_{\star}}_{n , n_1}(n_2, n_3) & :=  \ind_{n_1 \in \Ld} \ind_{\jb{n_1} \gg N_2 \vee N_3} \prod_{j=2}^3 \ind_{\jb{n_j} \sim N_j} \ind_{ \substack{ n_1 - n_2 + n_3 = n  \\ \jb{n_2}, \jb{n_3} \le N } } \ind_{n_2 \neq n_1, n_3} \\
& \qquad \times \jb{n}^s \jb {n_1} ^{-s}  \prod_{j \in B} \frac{1}{\jb{n_j}},
\end{split}
\label{rmt10}
\end{align}

\noi
and 

\noi
\begin{align}
\begin{split}
\mf f^{A} _{m,n, n_1, \mu, \ld} (z_2,z_3) & :=  \frac{e^{ -  \Phi_1 (\bar n ) } - e^{- t_{\max}(A)  \Phi_1 (\bar n ) } }{\Phi_1 (\bar n ) }    \\
& \qquad \times  \prod_{j = 1}^3 \ind_{[0,1]}(t_j) \cdot \prod_{j \in A} \Big( \sqrt{2 \g} e^{(\g + i)t_j \jb{n_j}^2 } \Big)^{\iota_j}.
\end{split}
\label{rmt11}
\end{align}

\noi
We then have from \eqref{rmt4b}, \eqref{rmt7} and the mean value theorem

\noi
\begin{align}
\Big| \frac{e^{ -  \Phi_1 (\bar n ) } - e^{- t_{\max}(A)  \Phi_1 (\bar n ) } }{\Phi_1 (\bar n ) }  \Big| & \les  \frac{e^{- \g t_{\max}(A) \be_1 (\bar n ) }}{ \jb{ \Phi_1 (\bar n )} }
\label{rmt12}
\end{align}

\noi
Hence, putting together \eqref{rmt11} and \eqref{rmt12} yields (as in \eqref{c8})

\noi
\begin{align}
\| \mf f^{A} _{m,n, n_1, \mu, \ld}  \|_{L^2_{t_2, t_3 } ([0,1]^2)} \les \frac{\g^{ \frac{|A|}{2} }}{\jb{ \Phi_1 (\bar n ) }} \min \Big( 1 , \frac{1}{\g^{\frac12} r(A)  } \Big) =: d_\g (A), 
\label{rmt13}
\end{align}

\noi
where 

\noi
\begin{align}
r(A) =  \frac{\ind_{A = \{2,3\}}}{\jb{n_{\min} (A)}} + \frac{  \ind_{A = \{2\} } }{\jb{ n_3 }} +  \frac{  \ind_{A = \{3\} } }{\jb{ n_2 }} + \ind_{A = \emptyset}.
\label{rmt13b} 
\end{align}
\noi
Note that the factor $\jb{n_{\min} (A)}$ in \eqref{rmt13b} comes from the case where $n_{\max}(A)$ and $t_{\max}(A)$ correspond to the same index (say $n_{\max}(A) = n_2$ and $t_{\max}(A) = t_2$, for instance). In other cases, $\jb{n_{\min} (A)}$ may be replaced by the better factor $\jb{n_{\max} (A)}$.

Let $\mathfrak{L}^{A, N_{\star}} = \mathfrak{L}^{A, N_{\star}}_{nn_1 n_2 n_3}$ be the tensor (which also depends on $\mu$ and $\ld$) defined by

\noi
\begin{align}
\begin{split}
\mathfrak{L}^{A, N_{\star}}_{nn_1 n_2 n_3} & = \ind_{n_1 \in \Ld} \ind_{\jb{n_1} \gg N_2 \vee N_3} \prod_{j=2}^3 \ind_{\jb{n_j} \sim N_j} \cdot \ind_{ \substack{ n_1 - n_2 + n_3 = n  \\ \jb{n_2}, \jb{n_3} \le N } } \ind_{n_2 \neq n_1, n_3} \\
& \qquad \times  \prod_{j \in B} \frac{1}{\jb{n_j}} \cdot d_\g(A).
\end{split}
\label{rmt14}
\end{align}

\noi
From Lemma \ref{LEM:DNY}, \eqref{rmt9b} and \eqref{rmt13}, we have 

\noi
\begin{align}
\eqref{rmt200} & \les  p (N_2 \vee N_3)^\eps \max_{A \subset \{2,3\}} \sup_{\mu,\ld \in \R} \| \mathfrak{L}^{A, N_{\star}} \|_1,
\label{rmt15}
  \end{align}
  
\noi
for any $\eps >0$, and with $\| \cdot \|_1$ as in \eqref{tnorm1}. Thus, to obtain \eqref{rmt200}, it suffices to show

\noi
\begin{align}
 \| \mathfrak{L}^{A, N_{\star}} \|_1 \les \g^{ \ta |A|  } (N_2 \vee N_3)^{-\ta},
\label{rmtg}
\end{align}

\noi
for any $A \subset \{2, 3\}$, $\g \in [0,1]$, some small $\ta >0$, and uniformly in $(\mu, \ld) \in \R^2$. In the rest of the proof we hence prove \eqref{rmtg}.

We divide our analysis into three cases: (i) $A = \{2,3\}$ and $B = \emptyset$, (ii) $A = \{2\}$ and $B = \{3\}$ or $A = \{3\}$ and $A = \{2\}$, and (iii) $A = \emptyset$ and $B = \{2,3\}$.

\medskip

\noi
$\bullet$
{\bf Case (i):}  $A = \{2,3\}$ and $B = \emptyset$. In this case, we have 

\noi
\begin{align*}
d_\g(  \{2,3\} )= \frac{\g}{\jb{ \Phi_1 (\bar n ) }} \min \Big( 1 , \frac{1}{\g^{\frac12} \jb{ n_{\min}( \{2,3\} )}  } \Big). 
\end{align*}

\noi
We obtain several bounds depending on the size of $d_\g(  \{2,3\} )$. Namely, by \eqref{rmt4b} and \eqref{rmt7}, we have the following three bounds:

\noi
\begin{align}
d_\g(  \{2,3\} ) & \les \frac{1}{ \jb{n_{\max}(\{2,3\})}^2}, \label{rmt16} \\
d_\g(  \{2,3\} ) & \les  \frac{ \g^{\frac12} }{ \jb{\ld - \mu - \kk( \bar n )}  \jb{n_{\min}(\{2,3\})}}, \label{rmt17} \\
d_\g(  \{2,3\} ) & \les \frac{\g^{- \frac12}}{ \jb{n_{\max}(\{2,3\})}^2 \jb{n_{\min}(\{2,3\})}}. \label{rmt18}
\end{align}

\noi
Plugging \eqref{rmt16} into \eqref{rmt14} gives, by Lemma \ref{LEM:t1} (i), 

\noi
\begin{align}
\| \mathfrak{L}^{\{2,3\}, N_{\star} } \|_1 \les (N_2 \vee N_3)^{-2} N_2 N_3 = (N_2 \vee N_3)^{-1} (N_2 \wedge N_3).
\label{rmt19}
\end{align}

Next, we have the following decomposition:

\noi
\begin{align}
\| \mathfrak{L}^{\{2,3\}, N_{\star}} \|_1 &  \leq \big\| \mathfrak{L}^{\{2,3\}, N_{\star} } \ind_{ \jb{\ld - \mu - \kk( \bar n )} \leq (N_2 \vee N_3)^{10} }  \big\|_1 \label{rmt20} \\
& \quad +  \big\| \mathfrak{L}^{\{2,3\}, N_{\star}} \ind_{ \jb{\ld - \mu - \kk( \bar n )} > (N_2 \vee N_3)^{10} }  \big\|_1 \label{rmt21}
\end{align}

\noi
By Lemma \ref{LEM:t1} (i), \eqref{rmt14} and \eqref{rmt17}, we have

\noi
\begin{align}
\eqref{rmt21} \les \g^{\frac12} (N_2 \vee N_3)^{-10} (N_2 \wedge N_3)^{-1} N_2 N_3 \les \g^{\frac12} (N_2 \vee N_3)^{-9}. \label{rmt22}
\end{align}

\noi
Furthermore, we may estimate \eqref{rmt20} by fixing the value of the phase function $\kk( \bar n)$. More precisely, we have from Lemma \ref{LEM:t1} (ii), along with \eqref{rmt14} and \eqref{rmt17},

\noi
\begin{align}
\eqref{rmt20} & \les \g^{\frac12} \sum_{ \substack{ m \in \Z \\ \jb{ \ld - \mu - m } \leq (N_2 \vee N_3)^{10}} } \frac{1}{\jb{ \ld - \mu - m}}  \big\| \mathfrak{L}^{\Ld, \{2,3\} } \ind_{ \kk( \bar n ) = m }  \big\|_1 \notag \\
& \les \g^{\frac12} \log \big(1 + N_2 \vee N_3 \big) \sup_{m \in \Z} \big\| \mathfrak{L}^{\Ld, \{2,3\} } \ind_{ \kk( \bar n ) = m }  \big\|_1 \notag \\
& \les \g^{\frac12} (N_2 \vee N_3)^{\frac12 + \ta} (N_2 \wedge N_3)^{-\frac12} \label{rmt23},
\end{align}

\noi
for any $\ta >0$. Hence, \eqref{rmt22} and \eqref{rmt23} yield

\noi
\begin{align}
\| \mathfrak{L}^{\{2,3\}, N_{\star}} \|_1 \les \g^{\frac12} (N_2 \vee N_3)^{\frac12 + \ta} (N_2 \wedge N_3)^{-\frac12},
\label{rmt24}
\end{align}

\noi
for any $\ta >0$.

Lastly, from \eqref{rmt18}, \eqref{rmt14} and Lemma \ref{LEM:t1} (i), we have

\noi
\begin{align}
\| \mathfrak{L}^{\{2,3\}, N_{\star}} \|_1 \les \g^{-\frac12} (N_2 \vee N_3)^{-1}.
\label{rmt25}
\end{align}

Finally, interpolating \eqref{rmt19}, \eqref{rmt24}, and \eqref{rmt25} gives

\noi
\begin{align*}
\| \mathfrak{L}^{\{2,3\}, N_{\star}} \|_1 \les \g^{\ta} (N_2 \vee N_3)^{-\ta},
\end{align*}

\noi
for some small $\ta >0$; which is acceptable in view of \eqref{rmtg}.

\medskip

\noi
$\bullet$
{\bf Case (ii):}  $A = \{2\}$ and $B = \{3\}$ or $A = \{3\}$ and $B = \{2\}$. We assume $A = \{2\}$ and $B = \{3\}$ as the other case is similar. In this case, we have 

\noi
\begin{align*}
d_\g(  \{2\} )= \frac{\g^{\frac12}}{\jb{ \Phi_1 (\bar n ) }} \min \Big( 1 , \frac{1}{\g^{\frac12} \jb{ n_3 }  } \Big). 
\end{align*}

\noi
As in case (i), we hence have the bounds

\noi
\begin{align}
d_\g(  \{2\} ) & \les \frac{1}{ \jb{\ld - \mu - \kk( \bar n)} \jb{n_{3}}}, \label{rmt27} \\
d_\g(  \{2\} ) & \les  \frac{ \g^{\frac12} }{ \jb{\ld - \mu - \kk( \bar n )} }, \label{rmt28} \\
d_\g(  \{2\} ) & \les \frac{\g^{- \frac12}}{ \big(\jb{n_2} \vee \jb{n_3} \big)^2  }. \label{rmt29}
\end{align}

\noi
From \eqref{rmt27} with \eqref{rmt14}, and arguing as in the case \eqref{rmt17} to fix values of the phase $\kk (\bar n)$, we have by Lemma \ref{LEM:t1} (ii),

\noi
\begin{align}
\| \mathfrak{L}^{\{2\}, N_{\star}} \|_1 & \les N_3^{-2} (N_2\vee N_3)^{\frac12 + \ta} (N_2 \wedge N_3)^{\frac12} \notag \\
& \les (N_2\vee N_3)^{\frac12 + \ta} N_3^{- \frac32}, \label{rmt30}
\end{align}

\noi
for any $\ta >0$.

Similarly, we have from \eqref{rmt14}, \eqref{rmt28}, and Lemma \ref{LEM:t1} (ii),

\noi
\begin{align}
\| \mathfrak{L}^{\{2\}, N_{\star}} \|_1 \les \g^{\frac12}  (N_2\vee N_3)^{\frac12 + \ta} N_3^{- \frac12}.
\label{rmt31}
\end{align}

\noi
for any $\ta >0$.

Lastly, by \eqref{rmt14}, \eqref{rmt29} and Lemma \ref{LEM:t1} (i),

\noi
\begin{align}
\| \mathfrak{L}^{\{2\}, N_{\star}} \|_1 \les \g^{- \frac12}  (N_2\vee N_3)^{- 1},
\label{rmt32}
\end{align}

\noi
Interpolating \eqref{rmt30}, \eqref{rmt31}, and \eqref{rmt32} gives \eqref{rmtg} in this case.

\medskip

\noi
$\bullet$
{\bf Case (iii):} $A = \emptyset$ and $B=\{2,3\} $. In this case, we have

\noi
\begin{align*}
d_\g(\emptyset) = \frac{1}{\jb{\Phi_1(\bar n)}} \les \frac{1}{\jb{ \ld - \mu - \kk (\bar n)}}
\end{align*}

\noi
Hence, from the above, \eqref{rmt14}, we have by arguing as in \eqref{rmt17},

\noi
\begin{align}
\| \mathfrak{L}^{\emptyset, N_{\star} } \|_1 \les N_2^{-1} N_3^{-1} (N_2 \vee N_3)^{\frac12 + \ta} (N_2 \wedge N_3)^{\frac12},
\end{align}

\noi
for any $\ta >0$, which is acceptable in view of \eqref{rmtg}. This concludes the proof of \eqref{rmt1b}.
\end{proof}

We now prove Lemma \ref{LEM:RMT2}.

\noi
\begin{proof}[Proof of Lemma \ref{LEM:RMT2}]
We focus on the operator $\{ \g \mapsto \I_\g \NN^{1,<}(\cdot, \<1>_{\g,N}, \<1>_{\g,N} ) \}_{N \in \N}$ since the case $j=2$ is similar; namely, we fix $j=1$. By Proposition \ref{PROP:duha1}, \eqref{rop}, and \eqref{rop2} it suffices to prove that $\{ \g \mapsto \NN^{1,<}(\cdot, \<1>_{\g,N}, \<1>_{\g,N} ) \}_{N \in \N}$ is a Cauchy sequence in $\L_{T_0}^{s,s, - \frac12+\dl, \frac12+\eps}$ (without the $T^\ta$ factor gain in the $\L_{T_0}^{s,s, - \frac12+\dl, \frac12}$-norm \eqref{op2}) for some small $\dl >0$.  (We only prove the relevant uniform in $N \in \N$ bound as convergence follow these and considerations as in Lemma \ref{LEM:sto_cubic}.) We aim at proving the following bound:

\noi
\begin{align}
\sup_{N \in \N}   \sup_{\|v \|_{C_\g X^{s,\frac12 +\eps}} \le 1 }  \big\| \ind_{[0,1]}(t) \NN^{1,<} (v, \<1>_{\g,N} , \<1>_{\g,N} )  \big\|_{C_\g X^{s, -\frac12 + \dl}}  < \infty,
\label{rmtz}
\end{align}

\noi
almost surely. By proceeding as in the proof of Lemma \ref{LEM:RMT}, it suffices to prove the following bound

\noi
\begin{align}
\sup_{\g \in [0,1]} \Big\| \sup_{\|v \|_{C_{\g_0} X^{s,\frac12+ \eps}} \le 1 }  \big\| \ind_{[0,1]}(t) \NN \big( \P_{N_1}v, \P_{N_2} \<1>_{\g,N} , \P_{N_3} \<1>_{\g,N} )  \big\|_{X^{s, -\frac12 + \dl}} \Big\|_{L^p(\O)} \les N_{\max}^{-\dl_0},
\label{rmtz1}
\end{align}

\noi
almost surely and for some small $\dl_0>0$, uniformly in $N \in \N$, and for dyadic numbers $N_1, N_2, N_3$. Here, $N_{\max} = \max(N_1, N_2, N_3)$.

We now prove \eqref{rmtz1} with $-\frac12 + \dl$ replaced by $- \frac12 - \dl$ (arguing as in Lemma \ref{LEM:sto_cubic}). To this end, let us fix dyadic numbers $N_1, N_2, N_3$ and let $N_{\star} = (N_1, N_2, N_3)$. By similar considerations as in (the proof of) Lemma \ref{LEM:RMT} and Lemma \ref{LEM:DNY},

\noi
\begin{align}
\begin{split}
& \Big\| \sup_{\|v \|_{C_{\g_0} X^{s,\frac12+ \eps}} \le 1 } \big\| \ind_{[0,1]}(t) \NN \big(\P_{N_1}  v, \P_{N_2} \<1>_{\g,N} , \P_{N_3} \<1>_{\g,N}  \big)  \big\|_{X^{s, -\frac12 - \dl}} \Big\|_{L^p(\O)} \\
& \quad \les p  N_{\max}^\eps \max_{A \subset \{2,3\}} \sup_{\mu,\ld \in \R} \| \mathfrak{L}^{A, N_{\star}} \|_1 ,
\end{split}
\label{rmtz2}
\end{align}

\noi
for any $\eps >0$, and with $\| \cdot \|_1$ as in \eqref{tnorm1}, and where $\mathfrak{L}^{A, N_{\star}} = \mathfrak{L}^{A, N_{\star}}_{nn_1 n_2 n_3}$ is the tensor (which also depends on $\mu$ and $\ld$) given by

\noi
\begin{align}
\begin{split}
\mathfrak{L}^{A, N_{\star}}_{nn_1 n_2 n_3} & =  \jb{n}^s \jb{n_1}^{-s}  \prod_{j=1}^3 \ind_{\jb{n_j} \sim N_j} \cdot \ind_{ \substack{ n_1 - n_2 + n_3 = n  \\ \jb{n_2}, \jb{n_3} \le N } } \ind_{n_2 \neq n_1, n_3} \\
& \qquad \times  \prod_{j \in B} \frac{1}{\jb{n_j}} \cdot d_\g(A),
\end{split}
\label{rmtz3}
\end{align}

\noi
with $d_\g (A)$ as in \eqref{rmt13}, \eqref{rmt13b}, and \eqref{rmt7}. Thus, to obtain \eqref{rmtz1}, it suffices to show

\noi
\begin{align}
 \| \mathfrak{L}^{A, N_{\star}} \|_1 \les \g^{ \ta |A|  } N_{\max}^{-\ta},
\label{rmtzg}
\end{align}

\noi
for any $A \subset \{2, 3\}$, $\g \in [0,1]$, some small $\ta >0$, and uniformly in $(\mu, \ld) \in \R^2$. We assume that $N_1 \ll N_2 \wedge N_3$ (hence, $N_{\max} = N_2 \vee N_3$) for now and divide our analysis into three cases: (i) $A = \{2,3\}$ and $B = \emptyset$, (ii) $A = \{2\}$ and $B = \{3\}$ or $A = \{3\}$ and $A = \{2\}$, and (iii) $A = \emptyset$ and $B = \{2,3\}$.

\medskip

\noi
$\bullet$
{\bf Case (i):}  $A = \{2,3\}$ and $B = \emptyset$. In this case, we have 

\noi
\begin{align*}
d_\g(  \{2,3\} )= \frac{\g}{\jb{ \Phi_1 (\bar n ) }} \min \Big( 1 , \frac{1}{\g^{\frac12} \jb{ n_{\min}( \{2,3\} )}  } \Big). 
\end{align*}

\noi
As in the proof of Lemma \ref{LEM:RMT}, we obtain by using the different bounds on $d_\g(A)$ \eqref{rmt16}, \eqref{rmt17} and \eqref{rmt18}, along with Lemma \ref{LEM:t1} (i) and (ii),

\noi
\begin{align}
 \| \mathfrak{L}^{\{2,3\}} \|_1 &  \les  N_{1}^{-s} (N_2 \vee N_3)^{s-1} (N_2 \wedge N_3)  \label{rmtz4}, \\
  \| \mathfrak{L}^{\{2,3\}} \|_1  & \les   \g^{\frac12} N_{1}^{-s} (N_2 \vee N_3)^{s + \frac12 + \ta} (N_2 \wedge N_3)^{-\frac12}, \label{rmtz5} \\
   \| \mathfrak{L}^{ \{2,3\} } \|_1 & \les \g^{-\frac12}  N_{1}^{-s} (N_2 \vee N_3)^{s-1}, \label{rmtz6}
\end{align}

\noi
for any $\ta >0$. By interpolation with \eqref{rmtz5}, \eqref{rmtzg} follows from the bound

\noi
\begin{align}
 \| \mathfrak{L}^{A} \|_1 \les (N_2 \vee N_3)^{-\ta},
 \label{rmtz7}
\end{align}

\noi
for some small $\ta>0$; which we now prove. Fix any $0 <\ta_0 \ll 1$. By \eqref{rmtz6}, if the following bound holds

\noi
\begin{align*}
\g^{-\frac12}  (N_2 \vee N_3)^{s-1} \les (N_2 \vee N_3)^{-\ta_0},
\end{align*}

\noi
then \eqref{rmtz7} holds. Otherwise, we have 

\noi
\begin{align}
\g^{\frac12}  \les (N_2 \vee N_3)^{s -1 + \ta_0 },
\label{rmtz8}
\end{align} 

\noi
Plugging \eqref{rmtz8} into \eqref{rmtz5} gives

\noi
\begin{align}
\eqref{rmtz5} \les N_1^{-s} (N_2 \vee N_3)^{2s - \frac12 + \ta + \ta_0} (N_2 \wedge N_3)^{-\frac12}.
\label{rmtz9}
\end{align}

\noi
We now note that if we have

\noi
\begin{align*}
(N_2 \vee N_3)^{s-1} (N_2 \wedge N_3) \les (N_2 \vee N_3)^{-\ta_0},
\end{align*}

\noi
then \eqref{rmtz7} holds. Otherwise, we have

\noi
\begin{align}
 (N_2 \wedge N_3) \ges (N_2 \vee N_3)^{1 - s - \ta_0},
 \label{rmtz10}
\end{align}

\noi
Putting \eqref{rmtz10} into \eqref{rmtz9} yields

\noi
\begin{align}
\eqref{rmtz5} \les N_1^{-s} (N_2 \vee N_3)^{\frac52 s - 1 + \ta + 2 \ta_0}.
\label{rmtz11}
\end{align}

\noi
Hence by choosing $s < \frac25$ and $\ta, \ta_0 >0$ small enough, \eqref{rmt11} ensures that \eqref{rmtz7} holds.

\medskip

\noi
$\bullet$
{\bf Case (ii):}  $A = \{2\}$ and $B = \{3\}$ or $A = \{3\}$ and $B = \{2\}$. We assume $A = \{2\}$ and $B = \{3\}$ as the other case is similar. In this case, we have 

\noi
\begin{align*}
d_\g(  \{2\} )= \frac{\g^{\frac12}}{\jb{ \Phi_1 (\bar n ) }} \min \Big( 1 , \frac{1}{\g^{\frac12} \jb{ n_3 }  } \Big). 
\end{align*}

\noi
As in the proof of Lemma \ref{LEM:RMT}, we obtain by using the different bounds on $d_\g(A)$, \eqref{rmtz3}, \eqref{rmt27}, \eqref{rmt28} and \eqref{rmt29}, along with Lemma \ref{LEM:t1} (i) and (ii),

\noi
\begin{align}
 \| \mathfrak{L}^{\{2\}} \|_1 &  \les  N_{1}^{-s} (N_2 \vee N_3)^{s + \frac12 + \ta}  N_3^{-\frac32}  \label{rmtz12}, \\
  \| \mathfrak{L}^{\{2\}} \|_1  & \les   \g^{\frac12} N_{1}^{-s} (N_2 \vee N_3)^{s +\frac12 + \ta}  N_3^{-\frac12}, \label{rmtz13} \\
   \| \mathfrak{L}^{ \{2\} } \|_1 & \les \g^{-\frac12}  N_{1}^{-s} (N_2 \vee N_3)^{s-1}, \label{rmtz14}
\end{align}

\noi
As before, by interpolation with \eqref{rmtz13}, it suffices to prove

\noi
\begin{align}
 \| \mathfrak{L}^{\{2\}} \|_1 \les N_{\max}^{- \ta},
\label{rmtz15}
\end{align}

\noi
for some small $\ta >0$. Fix any $0 < \ta_0 \ll 1$. If we have

\noi
\begin{align}
\g^{-\frac12}  (N_2 \vee N_3)^{s-1} \les (N_2 \vee N_3)^{-\ta_0},
\label{rmtz16}
\end{align}

\noi
Plugging \eqref{rmtz16} into \eqref{rmtz13} gives

\noi
\begin{align}
\eqref{rmtz13} \les N_{1}^{-s} (N_2 \vee N_3)^{2s - \frac12 + \ta}  N_3^{-\frac12}.
\label{rmtz17}
\end{align}

\noi
Thus, from \eqref{rmtz17}, we can conclude that \eqref{rmtz15} holds under the condition $s< \frac14$.

\medskip

\noi
$\bullet$
{\bf Case (iii):} $A = \emptyset$ and $B=\{2,3\} $. In this case, we have

\noi
\begin{align*}
d_\g(\emptyset) = \frac{1}{\jb{\Phi_1(\bar n)}} \les \frac{1}{\jb{ \ld - \mu - \kk (\bar n)}}
\end{align*}

\noi
Hence, from the above, \eqref{rmtz3}, we have by arguing as in \eqref{rmt17},

\noi
\begin{align*}
\| \mathfrak{L}^{ \emptyset } \|_1 &\les N_1^{-s} N_2^{-1} N_3^{-1}  (N_2 \vee N_3)^{s+\frac12 + \ta} (N_2 \wedge N_3)^{\frac12} \\
& \les N_1^{-s} (N_2 \vee N_3)^{s-\frac12 + \ta} (N_2 \wedge N_3)^{-\frac12}
\end{align*}

\noi
for any $\ta >0$, which is acceptable in view of \eqref{rmtzg}. This proves \eqref{rmtzg} for $N_1 \ges N_2 \wedge N_3 $.

The case $N_1 \ges N_2 \wedge N_3 $ is a direct consequence of the above computations for the case $N_1 \ll N_2 \wedge N_3$ and Lemma \ref{LEM:t1} (iii) (since if $N_1 \ges N_2 \wedge N_3 $, we can exploit the extra $N_1^{-s}$-factor to obtain better estimates). This concludes the proof of \eqref{rmtz1}.
\end{proof}

\subsection{Bilinear random operators}

\noi
The purpose of this subsection is to treat the bilinear random operator terms $\TT_\g ^{1}$ and $\TT_\g ^{2}$ in \eqref{rbili}. 

\noi
\begin{proposition}\label{PROP:bilin}
Fix $0 < s < \frac14$, $0 < \eps, \ta \ll 1$, and $0 < T_0 \le 1$. Let $j \in \{1,2\}$. The sequence $ \{\TT^j_{\g,N}\}_{N \in \N}$ defined in \eqref{rbili} is a Cauchy sequence in the class $\B_{T_0}^{s,s,\frac12+ \eps, \frac12+ \ta}$, almost surely. In particular, denoting the respective limits by $\TT_\g ^{j}$, we have

\noi
\begin{align*}
\TT_\g ^{j} \in \B_{T_0}^{s,s,\frac12+ \eps, \frac12+ \ta} .
\end{align*}
\end{proposition}

Let $j \in \{1,2\}$. By \eqref{sto2}, we can decompose $\TT^j_{\g,N}$ as 

\noi
\begin{align}
\TT^j_{\g,N} = \TT^{j,\ominus}_{\g,N} + \TT^{j,\oplus}_{\g,N}
\label{rbili10}
\end{align}

\noi
with

\noi
\begin{align}
\begin{split}
\TT^{j,\ominus}_{\g,N} (u,v) & := \I_\g \NN ( \Psi_{\g,N},u,v ) \\ 
\TT^{j,\oplus}_{\g,N} (u,v) & := \I_\g \NN ( S_\g(t) \phi_N ,u,v )
\end{split}
\label{rbili11}
\end{align}

\noi
We reduce the proof of Proposition \ref{PROP:bilin} to the construction of the sequence of operators $\{ \TT^{j,\ominus}_{\g,N} \}_{N \in N}$ and $\{ \TT^{j,\oplus}_{\g,N} \}_{N \in N}$ in appropriate spaces.

\noi
\begin{proposition}\label{PROP:bilin1}
Fix $0 < s < \frac14$, $0 < \eps, \ta \ll 1$, and $0 < T_0 \le 1$. Let $j \in \{1,2\}$. The sequence $ \{\TT^{j,\ominus}_{\g,N} \}_{N \in \N}$ defined in \eqref{rbili11} is a Cauchy sequence in the class $\B_{T_0}^{s,s,\frac12+ \eps, \frac12+ \ta}$, almost surely. In particular, denoting the respective limits by $\TT_\g ^{j, \ominus}$, we have

\noi
\begin{align*}
\TT_\g ^{j, \ominus} \in \B_{T_0}^{s,s,\frac12+ \eps, \frac12+ \ta} .
\end{align*}
\end{proposition}

\noi
\begin{proposition}\label{PROP:bilin2}
Fix $0 < s < \frac12$, $0 < \eps, \ta \ll 1$, and $0 < T_0 \le 1$. Let $j \in \{1,2\}$. The sequence $ \{\TT^{j, \oplus}_{\g,N}\}_{N \in \N}$ defined in \eqref{rbili11} is a Cauchy sequence in the class $\B_{T_0}^{s,s,\frac12+ \eps, \frac12+ \ta}$, almost surely. In particular, denoting the respective limits by $\TT_\g ^{j, \oplus}$, we have

\noi
\begin{align*}
\TT_\g ^{j, \oplus} \in \B_{T_0}^{s,s,\frac12+ \eps, \frac12+ \ta} .
\end{align*}
\end{proposition}

In what follows, we prove Proposition \ref{PROP:bilin1} as Proposition \ref{PROP:bilin2} follows from similar (and in fact simpler) arguments. To this end, we further decompose $\TT^{j,\ominus}_{\g,N}$, $j \in \{1,2\}$. First, we introduce the following frequency localized multilinear forms:

\noi
\begin{align}
\begin{split}
\NN^{1, \ominus,>}(u,v,w) & = \sum_{N_1 \gg N_2 \vee N_3} \NN \big( \P_{N_1}u, \P_{N_2} v, \P_{N_3}w \big), \\
\NN^{2, \ominus, >}(u,v,w) & = \sum_{N_2 \gg N_1 \vee N_3}\NN \big( \P_{N_1} u , \P_{N_2} v, \P_{N_3} w \big),
\end{split}
\label{nonlin3}
\end{align}

\noi
and 

\noi
\begin{align}
\NN^{j, \ominus, <}(u,v,w) := \NN^j(u,v,w) - \NN^{j, \ominus, >}(u,v,w),
\label{nonlin4}
\end{align}

\noi
for $j \in \{1,2\}$. We now introduce the operators

\noi
\begin{align}
\begin{split}
\TT^{1, \ominus, \dagger}_{\g,N}(u, v) & = \I_\g \NN^{1, \ominus, \dagger}\big(\<1>_{\g,N}, u, v \big) , \\
\TT^{2, \ominus, \dagger}_{\g,N}(u, v) & = \I_\g \NN^{2, \ominus, \dagger} \big(u, \<1>_{\g,N},  v \big),
\end{split}
\label{rop10}
\end{align}

\noi
for $\dagger \in \{ <, > \}$.

With these notations, Proposition \ref{PROP:bilin1} follows from the two next lemmas.

\noi
\begin{lemma}\label{LEM:bilin1}
Fix $s >0$, $0 < \eps, \ta \ll 1$, and $0 < T_0 \le 1$. Let $j \in \{1,2\}$. The sequence $ \{\TT^{j,\ominus, <}_{\g,N} \}_{N \in \N}$ defined in \eqref{rop10} is a Cauchy sequence in the class $\B_{T_0}^{s,s,\frac12+ \eps, \frac12+ \ta}$, almost surely. In particular, denoting the respective limits by $\TT_\g ^{j, \ominus, <}$, we have

\noi
\begin{align*}
\TT_\g ^{j, \ominus, <} \in \B_{T_0}^{s,s,\frac12+ \eps, \frac12+ \ta} .
\end{align*}
\end{lemma}

\noi
\begin{lemma}\label{LEM:bilin2}
Fix $0 < s < \frac14$, $0 < \eps, \ta \ll 1$, and $0 < T_0 \le 1$. Let $j \in \{1,2\}$. The sequence $ \{\TT^{j,\ominus, >}_{\g,N} \}_{N \in \N}$ defined in \eqref{rop10} is a Cauchy sequence in the class $\B_{T_0}^{s,s,\frac12+ \eps, \frac12+ \ta}$, almost surely. In particular, denoting the respective limits by $\TT_\g ^{j, \ominus, >}$, we have

\noi
\begin{align*}
\TT_\g ^{j, \ominus, >} \in \B_{T_0}^{s,s,\frac12+ \eps, \frac12+ \ta} .
\end{align*}
\end{lemma}

The rest of this subsection is devoted to the proof of Lemmas \ref{LEM:bilin1} and \ref{LEM:bilin2}.

\noi
\begin{proof}[Proof of Lemma \ref{LEM:bilin1}] We focus on the case $j=1$, i.e. on the operator $\{   \I_\g \NN^{1,\ominus, <}( \Psi_{\g,N}, \cdot, \cdot) \}_{N \in \N}$ since the case $j=2$ is similar; namely, we fix $j=1$. By Lemma \ref{LEM:duha1}, \eqref{nonlin4}, \eqref{rop10}, and it suffices to prove that $\{  \NN^{1,\ominus, <}( \Psi_{\g,N}, \cdot, \cdot) \}_{N \in \N}$ is a Cauchy sequence in $\B_{T_0}^{s,s,  \frac12+\eps, \frac12+\dl}$ and without the $T^\ta$-factor gain in the $\B_{T_0}^{s,s, \frac12+\eps, \frac12+ \ta}$-norm \eqref{op2} for some small $\dl >\ta$.  (We only prove the relevant uniform in $N \in \N$ bound as convergence follow these and considerations as in Lemma \ref{LEM:sto_cubic}.) Namely, we aim to prove the following bound:
\noi
\begin{align}
\sup_{N \in \N} \Big\| \sup_{\substack{ \|u \|_{C_{\g_0} X^{s,\frac12 +\eps}} \le 1 \\ \|v \|_{C_{\g_0} X^{s,\frac12 +\eps}} \le 1} }  \big\| \ind_{[0,1]}(t) \NN^{1,\ominus, <} (\Psi_{\g,N}, u, v )  \big\|_{C_{\g} X^{s, -\frac12 + \dl}} \Big\|_{L^p(\O)}  < \infty.
\label{bi100}
\end{align}

\noi
By arguing as in Lemma \ref{LEM:RMT}, \eqref{bi100} reduces to the following frequency localized estimate:

\noi
\begin{align}
\begin{split}
& \sup_{\g \in [0,1]} \sup_{N \in \N} \Big\| \sup_{\substack{ \|u \|_{C_{\g_0} X^{s,\frac12 +\eps}} \le 1 \\ \|v \|_{C_{\g_0} X^{s,\frac12 +\eps}} \le 1} }  \big\| \ind_{[0,1]}(t) \NN ( \P_{N_1} \Psi_{\g,N}, \P_{N_2} u, \P_{N_3} v )  \big\|_{X^{s, -\frac12 + \dl}} \Big\|_{L^p(\O)}  \\
&  \hspace{60mm} \les p (N_2 \vee N_3)^{-\dl_0},
\end{split}
\label{bi101}
\end{align}

\noi
for some small $\dl_0>0$, and for dyadic numbers $N_1, N_2, N_3$, with $N_1 \les N_2 \vee N_3$. Let us note that by Lemma \ref{LEM:S2}, if $N_2 \wedge N_3 \ges N_1^{\frac{1}{100}}$, then \eqref{bi101} follows directly by transferring some derivatives to $u$ and $v$ if needed. Hence, in the remaining of the proof, we assume the condition $N_2 \wedge N_3 \ll N_1^{\frac{1}{100}}$ and prove \eqref{bi101} in that case.

Let $(\g,N) \in [0,1] \times \N$, $N_1, N_2, N_3$ be dyadic numbers. Let us assume that we have $N_2 \ge N_3$ and hence we have $N_3 \ll N_1^{\frac{1}{100}}$.  We also write $N_{\star}  = (N_1,N_2,N_3)$. Next, by further decomposing $u$ into balls of radius $\sim N_2$ (as in the proof of Lemma \ref{LEM:RMT}) if need be, we may assume that $N_2 \sim N_1$. We compute with \eqref{ob1}

\noi
\begin{align}
\begin{split}
& \F_x \big( \ind_{[0,1]}(t) \NN \big( \P_{N_1} \Psi_{\g,N}, \P_{N_2} u, \P_{N_3} v \big) \big) (n,t)  = \int_{\R^2} d \mu_2 d\mu_3 \\
& \qquad \qquad \times \sum_{n_2, n_3 \in \Z^2} \cj{ \tw{u}(n_2, \mu_2) } \tw{v}(n_3, \mu_3) I_1 \big[ h^{\ominus, <, N_{\star}} _{nn_2 n_3 \mu_2 \mu_3 t} \big]
\end{split}
\label{bi102}
\end{align}

\noi
with

\noi
\begin{align}
\begin{split}
h^{\ominus,<, N_{\star}} _{nn_2 n_3  \mu_2 \mu_3 t} (n_1,t_1) & = \sqrt{2 \g} \, \ind_{\substack{n = n_1 - n_2 + n_3 \\ n_2 \neq n_1, n_3 }} \ind_{\jb{n_1} \le N} \prod_{j=1}^3 \ind_{\jb{n_j} \sim N_j} \\
& \quad \times e^{-t (  i(- \mu_2 + \mu_3 + \jb{n_1}^2 - \jb{n_2}^2 + \jb{n_3}^2 ) + \g \jb{n_1}^2    ) } \\
& \quad \times  e^{t_1 ( \g + i )\jb{n_1}^2} \ind_{[0,1]}(t)  \ind_{[0,t]}(t_1).
\end{split}
\label{bi103}
\end{align}

Taking the Fourier transform with Lemma \ref{LEM:B3} gives 

\noi
\begin{align}
\begin{split}
& \F_{x,t} \big( \ind_{[0,1]}(t) \NN \big( \P_{N_1} \Psi_{\g,N}, \P_{N_2} u, \P_{N_3} v \big) \big) (n,\ld)  = \int_{\R^2} d \mu_2 d\mu_3 \\
& \qquad \qquad \times \sum_{n_2, n_3 \in \Z^2} \cj{ \tw{u}(n_2, \mu_2) } \tw{v}(n_3, \mu_3) I_1 \Big[ \tw{h}^\ominus_{nn_2 n_3 \mu_2 \mu_3 t} \Big],
\end{split}
\label{bi104}
\end{align}

\noi
with

\noi
\begin{align}
\begin{split}
\tw{h}^{\ominus, <, N_{\star}}_{nn_2 n_3 \mu_2 \mu_3 \ld } (n_1,t_1) & = \sqrt{2 \g} \, \ind_{\substack{n = n_1 - n_2 + n_3 \\ n_2 \neq n_1, n_3 }} \ind_{\jb{n_1} \le N} \prod_{j=1}^3 \ind_{\jb{n_j} \sim N_j} \\
& \quad \times \frac{   e^{- \Phi_2 (\bar n)} - e^{- t_1 \Phi_2 (\bar n)}  }{\Phi_2( \bar n)} e^{t_1 ( \g + i )\jb{n_1}^2}  \ind_{[0,1]}(t_1).
\end{split}
\label{bi105}
\end{align}

\noi
In the above, we denoted by $\Phi_2(\bar n)$, the phase function

\noi
\begin{align}
\Phi_2 (\bar n) := i( \ld - \mu_2 - \mu_3 - \kk (\bar n)) + \g \jb{n_1}^2,
\label{bi106}
\end{align}

\noi
where $\kk( \bar n)$ is as in \eqref{dphase}. By the mean value theorem, we have

\noi
\begin{align}
\Big|  \frac{   e^{- \Phi_2 (\bar n)} - e^{- t_1 \Phi_2 (\bar n)}  }{\Phi_2( \bar n)}  \Big| \les \frac{e^{-t_1 \Phi_1 ( \bar n ) }}{\jb{ \Phi_2(\bar n) } }
\label{bi107}
\end{align}

We may replace $-\frac12 + \dl$ replaced by $- \frac12 - \dl$ in \eqref{bi101} (arguing as in Lemma \ref{LEM:sto_cubic}). By similar considerations as in (the proof of) Lemma \ref{LEM:RMT}, we have

\noi
\begin{align}
\begin{split}
& \Big\| \sup_{\substack{ \|u \|_{C_{\g_0} X^{s,\frac12 +\eps}} \le 1 \\ \|v \|_{C_{\g_0} X^{s,\frac12 +\eps}} \le 1} }  \big\| \ind_{[0,1]}(t) \NN ( \P_{N_1} \Psi_{\g,N}, \P_{N_2} u, \P_{N_3} v )  \big\|_{X^{s, -\frac12 - \dl}} \Big\|_{L^p(\O)} \\
& \quad \sup_{\ld, \mu_2, \mu_3 \in \R} \les \big\| \| \jb{n}^s \jb{n_2}^{-s} \jb{n_3}^{-s} \tw{h}^{\ominus, <, N_{\star}} _{nn_2 n_3 \mu_2 \mu_3 \ld} \|_{\l^2_n \to \l^2_{n_2 n_3}}  \big\|_{L^p(\O)} ,
\end{split}
\label{bi108}
\end{align}

\noi
Furthermore, applying Lemma \ref{LEM:DNY} gives

\noi
\begin{align}
\begin{split}
& \big\| \| \jb{n}^s \jb{n_2}^{-s} \jb{n_3}^{-s} \tw{h}^{\ominus,<, N_{\star}}_{nn_2 n_3  \mu_2 \mu_3 \ld } \|_{n \to n_2 n_3}  \big\|_{L^p(\O)} \\
& \qquad \qquad \les p (N_2 \vee N_3)^\ta \| \mathfrak{h}^{\ominus, <, N_{\star}} \|_{2},
\end{split}
\label{bi109}
\end{align}

\noi
for any $\ta >0$, and uniformly in $\ld, \mu_2, \mu_3 \in \R$, where $\mathfrak{h}^{\ominus, <, N_{\star}} = \mathfrak{h}^{\ominus, <, N_{\star}}_{n n_1 n_2 n_3}$ (which also depends on $\ld, \mu_2, \mu_3$) denotes the tensor

\noi
\begin{align}
\mathfrak{h}^{\ominus,<, N_{\star}}_{n n_1 n_2 n_3} = \ind_{\substack{n = n_1 - n_2 + n_3 \\ n_2 \neq n_1, n_3 }} \ind_{\jb{n_1} \le N} \prod_{j=1}^3 \ind_{\jb{n_j} \sim N_j} \cdot \frac{ \jb n ^s }{ \jb{n_2}^s \jb{n_3}^s } \cdot  \frac{  \sqrt{2 \g}    }{ \jb{ \Phi_2( \bar n)}}.
\label{bi110}
\end{align}
\noi

\noi
with $\| \cdot \|_2$ as in \eqref{tnorm2}. From \eqref{bi105}, \eqref{bi106}, \eqref{bi107}, and Lemma \ref{LEM:t2} (i) and (ii) (fixing the values of $\kk (\bar n)$ as in Lemma \ref{LEM:RMT2}), we obtain the bounds

\noi
\begin{align}
\|  \mathfrak{h}^{\ominus,<, N_{\star}} \|_2 & \les  \g^{-\frac12} (N_2 N_3)^{-s} N_2^{s-1} N_{3},    \label{bi111}\\
\|  \mathfrak{h}^{\ominus,<, N_{\star}} \|_2 & \les   \g^{\frac12} (N_2 N_3)^{-s} N_2^{s+\frac12+ \ta} N_{3}^{\frac12}, \label{bi112}
\end{align}

\noi
for any $\ta >0$. Interpolating \eqref{bi111} and \eqref{bi112} gives

\noi
\begin{align*}
\eqref{bi109} \les p \g^{\ta} N_1^{- \ta},
\end{align*}

\noi
for some small $\ta >0$. This concludes the proof.
\end{proof}

Next, we deal with the proof of Lemma \ref{LEM:bilin2}. It turns out that a straightforward adaptation of the proof of Lemma \ref{LEM:bilin1} fails: for some range of frequencies, the bounds corresponding to \eqref{bi111} and \eqref{bi112} are not sufficient to close the relevant estimates. In order to overcome this issue, we will use directly use the (heat) smoothing coming from the Duhamel operator $\I_\g$, $\g >0$, in Lemma \ref{LEM:duha1} \eqref{D01}, along with the spatial integrability of the stochastic convolution $\Psi_{\g,N}$; see Remark \ref{RMK:bilin} below.

\noi
\begin{proof}[Proof of Lemma \ref{LEM:bilin2}] We focus on the operator $\{\I_\g \NN^{1, \ominus, >} ( \Psi_{\g,N}, \cdot, \cdot) \}_{N \in \N}$ since the case $j=2$ is similar; namely, we fix $j=1$. By Lemma \ref{LEM:duha0}, \eqref{nonlin3}, \eqref{rop10}, and it suffices to prove that $\{  \I_\g \NN( \Psi_{\g,N}, \cdot, \cdot) \}_{N \in \N}$ is a Cauchy sequence in $\B_{T_0}^{s,s,  \frac12+\eps, \frac12+\dl}$ and without the $T^\ta$-factor gain in the $\B_{T_0}^{s,s, \frac12+\eps, \frac12+ \ta}$-norm \eqref{op2} for some small $\dl >\ta$.  (We only prove the relevant uniform in $N \in \N$ bound as convergence follow these and considerations as in Lemma \ref{LEM:sto_cubic}.) We aim at proving the following bound:

\noi
\begin{align}
\sup_{N \in \N} \Big\|  \sup_{\substack{ \|u \|_{C_\g X^{s,\frac12 +\eps}} \le 1 \\ \|v \|_{C_\g X^{s,\frac12 +\eps}} \le 1} }  \big\| \I_\g ( \ind_{[0,1]}(t) \NN (\Psi_{\g,N}, u, v ))  \big\|_{C_ \g X^{s, \frac12 + \ta}} \Big\|_{L^p(\O)}  < \infty,
\label{bi1}
\end{align}

\noi
By \eqref{nonlin3} and \eqref{nonlin4}, it is suffices to prove the following frequency localized

\noi
\begin{align}
 \Big\|  \sup_{\substack{ \|u \|_{C_\g X^{s,\frac12 +\eps}} \le 1 \\ \|v \|_{C_\g X^{s,\frac12 +\eps}} \le 1} } \big\| \I_\g \big( \ind_{[0,1]}(t) \NN \big( \P_{N_1} \Psi_{\g,N}, \P_{N_2} u, \P_{N_3} v \big) \big)  \big\|_{X^{s, \frac12 + \ta}} \Big\|_{L^p(\O)} \les p N_{1}^{-\dl_0},
\label{bi2}
\end{align}

\noi
for some small $\dl_0>0$, uniformly in $N \in \N$ and $\g \in [0,1]$, and for dyadic numbers $N_1, N_2, N_3$, with $N_1 \gg N_2, N_3$.

Next, fix $0 < \eta \ll 1$. Let $N_1, N_2, N_3 $ be dyadic numbers, and write $N_{\star} = (N_1, N_2, N_3)$. If we have $N_2 \wedge N_3 \ges N_1^{\frac12 + \eta}$. Then, by using Lemma \ref{LEM:duha1} and Lemma \ref{LEM:S2}, we obtain \eqref{bi2} by observing the inequality $N_1^s \les (N_2N_3)^{s - 10\eta}$. (Hence, $u$ and $v$ can absorb the derivatives coming from $\Psi_{\g,N}$.)

Thus, in the following, we assume the condition $N_2 \wedge N_3 \ll N_1^{\frac12 + \eta}$. By using Lemma \ref{LEM:duha1} and arguing as in the proof of Lemma \ref{LEM:bilin1}, we have 

\noi
\begin{align}
\begin{split}
&  \Big\|  \sup_{\substack{ \|u \|_{C_\g X^{s,\frac12 +\eps}} \le 1 \\ \|v \|_{C_\g X^{s,\frac12 +\eps}} \le 1} } \big\| \I_\g \big( \ind_{[0,1]}(t) \NN \big( \P_{N_1} \Psi_{\g,N}, \P_{N_2} u, \P_{N_3} v \big) \big)  \big\|_{X^{s, \frac12 + \ta}} \Big\|_{L^p(\O)} \\
&  \qquad \les \sup_{\ld, \mu_2, \mu_3 }  p N_1^\ta \| \mathfrak{h}^{\ominus, >, N_{\star}} \|_{2},
\end{split}
\label{bi10}
\end{align}

\noi
for any $\ta >0$, and uniformly in $\ld, \mu_2, \mu_3 \in \R$, where $\mathfrak{h}^{\ominus, >, N_{\star}} = \mathfrak{h}^{\ominus, >, N_{\star}}_{n n_1 n_2 n_3}$ (which also depends on $\ld, \mu_2, \mu_3$) denotes the tensor

\noi
\begin{align}
\mathfrak{h}^{\ominus, >, N_{\star}}_{n n_1 n_2 n_3} = \ind_{\substack{n = n_1 - n_2 + n_3 \\ n_2 \neq n_1, n_3 }} \ind_{\jb{n_1} \le N} \prod_{j=1}^3 \ind_{\jb{n_j} \sim N_j} \cdot \frac{ \jb n ^s }{ \jb{n_2}^s \jb{n_3}^s } \cdot \frac{  \sqrt{2 \g}    }{ \jb{ \Phi_2( \bar n)}}.
\label{bi11}
\end{align}
\noi

\noi
with $\| \cdot \|_2$ as in \eqref{tnorm2} and $\Phi_2 (\bar n)$ as in \eqref{bi106}. From \eqref{bi11}, and Lemma \ref{LEM:t2} (i) and (ii) (fixing the values of $\kk (\bar n)$ as in Lemma \ref{LEM:RMT2}, we obtain the bounds

\noi
\begin{align}
\|  \mathfrak{h}^{\ominus, >, N_{\star}} \|_2 & \les  \g^{-\frac12} (N_2 N_3)^{-s} N_1^{s-1} (N_2 \wedge N_3),    \label{bi12}\\
\|  \mathfrak{h}^{\ominus, >, N_{\star}} \|_2 & \les   \g^{\frac12} (N_2 N_3)^{-s} N_1^{s+\frac12+ \ta} (N_2 \wedge N_3)^{\frac12}, \label{bi13}
\end{align}

\noi
for any $\ta >0$. Unfortunately, in the regime $N_2 \wedge N_3 = c N_1^{\frac12 + \eta}$ for some $0 < c \ll 1$, interpolating the \eqref{bi12} and \eqref{bi13} does not suffice to obtain acceptable bounds on \eqref{bi10}.

Instead, we rely on the derivative gain from the Duhamel operator $\I_\g$ for $\g >0$ in Lemma \ref{LEM:duha1}. Let us fix $T>0$, $\g >0$ and $u,v$ with $\|u \|_{C_\g X^{s,\frac12 +\eps}} \le 1$ and $\|v \|_{C_\g X^{s,\frac12 +\eps}} \le 1$. We have with  \eqref{D01} in Lemma \ref{PROP:duha1}, the following estimate:

\noi
\begin{align}
& \big\| \I_\g \big( \ind_{[0,1]}(t) \NN \big( \P_{N_1} \Psi_{\g,N}, \P_{N_2} u, \P_{N_3} v \big) \big)  \big\|_{C_\g X^{s, \frac12 + \ta}} \notag \\
& \qquad \qquad \les \g^{-\frac12 + \dl} \big\| \NN \big( \P_{N_1} \Psi_{\g,N}, \P_{N_2} u, \P_{N_3} v \big)  \big\|_{C_\g L_t^2 ([0,1]) H_x^{s-1 + 2 \dl}} \notag \\
& \qquad \qquad \les \g^{-\frac12 + \dl} N_1^{s-1 + 2\dl} \big\| \NN \big( \P_{N_1} \Psi_{\g,N}, \P_{N_2} u, \P_{N_3} v \big)  \big\|_{C_\g L_t^2 ([0,1]) L_x^2} \label{bib1},
\end{align}

\noi
where we used $N_1 \gg N_2, N_3$. Furthermore, we have

\noi
\begin{align}
& \F_x \big( \NN \big( \P_{N_1} \Psi_{\g,N}, \P_{N_2} u, \P_{N_3} v \big)  \big)(n,t)  \notag \\
& \qquad \qquad = \sum_{\substack{n = n_1 - n_2 + n_3 \\ n_2 \neq n_1, n_3 }} \ft{\P_{N_1} \Psi_{\g,N}}(n_1,t) \cj{ \ft{  \P_{N_2} u}}(n_2,t) \ft{ \P_{N_3} v }(n_3,t) \notag \\
& \qquad \qquad =  \sum_{n = n_1 - n_2 + n_3} \ft{\P_{N_1} \Psi_{\g,N}}(n_1,t) \cj{ \ft{  \P_{N_2} u}}(n_2,t) \ft{ \P_{N_3} v }(n_3,t) \notag \\
& \qquad \qquad \quad - \ft{\P_{N_1} \Psi_{\g,N}}(n,t) \sum_{n_2 \in \Z^2} \cj{ \ft{  \P_{N_2} u}}(n_2,t) \ft{ \P_{N_3} v }(n_2,t) \notag \\
& \qquad \qquad  = \1 - \II. \label{bib3}
\end{align}

\noi
Note that in the last computation, we have removed the condition $n_2 \neq n_1$ under the assumption $N_1 \gg N_1, N_2$. Given the fact that $\1 = \F_x \big(   \P_{N_1} \Psi_{\g,N} \cj{ \P_{N_2} u} \P_{N_3} v \big)(n,t)$, we estimate using H\"older and Cauchy-Schwarz inequalities along with Lemma \ref{LEM:S1},

\noi
\begin{align}
\| \1 \|_{C_\g L_t^2 ([0,1]) L_x^2} & \les \|  \P_{N_1} \Psi_{\g,N} \|_{C_{\g,t} ([0,1]^2) L_x^{\infty}} \big\| \cj{ \P_{N_2} u} \P_{N_3} v \big\|_{C_\g L_t^2 ([0,1]) L_x^2} \notag \\
& \les \|  \P_{N_1} \Psi_{\g,N} \|_{C_{\g,t} ([0,1]^2) L_x^{\infty}} \|  \P_{N_2} u \|_{C_\g L^4 ( [0,1] \times \T^2 )}  \|\P_{N_3} v \big\|_{C_\g L^4 ( [0,1] \times \T^2 )}  \notag\\
& \les N_1^{\ta} \| \P_{N_1} \Psi_{\g,N} \|_{C_{\g,t} ([0,1]^2) W_x^{-\ta, \infty}} \|  \P_{N_2} u \|_{C_\g X^{\ta,\frac12+\eps}}  \|\P_{N_3} v \big\|_{C_\g X^{\ta, \frac12 + \eps}} \notag \\
& \les N_1^{\ta} (N_2 N_3) ^{-s+ \ta} \|  \Psi_{\g,N} \|_{C_{\g,t} ([0,1]^2) W_x^{-\ta, \infty}},
\label{bib4}
\end{align}

\noi
for any $\ta >0$. Similarly, we have by Cauchy-Schwarz inequality,

\noi
\begin{align}
\| \II \|_{C_\g L_t^2 ([0,1]) L_x^2} & \les  \| \P_{N_1} \Psi_{\g,N}\|_{C_{\g,t} ([0,1]^2) L_x^2} \sup_{(\g, t) \in [0,1]^2} \Big| \sum_{n_2 \in \Z^2} \cj{ \ft{  \P_{N_2} u}}(n_2,t) \ft{ \P_{N_3} v }(n_2,t) \Big| \notag \\
& \les  N_1^{\ta} \| \Psi_{\g,N} \|_{C_{\g,t} ([0,1]^2) H_x^{-\ta}} \| \P_{N_2} u \|_{C_{\g,t} ([0,1]^2) L_x^2} \| \P_{N_3} v \|_{C_{\g,t} ([0,1]^2) L_x^2} \notag \\
& \les  N_1^{\ta} (N_2 N_3) ^{-s} \| \Psi_{\g,N} \|_{C_{\g,t} ([0,1]^2) H_x^{-\ta}} \| \P_{N_2} u \|_{C_\g X^{s,\frac12 + \eps}} \| \P_{N_3} v \|_{C_\g X^{s,\frac12 + \eps}} \notag \\
& \les  N_1^{\ta} (N_2 N_3) ^{-s} \| \Psi_{\g,N} \|_{C_{\g,t} ([0,1]^2) H_x^{-\ta}}, \label{bib5}
\end{align}

\noi
for any $\ta >0$. Note that we used the continuous embedding $C_T H^s \subset X^{s,b}$ for $b >\frac12$ in the above. Collecting \eqref{bib1}, \eqref{bib3}, \eqref{bib4}, and \eqref{bib5} gives (with $0 < T \le 1$)

\noi
\begin{align}
& \Big\|\sup_{\substack{ \|u \|_{C_\g X^{s,\frac12 +\eps}} \le 1 \\ \|v \|_{C_\g X^{s,\frac12 +\eps}} \le 1} }\big\| \I_\g \big( \ind_{[0,1]}(t) \NN \big( \P_{N_1} \Psi_{\g,N}, \P_{N_2} u, \P_{N_3} v \big) \big)  \big\|_{C_\g X^{s, \frac12 + \ta}} \Big\|_{L^p(\O)} \notag \\
& \qquad \qquad  \les  \g^{-\frac12 + \dl} N_1^{s-1 + 2\dl + \ta} (N_2 N_3) ^{-s+ \ta} \big\| \|  \Psi_{\g,N} \|_{C_{\g,t} ([0,1]^2) W_x^{-\ta, \infty}}   \big\|_{L^p(\O)} \notag \\
& \qquad \qquad  \les   \g^{-\frac12 + \dl} N_1^{s-1 + 2\dl + \ta} (N_2 N_3) ^{-s+ \ta} p.\label{bib6}
\end{align}

\noi
Hence, interpolating \eqref{bib6} and \eqref{bi10} with \eqref{bi13} gives \eqref{bi2} and concludes the proof.
\end{proof}

\noi
\begin{remark}\rm \label{RMK:bilin}
Let us note that \eqref{bib6} improves on \eqref{bi12} since it saves one derivative in the smallest frequency $N_2 \wedge N_3$.
\end{remark}

\section{Proof of Theorem \ref{THM:main}} \label{SEC:main}

In this section, we present the proof of Theorem \ref{THM:main}. Fix $N \in \N \cup \{ \infty \}$. With the notations of Subsection \ref{SUBSEC:outline}, we first recall the {\it enhanced integral equation}\footnote{Here, we mean that \eqref{veqf} is an equation in the variable $(\g, t,x)$.} for the nonlinear remainder $v_N = v_N( \g, t,x)$:

\noi
\begin{align}
\begin{split}
v_{\g,N}  & = - (\g + i) \P_{\le N} \Big(  \I_\g \NN(v_{\g,N})  + \<30>_{\g,N}  \\
& \quad + 2  \mathfrak{T}^1_{\g,N} \big( v_{\g,N}, v_{\g,N} \big) + \mathfrak{T}^2_{\g,N} \big( v_{\g,N}, v_{\g,N} \big) \\
& \quad + 2 \mathfrak{M}^1_{\g,N} ( v_{\g,N} ) + \mathfrak{M}^2_{\g,N}(v_{\g,N})  \Big) \\
& \quad + \I_\g \Rr \big(\<1>_{\g,N} + v_{\g,N} \big) \Big),
\end{split}
\label{veqf}
\end{align}

\noi
with $\NN$ and $\Rr$ as in \eqref{N3} and \eqref{N4}. Here, when $N = \infty$ in \eqref{veqf}, $\P_{\le N}$ is interpreted as the indentity operator $\operatorname{Id}$ and the stochastic objects in \eqref{veqf} are interpreted as their respective limits (given by the results in Subsection \ref{SEC:6}).

We now start state the main local well-posedness result which settles the local-in-time aspects of Theorem \ref{THM:main}. Its proof is a direct consequence of the bounds in Subsection \ref{SEC:6} and Banach fixed point theorem.

\noi
\begin{proposition}\label{PROP:lwp}
Let $0< s < \frac 14$, $0 <\ta, \eps \ll 1 $, and $T_0 >0$. Let us assume that 

\noi
\begin{itemize}

\item $\<1>_\g $ is a distribution-valued function belonging to $C\big( [0,1] \times [0, T_0]; W^{-\eps, \infty}(\T^2) \cap \FL^{1- \eps, \infty} (\T^2) \big) \bigcap C \big( [0,1]; X^{-\eps, \frac12 - \ta} ([0,T_0]) \big)$,

\smallskip
\item   
$\<30>_\g $ is a distribution-valued function belonging to $C\big( [0,1] ; X^{-\eps, \frac12+\ta} ([0,T_0]) \big)$, 

\smallskip
\item   
the operators $\mathfrak{M}^1_{\g}$ and $\mathfrak{M}^2_{\g}$ belong to the class $\L_{T_0}^{s,s,\frac12+ \eps, \frac12+ \ta}$ defined in \eqref{op1},

\smallskip
\item   
the operators $\mathfrak{T}^1_{\g}$ and $\mathfrak{T}^2_{\g}$ belong to the class $\B_{T_0}^{s,s,\frac12+ \eps, \frac12+ \ta}$ defined in \eqref{op1}.

\smallskip

\end{itemize}

Now, let $\Gamma$ be the map defined by

\noi
\begin{align}
\begin{split}
\Gamma (v) & = -(\g + i) \Big( \I_\g \NN(v)  + \I_\g \NN(\<1>_{\g}) \\
& \quad + 2 \,  \mathfrak{T}^1_{\g}(v,v) + \mathfrak{T}^2_{\g}(v,v) \\
& \quad + 2 \, \mathfrak{M}^1_{\g}(v) + \mathfrak{M}^2_{\g}(v) \\
& \quad + \I_\g \Rr \big(\<1>_{\g} + v \big) \Big),
\end{split}
\label{tnls}
\end{align}

\noi
with $\NN$ and $\Rr$ as in \eqref{N3} and \eqref{N4}. Then, the equation

\noi
\begin{align}
v = \Gamma (v),
\label{l1}
\end{align}

\noi
with initial data $\phi_0$ in $C \big( [0,1];  H^s( \T^2) \big)$, has a unique solution in $C\big( [0,1]; X^{s,\frac12+\eps}( [0,T] )\big)$ for some small almost surely positive time $T$ such that $T \sim \max \big( \| \Xi \|_{\mathcal{X}^{s, \eps, \ta}_{T_0}}, \| \phi_0 \|_{C_\g H^s_x} \big)^{-\eta}$ for some absolute constant $\eta >0$ and with $\| \cdot \|_{\mathcal{X}^{s, \eps, \ta}_{T_0}}$ as in \eqref{Xi_norm} below. 

Furthermore, the solution $v$ depends continuously on the enhanced data set 

\noi
\begin{align}
\Xi = \big( \<1>_\g , \<30>_\g , \mathfrak{M}^1_{\g} , \mathfrak{M}^2_{\g}, \mathfrak{T}^1_{\g} , \mathfrak{T}^2_{\g}  \big)
\label{data}
\end{align}

\noi
in the class

\begin{align}
\begin{split}
\mathcal{X}^{s, \eps, \ta}_{T_0} & =  C\big( [0,1] \times [0, T_0]; W^{-\eps, \infty}(\T^2) \cap \FL^{1- \eps, \infty} (\T^2) \big) \bigcap C \big( [0,1]; X^{-\eps, \frac12 - \ta} ([0,T_0]) \big) \\
& \qquad \times C\big( [0,1] ; X^{-\eps, \frac12+\ta} ([0,T_0]) \big) \times \big( \L_{T_0}^{s,s,\frac12+ \eps, \frac12+ \ta} \big)^2 \times \big( \B_{T_0}^{s,s,\frac12+ \eps, \frac12+ \ta} \big)^2
\label{Xi_norm}
\end{split}
\end{align}
\end{proposition}

We now prove Theorem \ref{THM:main}.

\noi
\begin{proof}[Proof of Theorem \ref{THM:main}] 

\noi
{\bf $\bul$ Step 1:} Fix $N \in \N \cup \{ \infty \}$ and consider the enhanced data set $\Xi_N =\big( \<1>_{\g,N} , \<30>_{\g,N} , \mathfrak{M}^1_{\g,N} , \mathfrak{M}^2_{\g,N}, \mathfrak{T}^1_{\g,N} , \mathfrak{T}^2_{\g,N}  \big) $ (when $N = \infty$, the stochastic objects are interpreted as their respective limits per Subsection \ref{SEC:6}). By Proposition \ref{PROP:lwp}, and Lemmas \ref{LEM:sto1} and \ref{LEM:sto_cubic}, Propositions \ref{PROP:RMT} and \ref{PROP:bilin}, there exists a unique solution to $v_{N}$ to \eqref{veqf} in $C\big( [0,1]; X^{s,\frac12+\eps}( [0,T] )\big)$ for some small (random) time $T>0$, $\wt \rho \otimes \Prob$-almost surely.\footnote{Strictly speaking, the proof of Proposition has to be adapted to take into account the frequency projection $\P_{\le N}$ in \eqref{veqf}, but we ignore these technicalities for convenience.}\footnote{Here, we can freely pass from the Gaussian data given by \eqref{GFF} to the renormalized Gibbs measure $\wt \rho$ \eqref{Rmes} in view of Proposition \ref{PROP:mes}.} Furthermore, by the continuity property of the solution with respect to the enhanced data set $\{ \Xi \}_{N \in \N}$ in Proposition \ref{PROP:lwp}, we also obtain the convergence of $v_N $ to $v$ in $C\big( [0,1]; X^{s,\frac12+\eps}( [0,T] )\big)$. 

Now, for each $N \in \N \cup \{ \infty \}$, write $u_{\g,N} := \<1>_{\g,N} + v_N (\g)$, with $\<1>_{\g,N}$ as in \eqref{sto2} (or \eqref{sto1} for $N = \infty$). Then, for each finite $N$, $u_{\g,N}$ solves \eqref{rSCGL1} on $[0,T]$ and converges, in view of Lemma \ref{LEM:sto1}, to the process $u_\g := \<1>_\g + v (\g)$ as $N \to \infty$. This proves the local-in-time part of Theorem \ref{THM:main}.

\smallskip

\noi
{\bf $\bul$ Step 2:} Now, let us fix $\g \in [0,1]$. In view of Bourgain's invariant measure argument \cite{BO94, BO96}, we can extend both $u_{\g,N}$, for each $N \in \N$, and $u_\g$ globally in time. Furthermore, the law of $u_\g$ is given by $\wt \rho$. Together with Step 1, this finishes the proof of Theorem \ref{THM:main} (i).

\medskip

\noi
{\bf $\bul$ Step 3:} Let us assume that the solution $v$ to \eqref{veqf} for $N = \infty$ has been constructed on $[0,T_0]$ for some $T_0 >0$. Then, we have $v \in C \big( [0,1 ] \times [0,T_{0}]; H^s (\T^2) \big)$, almost surely. Fix any arbitrary (large) $T>T_0$. We aim to prove that we can extend $v$ on the larger time interval $[0,T]$.  Fix $\g \in [0,1]$. Again, by a modification of Bourgain's invariant argument (see for instance \cite{ORTz}), there exists a set of full $\wt \rho \otimes \Prob$-probability $\O_\g$ such that $u_\g$ exists globally on $\O_\g$ we have

\noi
\begin{align}
\| v_\g \|_{C_T H^{s}} \le C (T) \qquad \text{on $\O_\g$}.
\label{gb}
\end{align}

Now, let $\O_g$ be the full $\wt \rho \otimes \Prob$-probability set given by

\noi
\begin{align}
\O_g = \bigcap_{\g \in \mathbb{Q} \cap [0,1]}.
\label{gset}
\end{align}

\noi
Then, by continuity in $\g$, we have that 

\noi
\begin{align}
\| v \|_{C_{\g,T_0} H^{s}} \le C (T) \qquad \text{on $\O_g$}.
\label{gb1}
\end{align}

\noi
Thus, by re-iterating the local well-posedness argument of Proposition \ref{PROP:lwp}, we can further extend $v$ to $[T_0,T]$. This proves Theorem \ref{THM:main} (ii).
\end{proof}

\appendix

\section{Multiple stochastic integrals}\label{SEC:B}

In this section, we go over the basic definitions
and properties of multiple stochastic integrals.
See \cite{Nua} and also \cite[Section 4]{Bring}
for further discussion.

Let $\ld$ be the measure on $Z: = \Z^2 \times \R_+ \times \{-1,1\}$ defined by 
\[ d \ld=   d n d t d \zeta, \]

\noi
where $d n$ and $d \ze$ are the counting measures on $\Z^2$ and $\{-1,1\}$. 
Given $ k \in \N$, 
we set $\ld_k = \bigotimes_{j =1}^k \ld$
and $L^2(Z^k) = L^2 \big( (\Z^2 \times \R_+ \times \{-1,1\})^k,  \ld _k \big)$. 

Let $\{ B_n ^{-1} \}_{n \in \Z^2}$ be an i.i.d. family of standard complex Brownian motions, independent from $\{ B_n  \}_{n \in \Z^2}$ in \eqref{W1} and which verifies 

\noi
\begin{align}
g_n = \int_0^1 d B_n^{-1}(s) 
\label{BM}
\end{align}

\noi
for any $n \in \Z^2$ with $\{g_n\}_{n \in \Z^2}$ as in \eqref{GFF}. In what follows, we write $\{ B_n ^{1} \}_{n \in \Z^2}$ for $\{ B_n  \}_{n \in \Z^2}$.

Given a function $f \in L^2(Z^k)$, 
we can  adapt 
the discussion in \cite[Section 1.1]{Nua}
(in particular,   \cite[Example 1.1.2]{Nua})
to the complex-valued setting and 
define the multiple stochastic integral $I_k[f]$ by 
\begin{align*}
I_k [f] = \sum_{n_1,\dots,n_k \in \Z^2} \int_{[0,\infty)^k} \sum_{\ze_1, \cdots, \ze_k \in \{ -1,1 \}}
f( n_1, t_1, \dots,  n_k, t_k) 
d B^{\ze_1}_{n_1} (t_1) \dots d B^{\ze_2}_{n_k} (t_k).
\end{align*}

\noi
Given a function $f \in L^2(Z^k)$, we define its symmetrization $\Sym(f)$ by
\begin{align}
\Sym(f)(z_1, \dots, z_k) = \frac{1}{k!} \sum_{\s \in S_k} f( z_{\s(1)},  \dots,z_{\s(k)}  ),
\label{sym}
\end{align}		
\noi
where 
$z_j = (n_j, t_j, \ze_j)$ and $S_k$ denotes the symmetric group on $\{1, \dots, k\}$.
Note that by Jensen's equality, we have 
\begin{align}
|\Sym(f)(z_1, \dots, z_k)|^p \le \frac{1}{k!} \sum_{\s \in S_k} |f( z_{\s(1)},  \dots,z_{\s(k)}  )|^p
\label{Jen}
\end{align}

\noi
for any $p \geq 1$.
We say that $f$ is symmetric if $\Sym(f) = f$.
We now recall some basic properties of  multiple stochastic integrals.

\noi
\begin{lemma}\label{LEM:B1}
Let $k, \l \in \N$. 
The following statements hold for any $f \in L^2(Z^k)$ and $g \in L^2(Z^\l)$\textup{:}
\begin{itemize}
\item[\textup{(i)}] $I_k : L^2(Z^k) \to \H_k \subset L^2(\O)$ is a linear operator, 
where $\H_k$ denotes the $k$th Wiener chaos defined in Subsection \ref{SEC:2_3}.\footnote{Strictly speaking, $\H_k$ is a space of real-valued random variables, but we mean here that the real and imaginary parts of $I_k$ have $\H_k$ as a target space.}

\smallskip

\item[\textup{(ii)}]
 $I_k[ \mathtt{Sym}(f)] = I_k[ f] $.

\smallskip
\item[\textup{(iii)}] Ito isometry\textup{:}
\[ \E \big[ I_k[f] \overline{ I_{\l}[g] } \, \big] = \ind_{k=\ell}\cdot  k! \int_{(\Z^2 \times \R \times \{-1,1\})^k} \mathtt{Sym}(f) \overline{ \Sym(g) }  d\ld_k.\]

\smallskip

\item[\textup{(iv)}]
Furthermore, suppose that $f$ is symmetric.
Then, we have 
\begin{align*}
I_k[f] = k! \sum_{n_1, \cdots, n_k \in \Z^2} \sum_{\ze_1, \dots, \ze_k \in \{ -1,1 \}} 
\int_{0}^\infty \int_{0}^{t_1} \int_{0}^{t_{k-1}} f(n_1,t_1, \ze_1, \dots, n_k,t_k, \ze_k)
 dB^{\ze_k}_{n_k}(t_k) \cdots dB^{\ze_1}_{n_1}(t_1),
\end{align*}

\noi
where the iterated integral on the right-hand side is understood as an iterated 
Ito integral.

\end{itemize}

\end{lemma}

We  state a version of Fubini's theorem for multiple stochastic integrals
that is convenient for our purpose.
See, for example, 
\cite[Theorem 4.33]{DPZ14}
for a version of the stochastic Fubini theorem.

\noi
\begin{lemma}\label{LEM:B3}
Let  $k \geq 1$. 
Given finite $T> 0$, let $f \in L^2 ( (\Z^2 \times [0, T] \times \{-1,1\})^k \times [0, T], 
d \ld _k \otimes  dt \big)$.  
\textup{(}In particular, we assume that the temporal support \textup{(}for the variables $t_1, \dots, t_k, t$\textup{)} of 
$f$ is contained in $[0, T]^{k+1}$
for any $(n_1, \ze_1, \dots, n_k, \ze_k)$.\textup{)}
Then,  we have
\begin{align}
\int_{0}^T I_k [f(\cdot,t)] d t = I_k \bigg[ \int_{0}^T f(\cdot,t) d t \bigg]
\label{fub0}
\end{align}

\noi
in $L^2(\O)$.

\end{lemma}

We conclude this section by stating the product formula (Lemma \ref{LEM:prod}).
Before doing so, we first recall the contraction of two functions.

\begin{definition} \label{DEF:B4}\rm
Let $k, \ell \in \N$. 
Given an integer $0 \leq r \leq \min (k,l)$, 
 we define the contraction
$ f \otimes_r g $
  of $r$ indices of $f\in L^2(Z^k)$ and $g\in L^2(Z^\l)$ by 
  
\noi
\begin{align*}
 (f \otimes_r g)(z_1, \dots, z_{k + \ell - 2 r})  
 & =   \sum_{m_1, \dots, m_r \in \Z^2} \int_{\R_+^r} \sum_{\iota_1, \dots, \iota_j \in \{-1,1\} }
f(z_1, \dots, z_{k-r}, \al_1, \dots, \al_r)\\
& \hphantom{XXXXX} \times 
g(z_{k+1-r}, \dots, z_{k+\ell -2r},
\wt \al_1, \dots, \wt \al_r ) ds_1 \cdots ds_r, 
\end{align*}

\noi
where
$\al_j = (m_j, s_j, \iota_j)$
and $\wt \al_j = (m_j, s_j, \iota_j)$. For convenience, we may simply write $\otimes$ instead of $\otimes_0$.

\end{definition}

Note that even if $f$ and $g$ are symmetric, 
their contraction $f\otimes_r \cj{g}$ is not symmetric in general.
We now state the product formula.
See \cite[Proposition 1.1.3]{Nua}.

\noi
\begin{lemma}[product formula]\label{LEM:prod}
Let  $k, \ell \in \N$ with $k \ge \l$. 
Let  $f \in L^2(Z^k)$ and $g \in L^2(Z^\l)$
be symmetric functions. Then,  we have 

\noi
\begin{align*}
& I_k[f] \cdot \cj{ I_{\ell}[g]} = \sum_{r=0}^{\l} r! \binom{k}{r} \binom{\ell}{r} I_{k + \ell - 2r}[f \otimes_r \cj{g}].
\end{align*}

\noi
and

\noi
\begin{align*}
& I_k[f] \cdot I_{\ell}[g] = I_{k + \l}[f \otimes g].
\end{align*} 
\end{lemma}

Let $f_{t} = f_{t}(n', t', \ze)$ and $g_{n,t} = g_{n,t}(z_1, z_2, z_3)$ (for $z_j = (n_j, t_j, \ze_j)$, $1 \le j \le 3$) be the functions defined by

\noi
\begin{align}
\begin{split}
f_{t}(n' ,t',\ze) &  = \ind_{\R_+}(t) \Big( \ind_{\ze = -1} \ind_{[0,1]} (t')  \frac{ e^{- ( \g  + i  ) t \jb{n'} ^2 }  }{ \jb{ n' } } \\
& \qquad \quad + \sqrt{2 \g} \ind_{\ze = 1}  \ind_{[0,t]} (t') e^{- ( \g  + i  ) (t-t') \jb{n'} ^2 } \Big)
\end{split}
\label{f}
\end{align}

\noi
and 

\noi
\begin{align*}
g_{n,t}(z_1, z_2, z_3) =  \ind_{\substack{ n = n_1 - n_2 + n_3 \\ n_2 \neq n_1, n_3}} \,  f_t \otimes \cj{f_t} \otimes_0 f_t (z_1, z_2, z_3),
\end{align*}

\noi
Then, from Lemma \ref{LEM:prod}, \eqref{conv1}, \eqref{sto1}, and \eqref{sto_cubic}, we have the formulae

\noi
\begin{align}
\begin{split}
\F_x ( \<1>_\g )(n,t) & = I_1 \big[ \ind_{n=n'} f_{t} \big] \\
\F_x ( \<3>_\g )(n,t) & = I_3 \big[ g_{n,t} \big]
\end{split}
\label{ob1}
\end{align}

\noi
\section{Random tensors}\label{SEC:C}

In this section, we provide the basic definition and 
some lemmas on (random) tensors from~\cite{DNY2, Bring}.
See~\cite[Sections 2 and 4]{DNY2} and \cite[Section 4]{Bring}
for further discussion.

\begin{definition} \label{def_tensor} \rm
Let $A$ be a finite index set. We denote by $n_A$ the tuple $ (n_j : j \in A)$. 
 A tensor $h = h_{n_A}$ is a function: $(\Z^2)^{A} \to \mathbb{C} $ with the input variables $n_A$. Note that the tensor $h$ may also depend on $\o \in \O$. 
 The support of a tensor $h$ is the set of $n_A$ such that $h_{n_A} \neq 0$. 

Given a finite index set  $A$, 
let $(B, C)$ be a partition of $A$. We define the norms 
 $\| \cdot \|_{n_A}$ and 
$\| \cdot \|_{n_{B} \to n_{C}}$ by 
\[ \| h \|_{n_A}  = \|h\|_{\l^2_{n_A}} = \bigg(\sum_{n_A} |h_{n_A}|^2\bigg)^\frac{1}{2}\]
and
\begin{align}
  \| h \|^2_{n_{B} \to n_{C}} = \sup \bigg\{ 
\sum_{n_{C}} \Big| \sum_{n_{B}} h_{n_A} f_{n_{B}} \Big|^2 :  \| f \|_{\l^2_{n_{B}}} =1  \bigg\},  
\label{Z0a}
\end{align}

\noi
where  we used the short-hand notation $\sum_{n_Z}$ for $\sum_{n_Z \in (\Z^2)^Z}$ for a finite index set $Z$.
Note that, by duality, we have  $\| h \|_{n_{B} \to n_{C}} = \| h \|_{n_{C} \to n_{B}} 
= \| \cj h \|_{n_{B} \to n_{C}}$ for any tensor $h = h_{n_A}$. 
If $B = \varnothing$ or $C = \varnothing$,  then we have
$  \| h \|_{n_{B} \to n_{C}} = \| h \|_{n_A}$.
\end{definition}

For example, when $A = \{1, 2\}$, 
the norm  $\| h \|_{n_{1} \to n_{2}}$ denotes the usual operator norm
$\| h \|_{\l^2_{n_{1}} \to \l^2_{n_{2}}}$
for an infinite dimensional matrix operator $\{h_{n_1 n_2}\}_{n_1, n_2 \in \Z^3}$.
By bounding the matrix operator norm by the Hilbert-Schmidt norm (= the Frobenius norm), we have
\begin{align}
\| h \|_{\l^2_{n_{1}} \to \l^2_{n_{2}}} \le \| h\|_{\l^2_{n_1, n_2}}
\label{Z0}
\end{align}

Let $(B, C)$ be a partition of $A$.
Then, 
by duality, we can write \eqref{Z0a} as 
\begin{align*}
  \| h \|_{n_{B} \to n_{C}} = \sup \bigg\{ 
\sum_{n_{C}} \Big| \sum_{n_{B}, n_C} h_{n_A} f_{n_{B}} g_{n_C}\Big| : 
\| f \|_{\l^2_{n_{B}}} =  \| g \|_{\l^2_{n_{C}}} =1   \bigg\}, 
\end{align*}

\noi
from which we obtain
\begin{align}
\sup_{n_A}|h_{n_A}|
= \sup_{n_B, n_C}|h_{n_Bn_C}|
\le   \| h \|_{n_{B} \to n_{C}}.
\label{Z0b}
\end{align}

%
%

\medskip

Next, we recall a key deterministic tensor bound in the study of the random cubic NLW from \cite{Bring}.

We conclude this section with the following random matrix estimate.
This lemma is essentially the content of Propositions 2.8 and 4.14 in \cite{DNY2};
see also Proposition 4.50 in \cite{Bring}.

Let $A$ be a finite index set. We set $z_A = (k_A,t_A, \ze_A)$ for $(k_A, t_A, \ze_A) \in  (\Z^2)^A\times  \R^A \times \{-1,1\}^A$
and write $f_{z_A} = f(z_A) = f(n_A, t_A, \ze_A)$.

\begin{lemma}\label{LEM:DNY}
Let $A$ be a finite index set with $k = |A| \geq 1$. Let $h = h_{bcn_A}$ be a tensor such that $n_j \in \Z^2$ for each $j \in A$ and $(b,c) \in (\Z^2)^d$ for some integer $d \geq 2$. 
Given  $N \geq 1$, 
 assume that 
\begin{align}\label{cond_DNY}
\supp h \subset \big\{ |b - b_\star|, |c - c_\star|, |n_j - n_{j,\star}| \les N
\text{ for each $j \in A$}
 \big\},
\end{align} 

\noi
for some $(a_\star, b_\star, (n_{j,\star})_{j \in A}) \in (\Z^2) ^2 \times (\Z^2)^k$. Given a \textup{(}deterministic\textup{)} tensor $h_{bcn_A} \in \l^2_{bcn_A}$, define the tensor $H = H_{bc}$ by
\begin{align}
H_{bc} =   I_k \big[ h_{bcn_A} f_{z_A} \big]
\label{Z1}
\end{align}

\noi
for   $f \in \l^{\infty}_{n_A} \big((\Z^2)^A; L^2_{t_A}(\R_+^A ) \times \l^2_{\ze_A}( \{-1, 1\} ) \big)$, 
where $I_k$ denotes the multiple stochastic integral 
defined in Appendix~\ref{SEC:B}.
Then,  for any $\theta > 0$,  we have
\begin{align}
\big\| \| H_{bc} \|_{b \to c} \big\|_{L^p(\O)} 
\les p^{\frac k2} N^{\theta}\Big( \max_{(B,C)} \big\| h \| f(n_A, t_A)\|_{ L^2_{t_A} \l^2_{\ze_A}} \big\|_{b n_B \to c n_C}\Big)
 , 
\label{Z1a}
\end{align}

\noi
where the maximum is taken over  all partitions $(B,C)$ of $A$.
\end{lemma}

\begin{proof} The proof is a (very) slight modification of the proof of \cite[Lemma C.3]{OWZ}.
\end{proof}
\begin{ackno}\rm
The author would like to thank his advisor, Tadahiro Oh, for suggesting this problem and his support throughout its completion. He is also grateful to Leonardo Tolomeo and Guangqu Zheng for helpful discussions. The author was supported by the European Research Council (grant no.~864138 ``SingStochDispDyn''). 
\end{ackno}

\end{document}